\documentclass[10pt]{article}
\usepackage[top=0.9in,bottom=1.1in,left=1.05in,right=1.05in]{geometry}

\usepackage{lipsum}
\usepackage{amsfonts}
\usepackage{amsthm}
\usepackage{graphicx}
\usepackage{epstopdf}
\usepackage{algorithm}
\usepackage{algorithmic}
\usepackage[ruled,algo2e]{algorithm2e}
\usepackage{mathtools}
\usepackage{natbib}
\usepackage{booktabs,makecell}
\usepackage{comment}
\usepackage{subcaption}
\usepackage{float}
\usepackage{xcolor}   
\usepackage[hidelinks]{hyperref}
\newcommand*{\email}[1]{\href{mailto:#1}{\nolinkurl{#1}} } 

\ifpdf
  \DeclareGraphicsExtensions{.eps,.pdf,.png,.jpg}
\else
  \DeclareGraphicsExtensions{.eps}
\fi

\renewcommand{\paragraph}[1]{\textbf{#1.}}

\theoremstyle{plain}
\newtheorem{theorem}{Theorem}[section]

\title{Deep Picard Iteration for High-Dimensional Nonlinear PDEs}

\author{Jiequn Han\thanks{Authors contributed equally and are listed alphabetically. Center for Computational Mathematics, Flatiron Institute (\email{jiequnhan@gmail.com}).} 
\and Wei Hu\thanks{Institute for Advanced Algorithms Research, Shanghai (\email{weihu.math@gmail.com}).}
\and Jihao Long\thanks{Institute for Advanced Algorithms Research, Shanghai (\email{longjh1998@gmail.com})}.
\and Yue Zhao\thanks{Center for Data Science, Peking University (\email{yuezhao.math@gmail.com}).}}

\usepackage{amsopn}
\DeclareMathOperator{\diag}{diag}

\newcommand{\RR}{\mathbb{R}}
\newcommand{\BS}{\mathbb{S}}
\newcommand{\tr}{\mathrm{tr}}
\newcommand{\rmd}{\,\mathrm{d}}
\newcommand{\II}{\mathrm{I}}
\newcommand{\EE}{\mathbb{E}}
\newcommand{\transpose}{^\mathrm{T}}

\newtheorem{assumption}{Assumption}

\begin{document}

\maketitle

\begin{abstract}
We present the Deep Picard Iteration (DPI) method, a new deep learning approach for solving high-dimensional partial differential equations (PDEs). The core innovation of DPI lies in its use of Picard iteration to reformulate the typically complex training objectives of neural network-based PDE solutions into much simpler, standard regression tasks based on function values and gradients. This design not only greatly simplifies the optimization process but also offers the potential for further scalability through parallel data generation. Crucially, to fully realize the benefits of regressing on both function values and gradients in the DPI method, we address the issue of infinite variance in the estimators of gradients by incorporating a control variate, supported by our theoretical analysis. Our experiments on problems up to 100 dimensions demonstrate that DPI consistently outperforms existing state-of-the-art methods, with greater robustness to hyperparameters, particularly in challenging scenarios with long time horizons and strong nonlinearity. The code is available at \textcolor{blue}{\href{https://github.com/DeepOptimalControl/DeepPicardIteration}{https://github.com/DeepOptimalControl/DeepPicardIteration}}.
\end{abstract}


\section{Introduction}
This paper aims to solve high-dimensional nonlinear partial differential equations (PDEs) of the parabolic form:
\begin{equation}\label{eq:pde}
    \begin{dcases}
        \partial_t u+ F(t,x ,u,  \nabla_x u, \nabla_x^2u) = 0, \,\textrm{on } [0,T)\times\RR^d,\\
        u(T,x ) = g(x), \,\textrm{on } \RR^d,
    \end{dcases}
\end{equation}
where the dimension $d \in \mathbb{N}^+$, time horizon $T > 0$, the nonlinearity $F:[0,T]\times \RR^d \times \RR \times \RR^d \times \BS^d \rightarrow \RR$ ($\BS^d$ is the set of symmetric $d\times d$ matrices) and the terminal condition $g:\RR^d \rightarrow \RR$. 
We assume the PDE is well-posed; see, e.g., \cite{lunardi2012analytic} for the well-established results on the well-posedness of such PDEs. These high-dimensional PDEs have wide applications across various disciplines, including optimal control, portfolio optimization, economics, and probabilistic modeling, among others (see, e.g., \cite{yong2012stochastic, achdou2022income}), and thus require efficient numerical algorithms. However, their numerical treatment presents formidable challenges, as classical mesh-based methods are severely constrained by the notorious curse of dimensionality.

In response to the curse of dimensionality, \cite{han2016deep} introduced the first deep learning-based algorithm for high-dimensional scientific computing problems, with a focus on stochastic control problems, which are closely related to Hamilton-Jacobi-Bellman PDEs. Shortly after, for the general case of semilinear PDEs where $F$ in~\eqref{eq:pde} is linear in $\nabla_x^2u$, the works \cite{han2017deep, han2018solving} pioneered the Deep BSDE method, marking a revolutionary use of modern machine learning methodologies to solve high-dimensional PDEs. 
This approach reformulates semilinear PDEs as backward stochastic differential equations (BSDEs)~\citep{pardoux1992backward} and solves a variational problem by discretizing BSDEs in time and approximating the solution using deep neural networks. Since its introduction, the Deep BSDE method and related methods (e.g., \citep{chan2019machine, han2020convergence, hure2020deep, han2020deep, ruthotto2020machine,zhou2021actor, beck2021deep, nusken2021interpolating, germain2022approximation,zhang2022fbsde,chassagneux2023learning}) have significantly expanded the potential for solving high-dimensional PDEs. However, these methods still face performance limitations when dealing with challenging problems characterized by strong nonlinearity, leading to the high nonconvexity inherent in the optimization problems these algorithms solve. Similar issues in optimization also affect other deep-learning-based methods for PDEs, such as the Deep Galerkin method \citep{sirignano2018dgm} and the physics-informed neural networks (PINN) method \citep{raissi2019physics}, both of which directly use the squared residuals of the PDEs as the loss function. 

Fully nonlinear PDEs present even greater challenges compared to semilinear PDEs due to the additional nonlinearity in the second-order terms, and there is notably less literature available on solving high-dimensional fully nonlinear PDEs. Some noteworthy approaches to tackle such PDEs include: (1) physics-informed neural network (PINN) method, which can be directly applied to fully nonlinear PDEs but suffer from similar optimization challenges; (2) methods based on the second-order backward stochastic differential equations (2BSDEs) representation for fully nonlinear PDEs~\cite{cheridito2007second}, as explored by \cite{beck2019machine, pham2021neural};
and (3) the method proposed by \cite{nguwi2024deep}, which represents the solution to fully nonlinear PDEs through a branching process and uses Monte Carlo sampling to generate labels for training neural networks with a least-square loss. However, the variance of Monte Carlo sampling increases dramatically as the time horizon grows, limiting its applicability to problems with short time horizons.

Picard iteration is a fundamental and powerful method in both theoretical and numerical analysis of differential equations. It constructs a sequence of increasingly accurate approximations of solutions by substituting an initial guess into a fixed-point form of the original differential equations. Combined with multi-level Monte Carlo integration, \citep{hutzenthaler2019multilevel,hutzenthaler2020multilevel,hutzenthaler2021multilevel} demonstrate that the multi-level Picard iteration method can solve semilinear PDEs at specific points without the curse of dimensionality. However, in practice, rather than obtaining the solution at a single point, it is often more desirable to obtain the solution as a function across a domain of interest. \cite{chassagneux2023learning} attempts to combine the ideas of Picard iteration and linear-quadratic optimization to find such a solution for semilinear PDEs within a finite-dimensional linear space. However, its applicability to high-dimensional problems is heavily constrained by the representational limitations of the linear space, and the methodology does not generalize well to fully nonlinear problems.

In this study, we present a novel deep learning approach called the \textit{Deep Picard Iteration (DPI)} method, designed to fully realize the potential of Picard iteration when combined with the powerful approximation capabilities of deep neural networks. The DPI method is applicable to both semilinear and fully nonlinear PDEs, offering a robust solution for these problems. By leveraging Picard iteration, our method reframes the optimization challenges inherent in neural network approximation of PDE solutions to standard regression problems involving function values and gradients. This reformulation underpins the enhanced capability of our method to handle difficult PDEs more effectively than alternative approaches. The effectiveness of regression-based formulations over other objective functions in optimization has also been demonstrated in recent studies on solving optimal feedback control problems~\cite{nakamura2021adaptive,azmi2021optimal,zang2022machine,zhao2024offline}, further reinforcing the advantages of our approach. To obtain labels at each step of the Picard iteration, we utilize both the Feynman-Kac formula for function values and the Bismut-Elworthy-Li formula for gradients. Direct application of the Bismut-Elworthy-Li can lead to issues with infinite variance in the estimators. We provide a theoretical analysis of this problem and demonstrate that a simple control variate can reduce the variance to a finite level. Numerical experiments demonstrate that DPI outperforms existing state-of-the-art methods, showing superior results on both semilinear and fully nonlinear PDEs. Moreover, compared to other methods, DPI exhibits greater robustness to hyperparameters and strong capacity for parallelization, making it well-suited for solving large-scale problems.

This paper is organized as follows. Section 2 provides the background on the Feynman-Kac formula for linear PDEs. Section 3 introduces the concept of Deep Picard Iteration with gradient-augmented regression at an abstract level, including a rigorous analysis of the variance of the gradient estimator providing regression labels. Section 4 details the numerical algorithm while Section 5 presents the numerical results. Finally, Section 6 concludes the paper with a discussion of future work.

\section{Background}
In this section, we briefly review the classical Feynman-Kac formula for linear PDE
\begin{equation}\label{eq:pde_linear}
    \begin{dcases}
        \partial_t u+ \mu(t, x) \cdot \nabla_x u + \frac{1}{2} \text{tr}(\sigma \sigma^\top(t, x) \nabla_{x}^2 u) + f(t, x) = 0, \,\textrm{on } [0,T)\times\RR^d,\\
        u(T,x ) = g(x), \,\textrm{on } \RR^d,
    \end{dcases}
\end{equation}
where \(\mu: [0, T] \times \mathbb{R}^d \to \mathbb{R}^d\), \(\sigma = (\sigma_1,\dots,\sigma_d): [0, T] \times \mathbb{R}^d \to \mathbb{R}^{d \times d}\), and $ \text{tr}(\cdot)$ denotes the trace operator. This PDE can be viewed as a special case of \eqref{eq:pde} in which $F$ is linear with respect to $\nabla_x u$ and $\nabla_x^2 u$. 
The Feynman–Kac formula allows us to express $u(t,x)$ as a conditional expectation under a probability measure.
To be specific,
let \((\Omega, \mathbb{F}, \{\mathcal{F}_t\}_{0 \leq t \leq T}, \mathbb{P})\) be a filtered probability space equipped with a \(d\)-dimensional standard Brownian motion \(\{W_t = (W_t^1,\dots,W_t^d)\transpose\}_{0 \leq t \leq T}\) starting from \(0\). 
Given the probability space, we introduce the forward stochastic differential equations (SDEs):
\begin{equation}
\label{eq:forward_SDE_general}
    X_s^{t,x} = x + \int_{t}^s \mu(r, X_r^{t,x}) \rmd r + \int_{t}^s \sigma(r, X_r^{t,x}) \rmd W_r, \quad s \in [t,T],
\end{equation}
where \((t, x) \in [0, T] \times \mathbb{R}^d\).
Throughout the paper, we assume the following standard assumption regarding $\mu$ and $\sigma$ holds.
\begin{assumption}
\label{assump1}
\(\mu\) and \(\sigma\) are continuously differentiable in both \(t\) and \(x\). $\nabla_x \mu$ and $\{\nabla_x \sigma_j\}_{j=1}^d$ are bounded continuous functions in $[0,T]\times \RR^d \to \RR^{d\times d}$, $\mu(t,0)$ and $\sigma(t,0)$ are bounded for $t \in [0,T]$. Furthermore, $\sigma$ satisfies that\footnote{Throughout this work, we will use $|\cdot|$ to denote the Euclidean norm in $\RR^d$.}
\[
     m |y|^2\le y\transpose(\sigma\sigma\transpose)(t,x) y\le M |y|^2,\,\forall x,y \in \RR^d \text{ and }t \in [0,T] 
\]
for some constant $0 < m \le M < +\infty$. 
\end{assumption}
Then the Feynman-Kac formula~\cite{kac1949distributions,oksendal2013stochastic} states that
\begin{equation}
\label{eq:FeynmanKac_semi}
    u(t,x) = \EE[g(X_T^{t,x})] + \int_{t}^T \EE[f(s,X_s^{t,x})]\rmd s.
\end{equation}
This formula offers a probabilistic approach to evaluate the PDE solution at any given time-space point $(t,x)$. To achieve this, one can numerically simulate multiple paths of \eqref{eq:forward_SDE_general} and approximate the expectation in \eqref{eq:FeynmanKac_semi} using Monte Carlo integration.
Unlike classical mesh-based methods, this approach does not require spatial discretization. Moreover, the convergence rate of Monte Carlo integration scales inversely with the square root of the number of samples, making it independent of the spatial dimension. This is the key reason why this method can overcome the curse of dimensionality in high-dimensional PDEs; see \citep{hutzenthaler2019multilevel,hutzenthaler2020multilevel,hutzenthaler2021multilevel} for generalizations to semilinear PDEs. Additionally, if one seeks to obtain the solution across a time-space domain of interest rather than a single point, the Feynman-Kac formula provides an efficient way to generate solution labels at various time-space points, enabling a regression task using different function approximators such as sparse grids, kernel methods, or neural networks~\cite{beck2021solving}.

\section{Deep Picard Iteration with Gradient-Augmented Regression}
\label{sec:dpi}
In this paper we aim to extend the power of the above method from the linear PDE to the fully nonlinear case and obtain the solution in  function form. A natural idea is to conduct an iteration, viewing the fully nonlinear PDE as a linear PDE given the current estimate of $\nabla_x u$ and $\nabla_x^2 u$. To be more specific, let
\begin{equation}
\label{eq:rewritten_nonlinear_f}
f(t, x, y, z, \gamma) = F(t,x,y,z,\gamma)-\mu(t, x) \cdot z - \frac{1}{2} \tr(\sigma \sigma^\top(t, x) \gamma),
\end{equation}
and define
\begin{equation*}
    f_u(t,x) \coloneqq f(t,x,u(t,x),\nabla_x u(t,x),\nabla^2_x u(t,x)).
\end{equation*}
The PDE~\eqref{eq:pde} can thus be rewritten in the form of PDE~\eqref{eq:pde_linear}, with $f$ replaced by $f_u$. Consequently, if $u$ is a smooth solution of PDE~\eqref{eq:pde}, it follows that $u$ is also a solution of PDE~\eqref{eq:pde_linear}, with $f$ replaced by $f_u$. Therefore, the Feynman-Kac formula \eqref{eq:FeynmanKac_semi} gives:
\begin{equation}
\label{eq:FeynmanKac}
    u(t,x) = \EE[g(X_T^{t,x})] + \int_{t}^T \EE[f_u(s,X_s^{t,x})]\rmd s.
\end{equation}

We view it as a fixed-point equation for $u$ and define the corresponding \textit{Picard iteration equation}
\begin{equation}
\label{eq:Picard}
    u_{k+1}(t,x) = \EE[g(X_T^{t,x})] + \int_{t}^T \EE[f_{u_k}(s,X_s^{t,x})]\rmd s,
\end{equation}
starting from $u_0(t,x) \equiv 0$.

Note that for the linear PDE~\eqref{eq:pde_linear}, the drift function $\mu$ and diffusion function $\sigma$ in \eqref{eq:forward_SDE_general} are uniquely determined by the PDE itself. However, this is not the case for fully nonlinear PDEs, where different choices for $\mu$ and $\sigma$ are possible, and the function $f$ in \eqref{eq:rewritten_nonlinear_f} can be defined accordingly. Further discussion on selecting these functions will be provided in Section~\ref{sec:alg} after~\eqref{eq:forward_SDE_x0}. Theoretically, when the PDE is semilinear, it is natural to select $\mu$ and $\sigma$ such that $f$ in \eqref{eq:rewritten_nonlinear_f} does not depend on $\gamma$ denoting $\nabla_x^2u$. With this choice and assuming that $f$ is globally Lipschitz continuous, \cite[Theorem 1.1]{hutzenthaler2021speed} demonstrate that the convergence rate of Picard iteration~\eqref{eq:Picard} is at least exponentially fast, with the error decaying as $c^k/{\sqrt{k!}}$. However, in the case of fully nonlinear PDEs, determining the conditions on $\mu$, $\sigma$, and $f$ that ensure the convergence of Picard iterations remains an open question.
    
Even without a theoretical guarantee of convergence for fully nonlinear cases, the Picard iteration defined in \eqref{eq:Picard} still offers a natural starting point for approximating the PDE solution with neural networks through a series of simpler tasks. Given the current approximation to $u_k(t,x)$, we view the right-hand side of \eqref{eq:Picard} as a way to generate samples of $u_{k+1}(t,x)$ at specific $(t, x)$ and then create a dataset of such samples for learning $u_{k+1}(t,x)$ through least-squares regression. 
Note that in order to generate samples through the right-hand side of \eqref{eq:Picard}, we need to evaluate $f_{u_k}$, which involves both the gradient term $\nabla_x u_k$ and the Hessian term $\nabla_x^2 u_k$. We compute these terms via automatic differentiation.

In regression, it is widely observed that incorporating gradient of the target function as additional labels can improve learning results~\cite{laurent2019overview,adcock2019compressive,azmi2021optimal,nakamura2021adaptive,o2024derivative}. We seek to realize a similar benefit in our scheme. To this end, we recall the Bismut-Elworthy-Li formula~\citep{bismut1984atiyah,ELWORTHY1994252,DaPrato1997,ma2002representation}, which gives $\nabla_x u$ through another stochastic representation with the similar spirit to Feynman-Kac formula: 
\begin{equation}
    \begin{aligned}\label{eq:B-E-L_formula_original}
             \nabla_x u(t,x) =~&\EE[g(X_T^{t,x}) H_T^{t,x}] +\int_{t}^T \EE [f_{u}(s,X_s^{t,x}) H_s^{t,x}]\rmd s,
    \end{aligned}
\end{equation}
where 
\begin{equation}
H_s^{t,x} \coloneqq (s-t)^{-1}\int_{t}^s [\sigma(r,X_r^{t,x})^{-1} D_r^{t,x}]\transpose \rmd W_r
\end{equation}
is called the Malliavin weight, and $D_s^{t,x} \in \mathbb{R}^{d\times d}$ is called the variational process (or Jacobian process) with respect to the forward process~\eqref{eq:forward_SDE_general}
\begin{equation}\label{eq:variation_process}
    D_s^{t,x} \coloneqq \II_d + \int_t^s \nabla_x \mu(r,X_r^{t,x}) D_r^{t,x} \mathrm{d}r + \sum_{j=1}^d \int_t^s \nabla_x \sigma_j(r,X_r^{t,x}) D_r^{t,x} \mathrm{d}W_r^j.
\end{equation}
Here $\II_d\in \mathbb{R}^{d\times d}$ denotes the identity matrix. Both $H_s^{t,x}$ and $D_s^{t,x}$ are defined for $s > t$ to measure the sensitivity of the solution with respect to the initial condition.
Given this formula, a natural idea is to again utilize the Monte-Carlo method to approximate the expectation in~\eqref{eq:B-E-L_formula_original} to generate labels on the gradients. However, the direct application of this formula does not work numerically since the corresponding Monte Carlo estimator will suffer from the \textit{infinite variance}, as shown in the theorem below. Note that such infinite variance phenomenon has also been observed in other similar contexts related to Malliavin calculus~\cite{kawai2010computation,andersson2017unbiased,henry2017unbiased}.
\begin{theorem}
\label{thm:infinite_var}
    Assume Assumption~\ref{assump1} holds. Given a fixed $t \in [0,T)$ and $x \in \mathbb{R}^d$, assume that $g(x) \in C^1(\RR^d)$ with $g(x) \neq 0$, and $f(t,x) \in C^1([0,T] \times \RR^d)$ with $f(t,x) \neq 0$, where both functions have bounded first-order derivatives. We have
    \begin{align*}
    \lim_{s \rightarrow T^-}\EE|g(X_T^{t,x}) H_T^{t,x}|^2 = +\infty,\, 
    \int_{t}^T \EE |f(s,X_s^{t,x}) H_s^{t,x}|^2\rmd s = +\infty.
    \end{align*}
\end{theorem}

For clarity, we defer the proof of  Theorem~\ref{thm:infinite_var} until after we identify the finite part of the variance. To resolve this fundamental issue of infinite variance and facilitate the Monte-Carlo approximation to the gradient, the key observation is that we can use simple control variates from $g(x)$ and $f(t,x)$ to reduce the variance to a finite value, thanks to the martingale property of Brownian motion. Notably, we have
\begin{align*}
    \EE[g(X_T^{t,x})H_T^{t,x}]   = \EE[(g(X_T^{t,x})-g(x)) H_T^{t,x}],
\end{align*}
and
\begin{align*}
    \int_{t}^T \EE [f_{u}(s,X_s^{t,x})H_s^{t,x}]\rmd s  = \int_{t}^T \EE [(f_{u}(s,X_s^{t,x})- f_u(t,x))H_s^{t,x}]\rmd s,
\end{align*}
which gives
\begin{equation}
    \begin{aligned}\label{eq:B-E-L_formula}
             \nabla_x u(t,x) =\EE[(g(X_T^{t,x})-g(x)) H_T^{t,x}] + \int_{t}^T \EE [(f_{u}(s,X_s^{t,x})- f_u(t,x))H_s^{t,x}]\rmd s.
    \end{aligned}
\end{equation}
This type of control variate dates back to~\cite{alanko2013reducing} for one-step computation of conditional expectations in the numerical solution of BSDEs, and has since appeared in various forms in~\cite{warin2018monte,chan2019machine,henry2019branching,germain2021neural} for variance reduction in stochastic simulations related to PDEs. We state its form above in our continuous-time setting based on the Feynman-Kac formulation. The theorem below shows that the estimator with the control variate has finite variance, and we provide the proof in this setting for completeness.

\begin{theorem}
\label{thm:finite_var}
    Assume Assumption~\ref{assump1} holds. For any $t \in [0,T)$, $x \in \RR^d$, $g \in C^1(\RR^d)$ and $f \in C^1([0,T]\times \RR^d)$ with bounded first-order derivatives, we have
    \begin{align*}
    \sup_{s \in [t,T)}\EE|(g(X_T^{t,x})-g(x)) H_T^{t,x}|^2 < +\infty,\, 
    \int_{t}^T \EE |(f(s,X_s^{t,x})- f(t,x))H_s^{t,x}|^2\rmd s < +\infty.
    \end{align*}
    \end{theorem}
    
    \begin{proof}
    Throughout the proof, we will use $C$ as a positive constant, which is independent of $t, s$ and $x$ and may vary from line to line.
    First, the Cauchy-Schwarz inequality gives us
    \begin{equation}\label{thm:eq1}
        \begin{aligned}
        &\EE|(g(X_T^{t,x})-g(x)) H_T^{t,x}|^2 \le \left(\EE|g(X_T)^{t,x} -g(x)|^4\right)^{\frac{1}{2}} \left(\EE|H_T^{t,x}|^4\right)^{\frac{1}{2}} .
    \end{aligned}
    \end{equation}
    By the mean value theorem, there exists $\eta \in [0,1]$ such that $g(X_T^{s,x}) - g(x) = \nabla_x g(\eta x + (1-\eta)X_T^{s,x})\cdot(X_T^{s,x} - x)$. Then, noticing that $\nabla_x g$ is bounded, we have
    \begin{equation*}
        \EE\left|g(X_T^{s,x}) - g(x)\right|^4 \le C\EE|X_T^{s,x} - x|^4.
    \end{equation*}
    Through the standard estimate of the forward SDE~\eqref{eq:forward_SDE_general} (see, e.g., \cite[Theorem 3.4.3]{zhang2017backward}), we have
    \begin{equation}\label{thm:eq2}
        \begin{aligned}
        &\EE\left|g(X_T^{s,x}) - g(x)\right|^4
        \le  C\EE|X_T^{s,x} - x|^4 \\
        \le &C \left(\EE\bigg[\int_s^T |\mu(r,0)| \rmd r\bigg]^4 + \EE\bigg[\int_s^T \tr(\sigma\sigma\transpose(r,0))\rmd r\bigg]^2\right)\\
        \le &C(T^2(T-s)^2 + (T-s)^2)
        \le C(T-s)^{2}.
    \end{aligned}
    \end{equation}
    Similarly, with \cite[Theorem 3.4.3]{zhang2017backward} again, we have
    $$
    \EE[\tr((D_r^{s,x})\transpose D_r^{s,x})]^2 \le C.
    $$ Therefore, the Burkholder-Davis-Gundy inequality~\cite[Theorem 2.4.1]{zhang2017backward} gives us
    \begin{equation}\label{thm:eq3}
        \begin{aligned}
        &\EE|H_T^{t,x}|^4 = (T-t)^{-4}
        \EE\bigg|\int_{s}^{T}[\sigma(r,X_r^{s,x})^{-1}D_r^{s,x}]\transpose\rmd W_r\bigg|^4\\ \le
        & C(T-t)^{-4} \EE \bigg[ \int_{s}^T \tr([\sigma^{-1}(\sigma^{-1})\transpose](r,X_r^{s,x}) (D_r^{s,x})\transpose D_r^{s,x}) \rmd r\bigg]^2 \\
        \le& C(T-t)^{-4}\EE \bigg[\int_{s}^T \tr((D_r^{s,x})\transpose D_r^{s,x})\rmd r\bigg]^2 \\
        \le& C(T-s)^{-2},
    \end{aligned}
    \end{equation}
    where we have used that
    \begin{equation}
        \begin{aligned}\label{thm:eq4}
        \tr([\sigma^{-1}(\sigma^{-1})\transpose](r,X_r^{s,x}) (D_r^{s,x})\transpose D_r^{s,x}) =& \tr(D_r^{s,x}[\sigma^{-1}(\sigma^{-1})\transpose](r,X_r^{s,x}) (D_r^{s,x})\transpose ) \\
        =& \sum_{i=1}^d D_r^{s,x,i}[\sigma^{-1}(\sigma^{-1})\transpose](r,X_r^{s,x}) (D_r^{s,x,i})\transpose \\
        \le& C\sum_{i=1}^d D_r^{s,x,i}(D_r^{s,x,i})\transpose \\
        =& C\,\tr((D_r^{s,x})\transpose D_r^{s,x}),
    \end{aligned}
    \end{equation}
    in which $D_r^{s,x,i}$ is the $i$-th row of $D_r^{s,x}$. Combining \eqref{thm:eq1} \eqref{thm:eq2} and \eqref{thm:eq3}, we obtain 
    $$\EE|(g(X_T^{t,x})-g(x)) H_T^{t,x}|^2 \le C.
    $$
    Similarly, we can prove that 
    \[
     \EE |(f(s,X_s^{t,x})- f(t,x))H_s^{t,x}|^2\le C.
    \]
    Hence,
\begin{align*}
       \int_{t}^T  \EE |(f(s,X_s^{t,x})- f(t,x))H_s^{t,x}|^2\rmd s \le C.
\end{align*}
\end{proof}

Now we return to explain why the original estimator has infinite variance.
\begin{proof}[Proof of Theorem~\ref{thm:infinite_var}]
    By the elementary inequality $a^2 + (a-b)^2 \geq  b^2/2$, we have that
    \begin{align*}
        \EE|g(X_T^{t,x}) H_T^{t,x}|^2 &\ge \frac{1}{2}\EE|g(x) H_T^{t,x}|^2 - \EE|(g(X_T^{t,x})- g(x)) H_T^{t,x}|^2,\\
        \int_{t}^T\EE|f(s,X_s^{t,x}) H_s^{t,x}|^2\rmd s &\ge \frac{1}{2}\int_{t}^T\EE|f(t,x) H_s^{t,x}|^2\rmd s \\
        &- \int_{t}^T\EE|(f(s,X_s^{t,x}) - f(t,x)) H_s^{t,x}|^2\rmd s.
    \end{align*}
Therefore, given Theorem~\ref{thm:finite_var}, we only need to prove
$$
\lim_{s \rightarrow T^{-}}\EE|H_T^{s,x}|^2 = +\infty,\,
\int_{t}^T \EE |H_s^{t,x}|^2\rmd s = +\infty.
$$
First, similar to \eqref{thm:eq4}, we have
    $$\tr([\sigma^{-1}(\sigma^{-1})\transpose](r,X_r^{s,x}) (D_r^{s,x})\transpose D_r^{s,x}) \ge C\,\tr((D_r^{s,x})\transpose D_r^{s,x}).$$ 
Therefore,
    \begin{align*}
        \EE|H_T^{s,x}|^2
        =& (T-s)^{-2}\EE\bigg[\int_{s}^{T}\tr([\sigma^{-1}(\sigma^{-1})\transpose](r,X_r^{s,x}) (D_r^{s,x})\transpose D_r^{s,x})\rmd r\bigg]^2 \\
        \ge& C(T-s)^{-2}\EE \int_{s}^T \tr((D_r^{s,x})\transpose D_r^{s,x})\rmd r.
    \end{align*}
    With \cite[Theorem 5.2.2]{zhang2017backward}, we have $$
    \EE|\tr((D_r^{s,x})\transpose D_r^{s,x}) - \tr(I_d\transpose I_d)|\le C(r-s).$$
    Hence, when $r-s \le C$, we have
    $$
    \EE\tr((D_r^{s,x})\transpose D_r^{s,x}) \ge \frac{d}{2}.
    $$
    Therefore,
    \begin{align*}
        \EE|H_T^{s,x}|^2 
        \ge& C(T-s)^{-2} \EE \int_{s}^{\min\{T,s+C\}}\tr((D_r^{s,x})\transpose D_r^{s,x})\rmd r \\
        \ge& C(T-s)^{-2}\min\{T-s,C\},
    \end{align*}
which means that 
$$
\lim_{s \rightarrow T^{-}}\EE|H_T^{s,x}|^2 = +\infty.$$ 
Similarly, we have
$$
 \EE |H_s^{t,x}|^2 \ge C(s-t)^{-2}\min\{s-t,C\},
$$
which means that 
\begin{align*}
\int_{t}^T\EE |H_s^{t,x}|^2\rmd s = +\infty.
\end{align*}
\end{proof}

Building on the above analysis, we can now apply the control-variate version of Bismut-Elworthy-Li formula to the Picard iteration defined in~\eqref{eq:Picard}, yielding a similar relationship:
\begin{equation}
\begin{aligned}\label{eq:picard_B-E-L}
        \nabla_x u_{k+1}(t,x) =  \EE[(g(X_T^{t,x})-g(x)) H_T^{t,x}] + \int_{t}^T \EE [(f_{u_k}(s,X_s^{t,x})- f_{u_k}(t,x))H_s^{t,x}]\rmd s.
\end{aligned}
\end{equation}
Accordingly, we can plug the current approximation to $u_k$ into the right-hand side of~\eqref{eq:picard_B-E-L} to generate gradient labels of $u_{k+1}$ for better regression.

We should mention that the Bismut-Elworthy-Li formula can be extended to estimate the Hessian term. For instance, when $\mu \equiv 0$ and $\sigma \equiv \II_d$ in \eqref{eq:forward_SDE_general}, the formula for the second derivative becomes
    \begin{align*}
         \nabla_x^2 u_{k+1}(t,x) =\,  &\EE\bigg[g(X_T^{t,x})\frac{4(W_T - W_{\frac{T+t}{2}})(W_{\frac{T+t}{2}} - W_t)\transpose}{(T-t)^2}\bigg] \\
     &+\int_{t}^T\EE\bigg[f_{u_k}(s,X_s^{t,x})\frac{4(W_s - W_{\frac{s+t}{2}})(W_{\frac{s+t}{2}} - W_t)\transpose}{(s-t)^2}\bigg]\rmd s.
    \end{align*}
    Readers seeking a more general formulation are referred to Theorem 2.3 in \citep{ELWORTHY1994252}. In Section A of the supplementary material, we present a control-variate version of the above estimate with finite variance, like \eqref{eq:B-E-L_formula}, along with numerical results for the DPI using the Hessian estimator. Since numerical experiments thus far suggest that incorporating the Hessian term primarily enhances the accuracy of the Hessian itself rather than the value or gradient, we have opted to place this in the supplementary material.

\section{Numerical Algorithm}
\label{sec:alg}
To numerically implement the methodology introduced in Section~\ref{sec:dpi}, we replace each $u_k$ with $u_{\theta_k}$, a neural network with parameters $\theta_k$. When the context is clear, references to $u_k$ henceforth (including those used in earlier equations) should be understood as $u_{\theta_k}$ without further specification. Given $u_k$, we use equations \eqref{eq:Picard} and \eqref{eq:picard_B-E-L} to generate labels $(y_i, z_i)$ for $u$ and $\nabla u$ on sampled points $(t_i, x_i)$ with $f_{u_k}$ evaluated through automatic differentiation of $u_k$.
We then train $u_{k+1}$ on those labels through supervised learning using the following loss function:
\begin{equation}\label{loss_dpi}
    \mathcal{L}_{\text{DPI}}(\theta) = \frac{1}{N} \sum_{i=1}^N \left[|y_i - u_\theta(t_i, x_i)|^2 + \frac{\lambda}{d} |z_i - \nabla_x u_\theta(t_i, x_i)|^2\right],
\end{equation}
where $\lambda \ge 0$ balances the loss between the value and gradient terms. The overall procedure is summarized in Algorithm \ref{alg_dpi}, and several computational details involved in Algorithm \ref{alg_dpi} are discussed below.

\begin{algorithm}[!ht]
\caption{Deep Picard Iteration (DPI) Algorithm}
\KwIn{Number of Picard iterations $K$, number of data points $N$ per iteration, number of Monte Carlo sampling $M$, number of epochs $E$ per iteration for training neural networks, and weight factor $\lambda \ge 0$.\\
\textbf{Initialize:} $u_0(t,x) = 0$. }

\For{$k = 0, 1,\dots, K-1$}{
Sample $N$ pairs $\{(t_i, x_i)\}_{i=1}^N$ by first sample $t_i$ uniformly from $[0, T]$ and $x_i$ according to the distribution of $X_{t_i}$ in~\eqref{eq:forward_SDE_x0}.

Compute labels $\{(y_i, z_i)\}_{i=1}^N$ according to \eqref{eq:mc_app_of_yi} and \eqref{eq:mc_app_of_zi} with $u_k$, respectively.

If \( k = 0 \), initialize the weights \(\theta_{k+1}\) in the neural network for \( u_{k+1} \) randomly; otherwise, initialize it with the optimized weights \(\theta_k\) from \( u_k \).

Train the neural network for $E$ epochs on the training data by minimizing the supervised loss \eqref{loss_dpi} to obtain \( u_{k+1} \) with optimized weights \(\theta_{k+1}\). 
}
\KwOut{$u_K(t,x)$}
\label{alg_dpi}
\end{algorithm}

\paragraph{Data distribution}
The loss function~\eqref{loss_dpi} is defined on data points $\{(t_i, x_i)\}_{i=1}^N$ for which we need to specify its distribution. We achieve this using the forward SDE, as commonly done in the literature.
Let $X_t$ denote the solution of the following SDEs
\begin{equation}\label{eq:forward_SDE_x0}
    X_t = \xi + \int_{0}^t \mu(s,X_s)\rmd s + \int_0^t \sigma(s,X_s)\rmd W_s, \quad t \in [0,T],
\end{equation}
where $\xi$ is a \(d\)-dimensional square-integrable random variable, which is independent of $\{W_t\}_{0 \le t \le T}$. 
First, we sample $t_i$ uniformly from $[0, T]$ and then $x_i$ according to the distribution of $X_{t_i}$. Uniform sampling in time ensures the solution is uniformly accurate over time for Picard iteration. The distribution of $x_i$ is more subtle, as it depends on three factors: the initial distribution $\xi$, the drift function $\mu$, and the diffusion function $\sigma$.
The support of $\xi$ mainly reflects the spatial region of interest for the solution at the initial time $t=0$.
As explained earlier, the choice of $\mu$ and $\sigma$ is also not unique but sometimes can be related to the underlying probabilistic problem associated with the PDE, such as a stochastic control or sampling problem. These coefficients should also guide the training process toward the regions where the PDE solution is relevant. For further discussion, see \cite{pham2021neural,nusken2021solving}. 
In the numerical experiments below, we mainly let $X_t$ be standard Brownian motion for simplicity, ensuring a fair comparison with other methods.

\paragraph{Monte Carlo integration}
At given $(t_i, x_i)$, the labels $y_i \approx u_{k+1}(t_i,x_i)$ and $z_i \approx \nabla_x u_{k+1}(t_i,x_i)$ are computed numerically using the Monte Carlo approximations according to \eqref{eq:FeynmanKac} and \eqref{eq:picard_B-E-L}, respectively:
\begin{align}
        &y_{i} = \frac{1}{M}\sum_{j=1}^M [g(X_T^{t_i,x_i,i,j}) +(T-t_i) f_{u_k}(s^{i,j},X_{s^{i,j}}^{t_i,x_i,i,j})],\label{eq:mc_app_of_yi}\\
        &z_{i} = \frac{1}{M}\sum_{j=1}^M \bigg[(g(X_T^{t_i,x_i,i,j})-g(x_i)) H_T^{t_i,x_i,i,j} +\label{eq:mc_app_of_zi}\\
        &\quad\quad(T-t_i) (f_{u_k}(s^{i,j},X_{s^{i,j}}^{t_i,x_i,i,j})- f_{u_k}(t_i,x_i)) H_{s^{i,j}}^{t_i,x_i,i,j}],\nonumber 
    \end{align}
where $\{W_r^{i,j}\}_{1 \leq i \leq N, 1 \leq j \leq M, r \in [t_i, T]}$ are independently sampled paths of Brownian motions, time points $\{s^{i,j}\}_{1 \leq i \leq N, 1\leq j\leq M}$ are uniformly sampled from $[t_i, T]$, and $X_s^{t,x,i,j}$ and $H_s^{t,x,i,j}$ are samples to $X_s^{t,x}$ and $H_s^{t,x}$ by replacing $W_t$ with $W_t^{i,j}$. 

\paragraph{Sample generation}
As already mentioned in the previous two paragraphs, Algorithm \ref{alg_dpi} requires sampling of \( X_t \), \( X_s^{t,x} \), and \( H_s^{t,x} \). Now we explain how these samples can be obtained directly for several commonly encountered SDEs, including those used in the numerical experiments below.  
In such scenarios, our numerical experiments suggest that computing labels for $z_i$ only requires less than 20\% more time than computing labels for $y_i$; further details are provided in Section~\ref{sec:cha}.
For general SDEs in which these quantities can not be directly sampled, one can use Euler-Maruyama or any other discretization schemes to generate these samples. We focus on the sampling of \( X_s^{t,x} \) and \( H_s^{t,x} \), the sampling for \( X_t \) is similar to that of \( X_s^{t,x} \). To ease the notation, we set \( t = 0 \) and omit the superscript \( t,x \) in the subsequent expressions. 

\begin{enumerate}
    \item Brownian motion ($\mu \equiv 0$ and $\sigma \equiv \II_d$): 
    $$
    X_s = x + W_s, \,
    D_s = \II_d, \, \text{and} \, 
    H_s = s^{-1}\int_0^s[\sigma(r,X_r)^{-1} D_r]\transpose \rmd W_r = s^{-1} W_s.
    $$
    \item Geometric Brownian motion ($\mu \equiv 0$ and $\sigma = \diag(x)$):
    \begin{align*}
    & X_s = \diag(\exp(-\frac12s + W_s)) x,  \quad
    D_s =\diag(\exp(-\frac12s + W_s)), \\
    & H_s = s^{-1}\int_{0}^s[\sigma^{-1}(r,X_r) D_r]\transpose \rmd W_r = s^{-1}\diag(x_1^{-1},\dots,x_d^{-1}) W_s.
    \end{align*}
    
    \item 
   
    Ornstein–Uhlenbeck process ($\mu = -\theta x$ and $\sigma \equiv \II_d$):
   
    \begin{align*}
        &X_s =  e^{-\theta s}x +  \int_{0}^s e^{\theta(r-s)}\rmd W_r, \quad D_s = e^{-\theta s}\II_d, \\
        & H_s = s^{-1}\int_{0}^s[\sigma^{-1}(r,X_r) D_r]\transpose \rmd W_r = s^{-1}\int_{0}^s e^{-\theta r} \rmd W_r.
    \end{align*} 
    To efficiently sample $X_s$ and $H_s$, consider the auxiliary random variables $U_s \coloneqq \int_{0}^s e^{\theta(r-s)}\rmd W_r$ and $V_s \coloneqq \int_{0}^s e^{-\theta r} \rmd W_r$, which are jointly Gaussian with zero mean. By It\^{o} isometry, we have
    \begin{align*}
      &\EE [U_s\transpose U_s] = \EE \int_{0}^s e^{2\theta(r-s)}\II_d\rmd r = \frac{1}{2\theta}(1 - e^{-2\theta s})\II_d, \\
    &\EE [V_s\transpose V_s] = \EE \int_{0}^s e^{-2\theta r}\II_d\rmd r = \frac{1}{2\theta}(1 - e^{-2\theta s})\II_d, \\
        &\EE [U_s\transpose V_s] = \EE \int_{0}^s e^{-\theta s} \II_d \rmd r = s e^{-\theta s}\II_d.
    \end{align*}
    Thus, joint samples of $(U_s, V_s)$ can be obtained by generating  $2d$-dimensional Gaussian vectors with zero mean and the corresponding covariance matrix, which can then be transformed directly to samples of $X_s$ and $H_s$.
    
    Such a scheme also facilitates further variance reduction of our Monte Carlo estimators in the OU setting. In particular, by the property of conditional expectation, the first term in~\eqref{eq:B-E-L_formula} can be rewritten as
    \begin{equation*}
        \begin{aligned}
            &\EE[(g(X_T)-g(x)) H_T] \\
            = \, &\EE[T^{-1}(g(e^{-\theta T}x + U_T)-g(x))V_T] \\
            =\, &\EE[\EE[T^{-1}(g(e^{-\theta T}x + U_T)-g(x))V_T \mid U_T]] \\
            =\, &\EE[T^{-1}(g(e^{-\theta T}x + U_T)-g(x))\EE[V_2 \mid U_T]].
        \end{aligned}
    \end{equation*}
    Therefore, replacing $V_T$ with $\EE[V_T \mid U_T]$ yields an unbiased estimator of $\EE[(g(X_T) - g(x)) H_T]$ with reduced variance. Noting that $U_T$ and $V_T$ are jointly Gaussian with zero mean, we have $
        \EE[V_T \mid U_T] = \rho_T U_T,
    $ where $\rho_T \coloneqq \frac{2\theta T e^{-\theta T}}{1 - e^{-2\theta T}}<1$. The same technique applies to the second term in~\eqref{eq:B-E-L_formula}, namely, $\int_t^T \EE[(f_u(s, X_s^{t,x}) - f_u(t,x)) H_s^{t,x}] \rmd s$.
    Therefore, we can use
    \begin{equation*}
        \tilde{H}_s = s^{-1}\rho_s M_1 = \frac{2\theta e^{-\theta s}}{1-e^{-2\theta s}}\int_0^s e^{\theta(r-s)}\rmd W_r    
    \end{equation*}
    instead of $H_s$ to achieve further variance reduction.
\end{enumerate}

\subsection{Conceptual Comparison with Established Methods}
In this subsection, we briefly review a few representative established methods from the literature that will be benchmarked in the numerical section, followed by a conceptual comparison with DPI.

\paragraph{PINN with Hutchinson trace estimation}
For the PDE~\eqref{eq:pde}, the PINN loss is formulated as
\begin{equation}
\begin{aligned}
\mathcal{L}_{\text{PINN}}(\theta) 
= &  \frac{1}{T}\int_{0}^T\mathbb{E} \left| \partial_t u_\theta(t,X_t) + F_{u_\theta}(t,X_t) \right|^2\rmd t + \lambda_T \mathbb{E} \left|u_\theta(T, X_T)-g(X_T)\right|^2,
\label{loss_pinn}
\end{aligned}
\end{equation}
where the weight $\lambda_T > 0$ is used to balance the residual and terminal losses. 
When using PINN to solve high-dimensional second-order PDEs, computing the Hessian matrix is often memory-intensive and time-consuming. To address this, \cite{hu2024hutchinson} proposes using Hutchinson trace estimation (HTE)~\citep{hutchinson1989stochastic} to estimate the trace of the Hessian matrix, rather than computing the full Hessian, to reduce computational costs. We implement this technique in our implementation and refer to the resulting method as ``PINN-HTE".
Specifically, HTE uses random variables \(\boldsymbol{v} \in \mathbb{R}^d\) that satisfy \(\mathbb{E}_{\boldsymbol{v} \sim p(\boldsymbol{v})}\left[\boldsymbol{v} \boldsymbol{v}^T\right] = \II_d\) to estimate the trace of a matrix \(A\) as \(\operatorname{Tr}(A) = \mathbb{E}_{\boldsymbol{v} \sim p(\boldsymbol{v})}\left[\boldsymbol{v}^{\mathrm{T}} A \boldsymbol{v}\right]\). This can be approximated by \(\sum_{i=1}^{N_v} \boldsymbol{v}_i^{\mathrm{T}} A \boldsymbol{v}_i/N_v\) through computing the vector-Hessian product instead of the full Hessian matrix.
Each random vector $\boldsymbol{v}_i \in \mathbb{R}^d$ is independently sampled from $p(\boldsymbol{v})$, which is recommended to be the Rademacher distribution to minimize the variance of HTE~\citep{skorski2021modern}. We further notice that HTE is applicable only to semilinear PDEs, while for fully nonlinear PDEs, computing the full Hessian matrix is unavoidable.

\paragraph{Deep BSDE with diffusion-type loss} 
The work \cite{nusken2021interpolating} proposes a powerful variation of Deep BSDE method for semilinear PDEs with a diffusion-type loss:
\begin{align}
   & \mathcal{L}_{\text {D-DBSDE}}(\theta) =  \frac{1}{T}\int_{0}^{T}\mathbb{E}\left|{u}_\theta\left(t_K, X_{t_K}\right) -  {u}_\theta \left(t, X_{t}\right) + \int_{t}^{t_K} f_{u_\theta}(s,X_s) \mathrm{d} s \right. \notag \\
    &-\left.\int_{t}^{t_K} \sigma^{\top}(s, X_s) \nabla  {u}_\theta  \left(s, X_s\right) \mathrm{d} W_s   \right|^2\rmd t
     + \lambda_T \mathbb{E} \left|u_\theta(T, X_T)-g({X_T})\right|^2.
     \label{loss_bsde}
\end{align}
Here, $\lambda_T$ again serves to penalize the terminal cost. The choice of $t_K$ determines the time of the diffusion process: as $t_K \rightarrow t^+$, the loss converges to that of PINN, and as $t_K \rightarrow T$, the loss can be seen as a simple variation of the loss used in the Deep BSDE method. The additional parameter $t_K$ enables us to balance the local approximation in the PINN loss with the global approximation in the BSDE loss, leading to improved performance. From this point on, we refer to this approach as ``D-DBSDE". 

\paragraph{Deep backward dynamic
programming (DBDP)} The DBDP method introduced in \citep{pham2021neural} generalizes the original DBDP method for semilinear PDEs \citep{hure2020deep} to fully nonlinear PDEs. Different from other methods, DBDP needs to use a single network that outputs a $(d+1)$-dimensional vector at each discrete time step to represent $u(t_i, \cdot)$ and $\nabla_x u(t_i, \cdot)$ on a predefined time grid  $0=t_0 < t_1 < \cdots < t_N=T$. This approach forms a series of networks denoted as $\{(u_i, z_i)(\cdot; \theta_i)\}_{i=0}^{N}$.  The first step involves learning $\theta_N$ to approximate the terminal condition $g$ through the square loss $\mathcal{L}_{\text{DBDP}}^N(\theta_N) = \EE|u_N(X_T;\theta_N)- g(X_T)|^2$.
 Then, at the $i$-th time step, DBDP learns $\theta_{N+1-i}$ through the loss $\mathcal{L}_{\text{DBDP}}^{N+1-i}(\theta_{N+1-i})$, where
\begin{align}
        \mathcal{L}_{\text{DBDP}}^i(\theta_i) =& \, \EE\left |u_{i+1}(X_{t_{i+1}};\theta_{i+1}) -  u_i(X_{t_i};\theta_i) - z_i(X_{t_i};\theta_i)\transpose\sigma(t_i,X_{t_i})\Delta W_i \right.  \notag \\
        & \quad~+ \left.  f(t_i,X_{t_i},u_i(X_{t_i};\theta_i),z_i(X_{t_i};\theta_i),\nabla_x z_i(X_{t_{i+1}
        };\theta_{i+1}))\Delta t \right|^2.
        \label{loss_dbdp}
\end{align}
Unlike other methods, where updating network parameters can improve approximation accuracy globally in time, DBDP requires optimal results at each time step to control error accumulation. This step-by-step optimization can make DBDP more time-consuming compared to other methods, especially when high accuracy is required.

\paragraph{Differential regression learning via stochastic control}
When $F$ in \eqref{eq:pde} is a convex Hamiltonian, \cite{lefebvre2023differential} reformulates the PDE as a stochastic optimal control problem. They first obtain the optimal control by optimizing a neural network-based policy, following the approach in~\cite{han2016deep}, and then apply differential regression learning to estimate the solution $u$ and its gradient $\nabla u$. While this approach is effective for relatively low-dimensional problems, it may face two key challenges in higher dimensions. First, directly solving the optimal control problem involves complex optimization, requiring substantial computational resources. Second, errors from the pre-computed control network can accumulate during the subsequent estimation of $u$ and $\nabla u$, potentially compromising overall accuracy. In our experiments with the 100-dimensional HJB example from Section~\ref{sec:hjb}, we find this method significantly more time-consuming and less accurate compared to the DPI method. Thus, detailed numerical comparisons are omitted in the next section.

With these methods outlined, we can now examine how our proposed DPI method compares conceptually. The most significant difference lies in the convexity of the loss functions as a variational problem, before considering neural network approximation. The loss function in DPI, given by~\eqref{loss_dpi}, is convex with the target function $u_{\theta}$, as a result of the least-squares regression formulation. In contrast, the loss functions used in other methods, such as ~\eqref{loss_pinn}, ~\eqref{loss_bsde}, and ~\eqref{loss_dbdp}, which are based directly on fixed-point equations, are not convex with respect to the target function. We believe that this fundamental difference persists even when training neural networks as a finite-dimensional optimization problem, resulting in a much easier optimization process for DPI compared to other methods, ultimately leading to better accuracy in the final solution, although the finite-dimensional optimization problem itself is non-convex with respect to neural network parameters. We further support this conceptual comparison with empirical evidence in Section B of the supplementary material, where we plot the test error curves for all three numerical examples considered in this paper. DPI consistently exhibits a smoother and more monotonic decrease in error throughout training, in contrast to the often fluctuating behavior observed in other methods.

It is also worth noting that the data generation process in DPI, the most time-consuming part of our algorithm, can be easily parallelized across multiple CPUs and/or GPUs, significantly accelerating the algorithm.
For example, in Section~\ref{sec:fully}, the data generation time takes more than six times longer than the training, which can be greatly reduced with additional computing resources.
This ease of parallelization is another key advantage of our regression-based approach, which separates data generation from the learning process, making it more scalable and efficient than other methods. 
Although our experiments used a single GPU and already achieved superior results, parallelization will enable us solve much larger problems more efficiently in the future. Furthermore, with more computation resources for generating labels in parallel, we can use larger $N$ and \(M\), achieving better performance in less time.

\section{Numerical Results}

\subsection{Experimental settings}\label{sec:exp_settings}

In this section, we use the proposed DPI to solve three distinct high-dimensional problems, comparing its performance against other state-of-the-art approaches. Specifically, we solve two semilinear problems in Section~\ref{sec:cha} and Section~\ref{sec:hjb}, where we compare our method to PINN-HTE and the diffusion-type Deep BSDE method (D-DBSDE)~\citep{nusken2021interpolating}. Additionally, we solve a fully nonlinear problem in Section~\ref{sec:fully}, benchmarking our approach against standard PINN and DBDP \cite{hure2020deep,pham2021neural}. We do not include DBDP in the first two examples, as its performance is always slightly worse than that of DBSDE, and including both would make the result presentation unnecessarily crowded. The spatial dimension $d$ in all the PDEs solved is fixed at $100$. 
All methods are executed within the same computation time constraints on a single V100 GPU with 32GB memory. Each experiment is repeated three times with different random seeds, and we report the mean and standard deviation of the results. 

In our experiments, we use multilayer perceptrons (MLPs) with the ELU activation function. Specifically, Section~\ref{sec:cha} uses an MLP with four hidden layers of 128 neurons each, while Section~\ref{sec:fully} uses a smaller MLP with three hidden layers of 64 neurons each. For the HJB problem in Section~\ref{sec:hjb}, we adopt a specialized network architecture commonly used in the literature, with details provided later.

The SDEs are simulated with $\mu \equiv 0.0$ and $\sigma \equiv 1.0$, starting at $X_0  = \xi = 0$ except for the case in Section~\ref{sec:hjb}.
This simulated data is used to define the data distribution in DPI loss~\eqref{loss_dpi}, as explained in Section~\ref{sec:alg}, and the same distribution is also used for the training objectives in PINN, Deep BSDE, and DBDP.
We use the Adam optimizer with a fixed learning rate of 0.001 and a batch size of 512 for all experiments. For the other methods, each network is trained for as many epochs as possible within the total computation time budget. For our DPI, since there is an outer Picard iteration, we also specify the number of epochs used in each iteration given $N$ samples in \eqref{loss_dpi}. 
Key hyperparameters for DPI across the three examples are summarized in  Table~\ref{tab:dpi_params}.
For PINN-HTE, following the recommendation in \cite{hu2024hutchinson}, we set $N_v=16$ when estimating the trace of a matrix $A$ through \(\sum_{i=1}^{N_v} \boldsymbol{v}_i^{\mathrm{T}} A \boldsymbol{v}_i/N_v\). For D-DBSDE, we set $t_K = \min\{t + \Delta t, T\}$ and approximate the integral numerically by discretizing the interval $[t, t_K]$ with a time step $\delta t = 0.005$. We have $\Delta t = 0.1$ in Section~\ref{sec:cha} and $\Delta t = 0.25$ in Section~\ref{sec:hjb}.

\begin{table}[!ht]
\centering
\caption{Hyperparameters used in DPI, including the total number of iterations $K$, the number of samples $M$ utilized in the Monte Carlo approximation at each data point,  the data set size $N$ employed in each Picard iteration step, and the number of epochs $E$ employed in each Picard iteration.}\label{tab:dpi_params}
\begin{tabular}{lccccccc}
\toprule
PDE  & $K$ & $M$ & $N$ & $E$ & \makecell[{{t}}]{Data generation\\ time (s)} & \makecell[{{t}}]{Training\\ time (s)}\\ \midrule
{Burgers-type~(Sec~\ref{sec:cha})}  & 80  &      4096      &    4096 &  16 &    182.4 & 153.6   \\ 
{HJB~(Sec~\ref{sec:hjb})}  & 40  &   4096     &  4096 & 16 & {717.6} &     {81.2}  \\
{Fully nonlinear~(Sec~\ref{sec:fully})}  & 40 &    2048        &   1024   & 16 & 688.0 & 40.5     \\
\bottomrule
\end{tabular}
\end{table}

For evaluation, we generate 10,000 data points from the same distribution used in training. We quantify the performance using the relative mean absolute error of value (rMAE) and relative mean absolute error of gradient (g-rMAE) as:
\begin{align*}
    \operatorname{rMAE} &=\frac{\displaystyle{\sum_i\left|u_\theta (t_i,X_{t_i})-u^*(t_i,X_{t_i})\right|}}{\displaystyle{\sum_i |u^*(t_i,X_{t_i})|}}, \\
    \operatorname{g-rMAE} &= \frac{1}{d} \sum_{j=1}^d \frac{\displaystyle{\sum_i\left|\partial_{x_j} u_\theta(t_i, X_{t_i}) - \partial_{x_j} u^*(t_i, X_{t_i})\right|}}{\displaystyle{\sum_i|\partial_{x_j} u^*(t_i, X_{t_i})|}},
\end{align*}
where $u^*$ denotes the ground-truth solution. We also compute the relative squared error in addition to the relative absolute error, and find that both types of errors lead to the same conclusions when comparing different methods. Therefore, to avoid redundancy, we will only report the rMAE and g-rMAE metrics.

\subsection{A semilinear Burgers-type PDE}\label{sec:cha}

In this subsection, we compare DPI with PINN-HTE and D-DBSDE in a semilinear Burgers-type PDE considered in~\citep{chassagneux2014linear, han2017deep} as follows:

\begin{equation}\label{eq:cha}
    \partial_t u(t, {x}) + \frac{\sigma^2}{2} \Delta u(t, {x}) + \left[ \frac{\kappa\sigma^2}{\sqrt{d}} (u - \frac{1}{2}) - \frac{\sqrt{d}}{\kappa}  \right] \sum_{i=1}^d \frac{\partial u}{\partial x_i}(t, {x}) = 0.
\end{equation}
 When the terminal condition is 
 \[
 g(x) = \frac{e^{(T + \frac{\kappa}{\sqrt{d}} \sum_{i=1}^d x_i)}}{1 + e^{(T + \frac{\kappa}{\sqrt{d}} \sum_{i=1}^d x_i)}},
 \] the exact solution is given by 
 \[
 u^*(t, {x}) = \frac{e^{(t + \frac{\kappa}{\sqrt{d}} \sum_{i=1}^d x_i)}}{1 + e^{(t + \frac{\kappa}{\sqrt{d}} \sum_{i=1}^d x_i)}}.
 \]
We follow the previous settings \(\sigma = 1.0\) and \(T = 1.0\). We enlarge the parameter \(\kappa\) from 1.0 to 2.5 and then to 5.0 to increase the nonlinearity of the PDE, allowing us to evaluate the performance of different methods across varying levels of nonlinearity. %
The weight $\lambda$ in DPI or $\lambda_T$ in PINN-HTE and D-DBSDE is tuned within a broad range from 0.01 to 10000. 

We first demonstrate the robustness of DPI's weight parameter $\lambda$ compared to the terminal weight $\lambda_T$ used in PINN-HTE and D-DBSDE. Taking $\kappa = 2.5$ as an example, Figure~\ref{fig:weight} shows that the terminal weight $\lambda_T$ significantly affects the performance of PINN-HTE and D-DBSDE, necessitating adjustments to $\lambda_T$ to achieve a reasonable solution. In contrast, DPI, with an extremely broad range of $\lambda$, maintains outstanding and robust performance, highlighting its superior stability in the weight tuning. We remark that for $\lambda = 0$ in DPI, where supervision is applied only to the function value of $u$ itself, the results are still sufficiently good, although not the best among all the tested weights.

\begin{figure}[!htb]
    \centering
    \begin{subfigure}[b]{0.49\textwidth}
        \centering
\includegraphics[width=\textwidth]{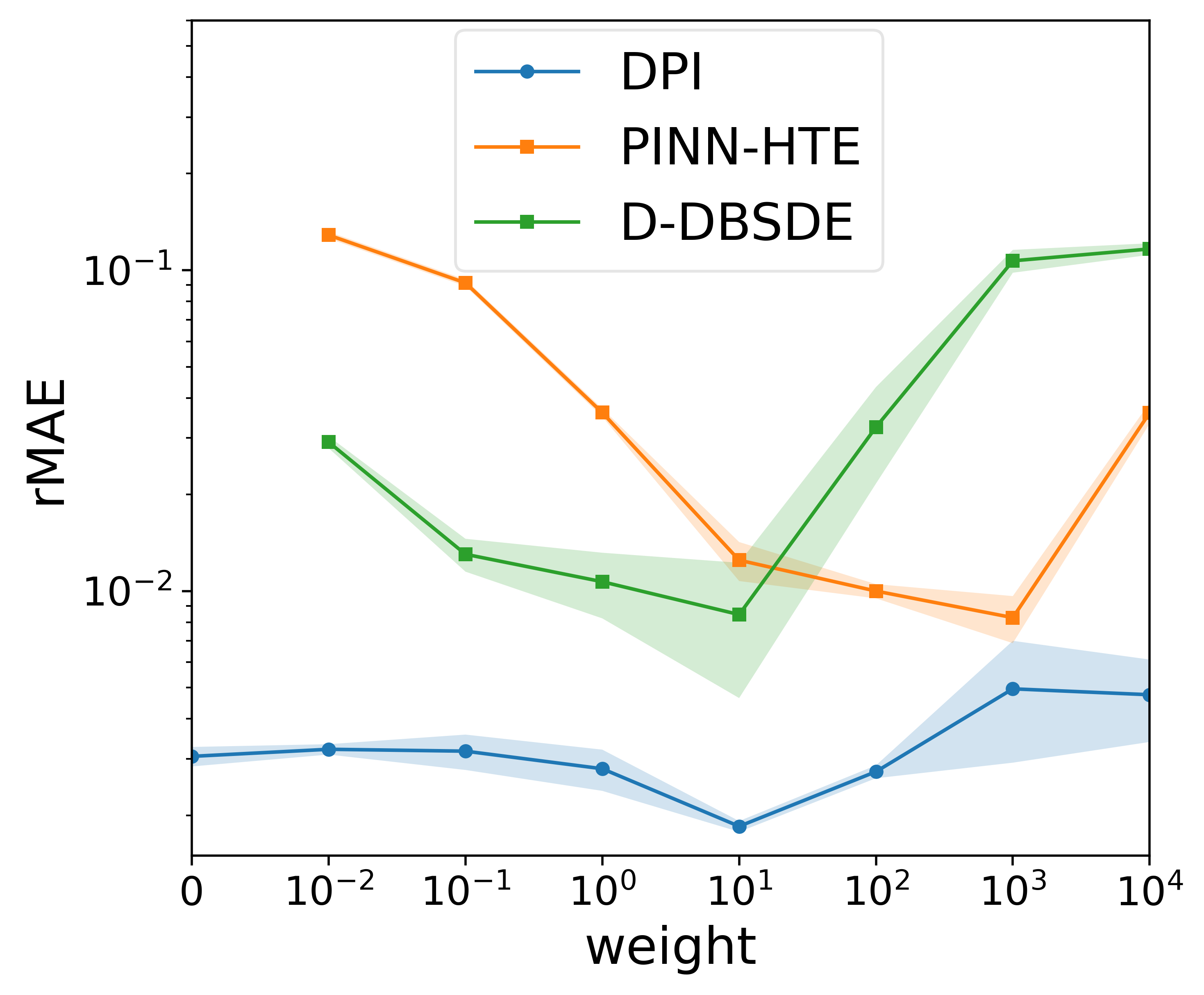}
    \end{subfigure}
    \hfill
    \begin{subfigure}[b]{0.49\textwidth}
        \centering
        \includegraphics[width=\textwidth]{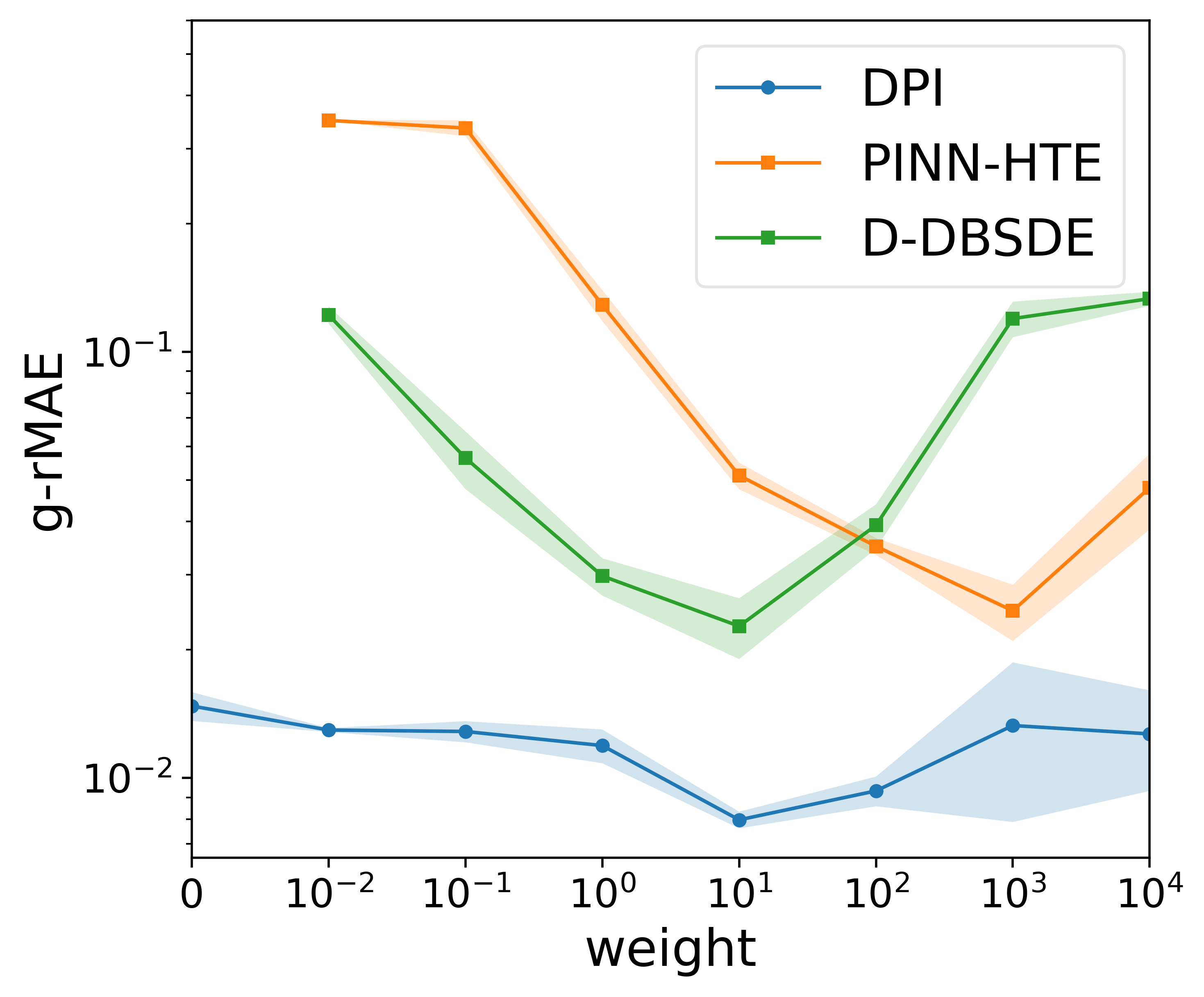}
    \end{subfigure}
\caption{Comparison of the relative errors for \( u \) and \( \nabla u \) among DPI, PINN-HTE and D-DBSDE with different weight hyperparameter $\lambda$ or $\lambda_T$ in loss in the Burgers-type PDE~\eqref{eq:cha} with $\kappa=2.5$. }
    \label{fig:weight}
\end{figure}

In Figure~\ref{fig:cha}, we summarize the optimal performance of each method after weight tuning for PDE~\eqref{eq:cha} with different $\kappa$. For $\kappa=1.0$, the problem is relatively simple, all methods perform well and DPI with gradient suipervision slightly outperforms the other methods. However, as \(\kappa\) increases to 5.0, indicating a more challenging problem, DPI substantially outperforms the other methods, showcasing superior robustness and efficacy. Moreover, DPI with gradient supervision consistently outperforms DPI without gradient supervision across various $\kappa$ values, demonstrating the benefit of incorporating gradients as additional labels.
It is noteworthy that for higher $\kappa$ values, PINN-HTE and D-DBSDE require significantly larger $\lambda_T$ to balance the loss and achieve optimal results. Conversely, DPI consistently exhibits stable and high-quality performance across different weights and problem parameters, demonstrating its potential for effectively and robustly addressing more complex problems.

\begin{figure}[!htb]
    \centering
    \begin{subfigure}[b]{0.49\textwidth}
        \centering
        \includegraphics[width=\textwidth]{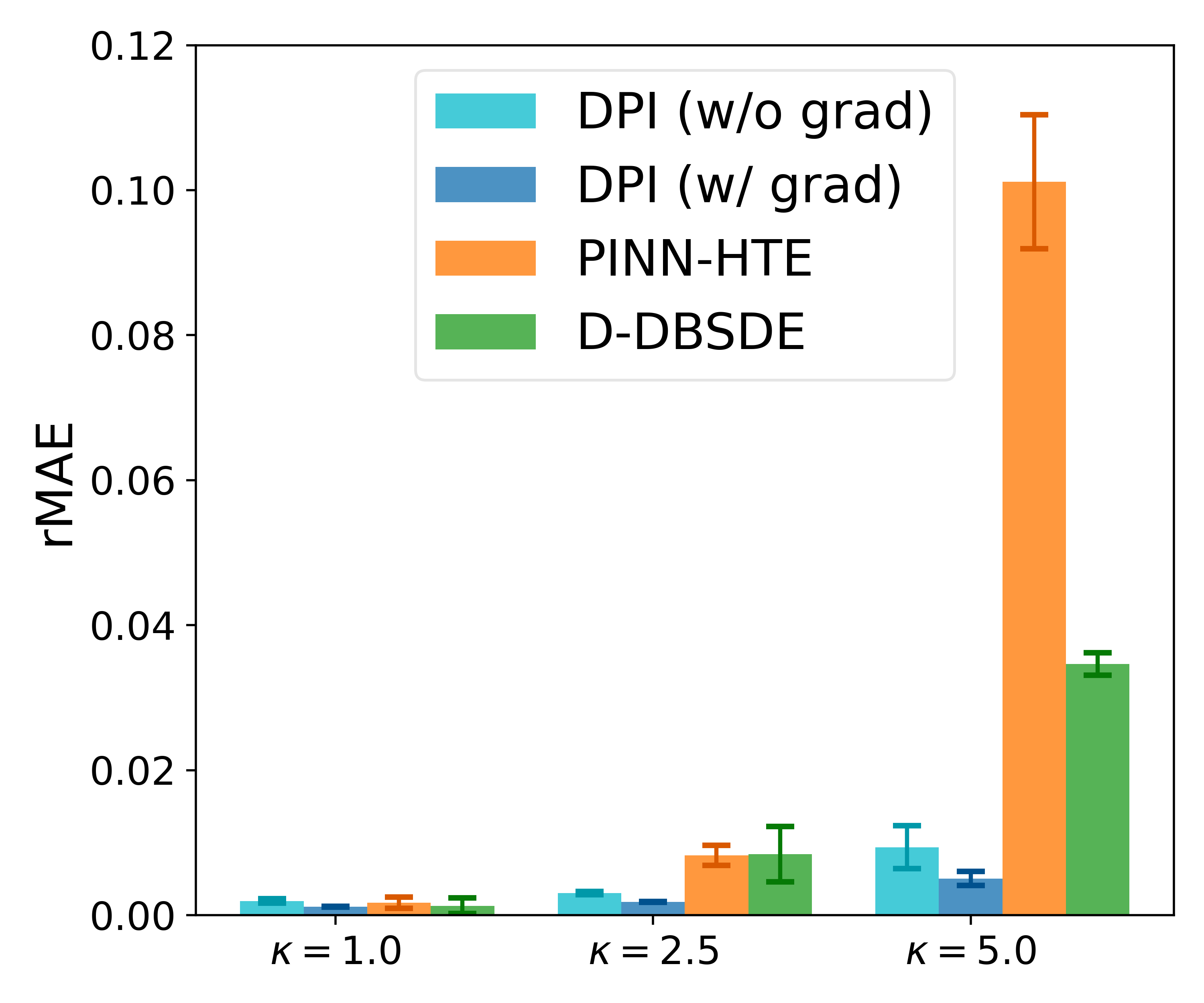}
    \end{subfigure}
    \hfill
    \begin{subfigure}[b]{0.49\textwidth}
        \centering
        \includegraphics[width=\textwidth]{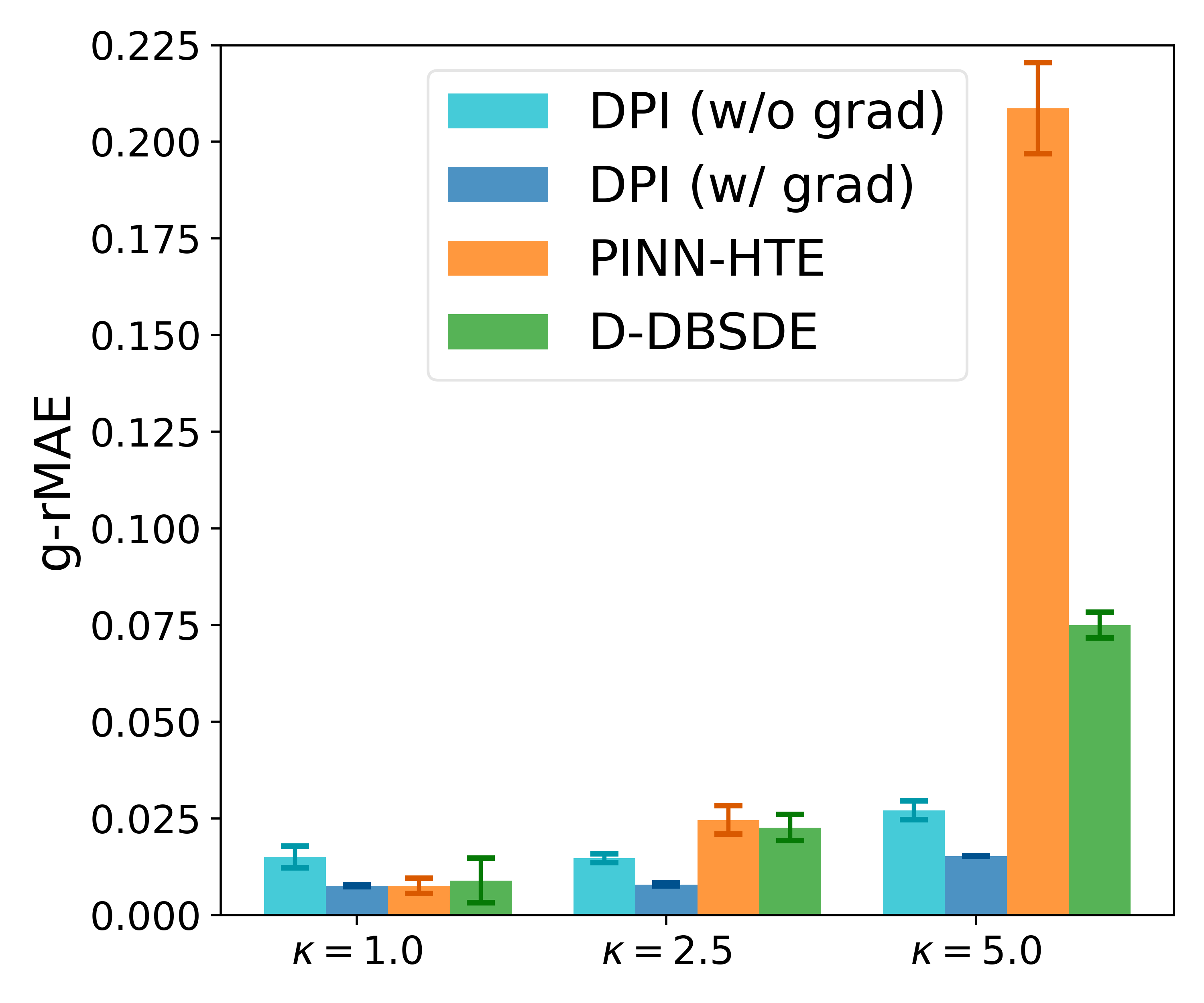}
    \end{subfigure}
\caption{Comparison of the relative errors for \( u \) and \( \nabla u \) among DPI, PINN-HTE, and D-DBSDE with different strength of nonlinearity (different $\kappa$) in the Burgers-type PDE~\eqref{eq:cha}.}
\label{fig:cha}
\end{figure}

We further evaluate the performance of DPI with varying hyperparameters for data generation ($M$ and $N$ in Algorithm~\ref{alg_dpi}) on the problem with $\kappa = 1.0$. As illustrated in Figure~\ref{fig:N_M} (left), we fix the number of samples used in the Monte Carlo approximation at each data point as $M = 4096$, the DPI weight as $\lambda = 1.0$ and the total iterations as $K = 20$. Then we vary the data size $N$ used in each Picard iteration step from 4096 to 131072. 
We observe that when $N$ is smaller, the results are less accurate compared to larger $N$, though they still provide sufficiently good solutions. In the right panel of Figure~\ref{fig:N_M}, we fix the number of data points used in each iteration at $N = 4096$ while varying the number of samples $M$ for the Monte Carlo approximation at each data point. As expected, increasing $M$ results in better outcomes and smaller variances, primarily due to the enhanced accuracy of the Monte Carlo approximation for generating labels.

\begin{figure}[!htb]
    \centering
    \begin{subfigure}[b]{0.49\textwidth}
        \centering
        \includegraphics[width=\textwidth]{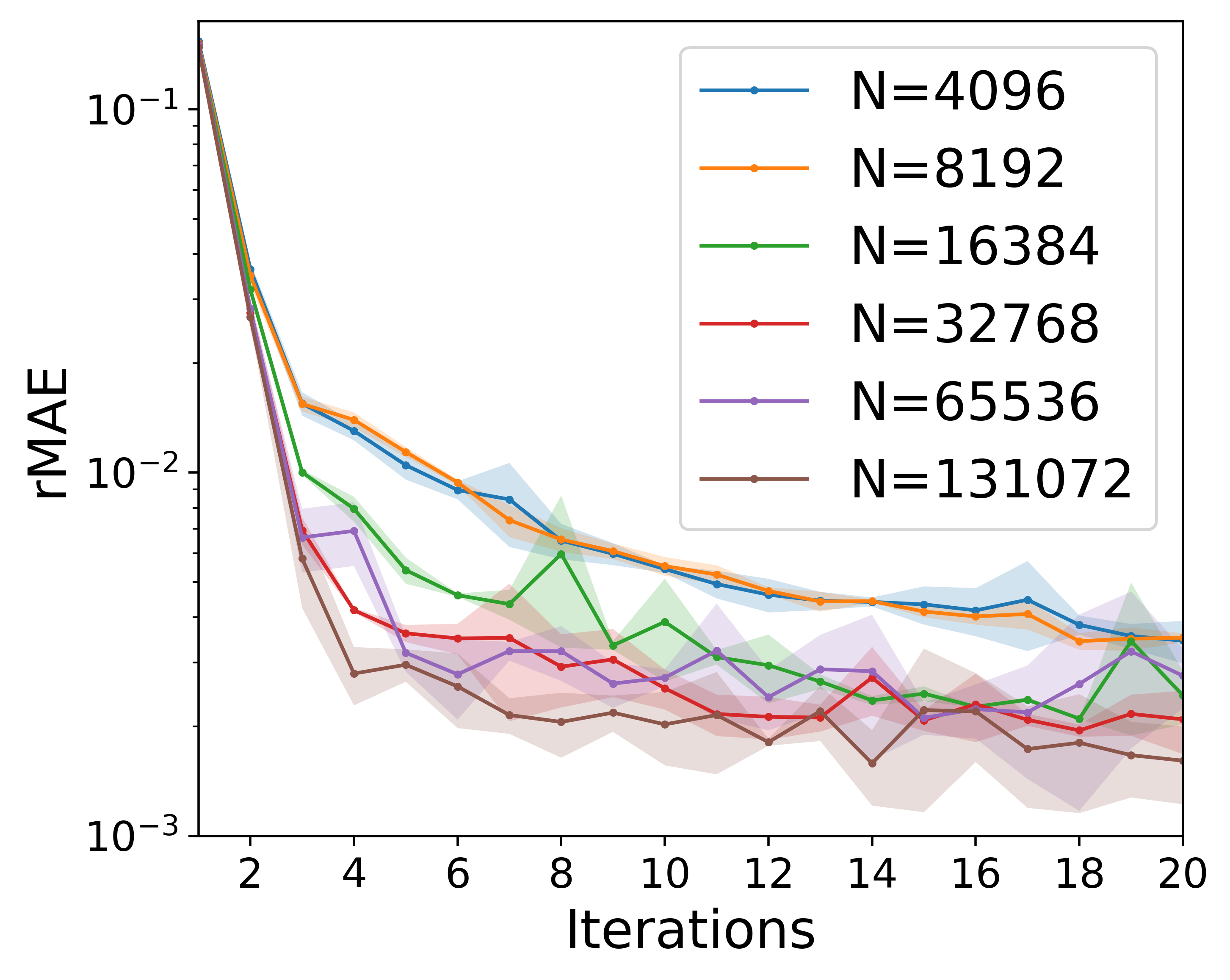}
        \label{fig:figure1}
    \end{subfigure}
    \hfill
    \begin{subfigure}[b]{0.49\textwidth}
        \centering
        \includegraphics[width=\textwidth]{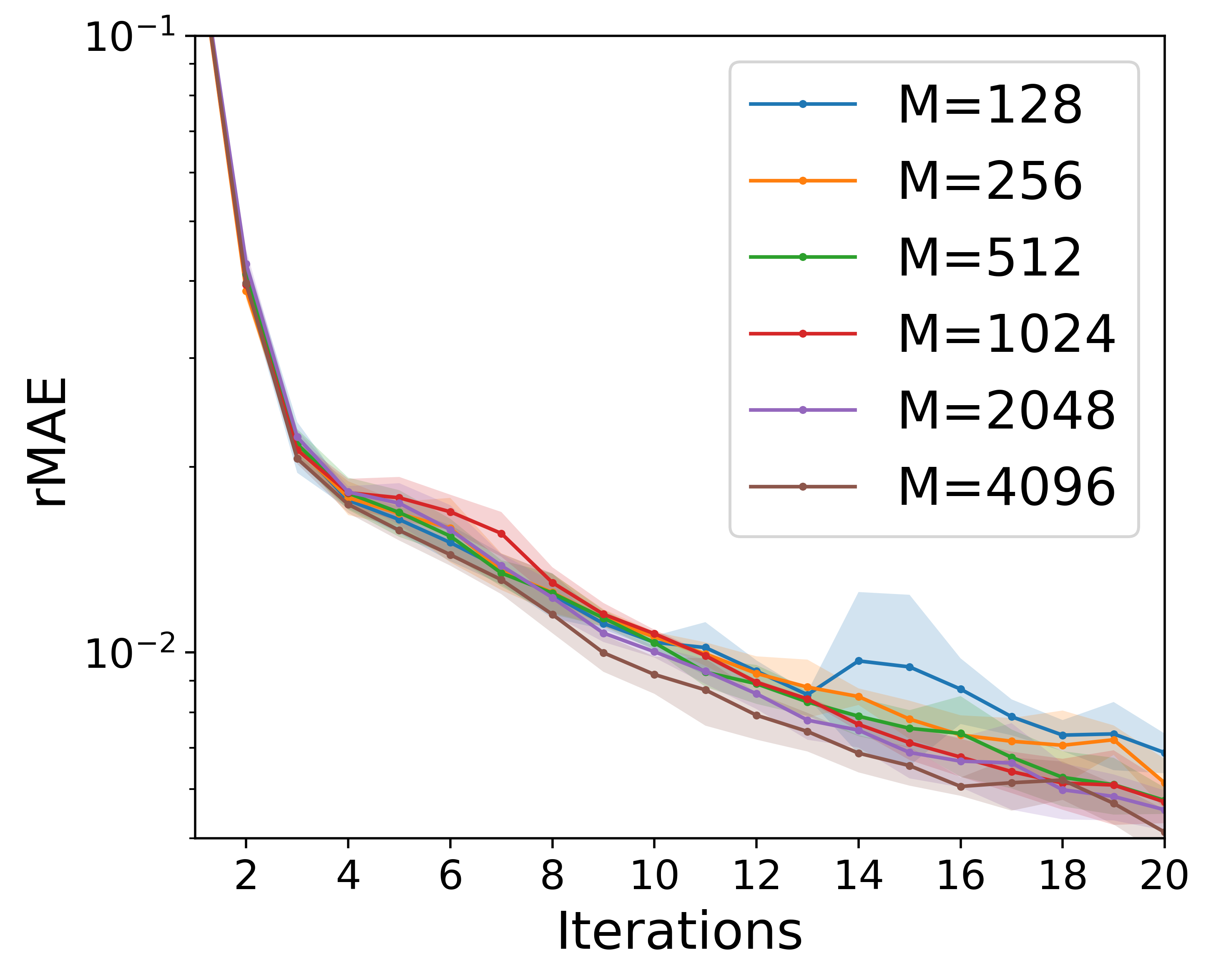}
        \label{fig:figure2}
    \end{subfigure}
    \caption{Relative error of $u$ in the Burgers-type PDE~\eqref{eq:cha} with varying data sizes $N$ and numbers of samples $M$ for Monte Carlo approximation at each data point in DPI. 
}
    \label{fig:N_M}
\end{figure}

In light of the above results, we provide guidance on selecting the key hyperparameters $K$, $M$, and $N$ to ensure that DPI yields a solution with the desired accuracy. This involves two key steps: first, selecting appropriate values for $M$ and $N$ to control the error within each Picard iteration step; and second, choosing the number of iterations $K$ to ensure overall accuracy. Specifically, at iteration step $k$, the targeted update $u_{k+1}$ is determined by the Picard iteration~\eqref{eq:Picard}, which is also equivalent to solving the linear PDE:
\begin{equation*}
    \begin{dcases}
        \partial_t u + \mu(t, x) \cdot \nabla_x u + \frac{1}{2}\mathrm{tr}(\sigma \sigma^\top(t, x) \nabla_{x}^2 u) + f_{u_k}(t, x) = 0, &\text{on } [0,T)\times\mathbb{R}^d,\\
        u(T,x) = g(x), &\text{on } \mathbb{R}^d.
    \end{dcases}
\end{equation*}
To assess the accuracy of each Picard step, one may monitor the DPI loss as a regression task using a small validation set, or evaluate the PINN loss~\eqref{loss_pinn} or DBSDE loss~\eqref{loss_bsde} associated with the linear PDE above. The values of $M$ and $N$ (as well as the network capacity, if necessary) can be gradually increased until the monitored loss falls slightly below the desired threshold. To estimate the overall error, the discrepancy between $u_{k+1}$ and the true solution $u$, one can either track the difference between consecutive iterates $u_{k+1}$ and $u_k$, or apply the PINN or DBSDE loss directly to the original nonlinear PDE~\eqref{eq:pde}. The iteration should continue until the estimated overall error saturates.

Finally, we remark on the computational cost associated with performing regression on the gradient.
When computing the gradient labels \( z \), the most time-consuming step is evaluating \( f_{u_k} \), which requires both the evaluation and automatic differentiation of the neural network \( u_k \). Nevertheless, since the computation of  $z$  involves evaluating  $f_{u_k}$  at the same points used for  $y$, these computations can be reused, significantly reducing the additional cost of computing  $z$.
To make a concrete example, the data generation times per Picard iteration in this example are 1.99s with and 2.28s without the calculation for \( z \), representing an increase of only 14.58\%. 
Additionally, supervising gradients increases the training time per Picard iteration from 1.35s to 1.92s by 42.22\%.

\subsection{A semilinear Hamilton-Jacobi-Bellman (HJB) equation}\label{sec:hjb}

The HJB equation is a fundamental PDE that arises in optimal control theory from dynamic programming principle, widely used across various fields such as finance, economics, and engineering. It plays a crucial role in determining the optimal strategy for controlling dynamic systems and thus is central to decision-making processes in complex, real-world systems. Recently, a specific HJB equation has also become pivotal in score-based generative modeling~\cite{song2021scorebased,bruna2024posterior,sun2024dynamical}, as explained below.

Consider a stochastic process following the Ornstein-Uhlenbeck (OU) process 
\begin{equation}
\rmd X_t = -X_t \rmd t+ \rmd W_t
\label{eq:OU_for_sampling}
\end{equation}
with $X_0 \sim \mu_0$. Assume $\mu_0$ has a density $p_0(x)$. Then the density of the distribution of $X_t$, $p(t,x)$, is governed by the Fokker-Planck equation
\begin{equation*}
\partial_t p = \nabla \cdot (xp) + \frac{1}{2} \Delta p.
\end{equation*}
With the transformation
\[u(t,x) = -\log p(T - t, x),\]
we derive the corresponding PDE of the HJB type:
\begin{equation}\label{eq:hjb}
    \partial_t u(t, x) + \frac{1}{2} \Delta u(t, x) + x^\top \nabla u(t, x) - \frac{1}{2} |\nabla u(t, x)|^2 - d = 0,
\end{equation}
with the terminal condition $g(x) = -\log p_0(x)$. If we can solve $u(t,x)$ from~\eqref{eq:hjb}, we can reverse the OU process~\eqref{eq:OU_for_sampling} in the distribution sense according to the reverse time formulation~\citep{anderson1982reverse, haussmann1986time}:
\begin{equation}\label{eq:reverse_sde}
\mathrm{d} \tilde{X}_t = \left(\tilde{X}_t - \nabla_x {u}(t, \tilde{X}_t)\right) \mathrm{d} t + \mathrm{d} \bar{W}_t, \quad \tilde{X}   _0 \sim p_T,
\end{equation}
such that $\tilde{X}_T$ has the density $p_0$.
Here $\tilde{W}_t$ is another independent Brownian motion, and $\nabla_x u(t,x)$ is usually known as the \textit{score}.
Note that, due to the exponential contraction property of the OU semigroup, \( p_T \) becomes close to the Gaussian distribution \(\mathcal{N}(0, \frac{1}{2} \II_d)\) given a sufficiently large \( T \), making it easy to sample from.
Therefore, solving the HJB equation~\eqref{eq:hjb} gives us a new approach to sample from the density $p_0$ (which may be high-dimensional and multimodal) by simulating~\eqref{eq:reverse_sde} from 0 to $T$. This method is quite different from traditional approaches like importance sampling or Markov chain Monte Carlo (MCMC) methods~\citep{kass1998markov}, which can easily struggle with multimodal distributions.

With this background, now we turn to solve the HJB equation~\eqref{eq:hjb} numerically with different methods. For this problem, we adopt a specialized network architecture inspired by \cite{pisgradnet,zhang2021path, salimans2021should}, commonly utilized to learn the scalar potential for sampling given its superior representational capabilities. We introduce two neural networks: a scalar-valued function \(r_\eta(t)\), implemented as a four-layer MLP with 64 neurons per layer for processing the time embedding, and a vector-valued function \(N_\gamma(t,x) \in \mathbb{R}^d\), implemented as a four-layer MLP with 512 neurons per layer. We then parameterize the solution \(u_\theta\) as
\begin{equation*}
u_\theta(T - t, x) = [r_\eta(t)-r_\eta(0)]\langle N_\gamma(t,x), x \rangle + [1 - r_\eta(t) + r_\eta(0)] g\left(e^{-\frac{t}{2}}x\right),
\end{equation*}
where \( \langle \cdot,\cdot \rangle \) denotes the standard inner product in \(\mathbb{R}^d\). We remark that this architecture explicitly enforces the terminal condition \(u_\theta(T, x) = g(x)\) within its design.

We set the target density $p_0(x)$ needed in the terminal condition of the PDE as the density of a Gaussian mixture model (GMM) in 100 dimensions with five components, with means $\mu^{(k)}_0$ uniformly sampled within \([-1, 1]\) in each dimension and a diagonal covariance matrix $\Sigma^{(k)}_0 = 2\II_d$, $k=1,\dots,5$. The weight $ w_k $ for each component is randomly initialized and then normalized. Under the OU process~\eqref{eq:OU_for_sampling}, we have 
 \[p(t, x) = \sum_{k=1}^5  w_k p(x;  {\mu}_t^{(k)}, {\Sigma}_t^{(k)}).\]
Here $p(x;\mu, \Sigma)$ denotes the density of a multivariate Gaussian distribution $\mathcal{N}(\mu, \Sigma)$.
The mean and covariance of each component at time $t$ are explicitly known as 
\[\mu_t^{(k)} = \mu_0^{(k)}  e^{- t},\quad \Sigma_t^{(k)} = \Sigma_0^{(k)} e^{-2 t} + \frac{1 - e^{-2 t}}{2}\II_d. \] 
According to our derivation above, the exact solution is $u^*(t, x) = -\log p(T-t, x)$.

We conduct experiments with different time horizons $T = 0.25, 0.5, 1.0$. For the forward SDE~\eqref{eq:forward_SDE_x0} used to define training data distribution, we set $\xi = \mathcal{N}(0, 4\II_d)$, $\mu \equiv 0, \sigma \equiv \II_d$. This choice ensures that the training data adequately covers the range of the OU process. 
Figure~\ref{fig:hjb_100d} shows the optimal results with tuned weights: DPI uses $\lambda = 100.0$ for $T = 0.25$ and $T=0.5$, and $\lambda = 10.0$ for $T=1.0$.
As shown in Figure~\ref{fig:hjb_100d}, DPI consistently outperforms PINN-HTE and D-DBSDE, with its advantage becoming more pronounced as the time horizon $T$ increases and the problem becomes more challenging. The performance of DPI with and without gradient supervision further highlights its robustness, particularly in tackling complex problems with longer time horizons.

\begin{figure}[!htb]
    \centering
    \begin{subfigure}[b]{0.49\textwidth}
        \centering
        \includegraphics[width=\textwidth]{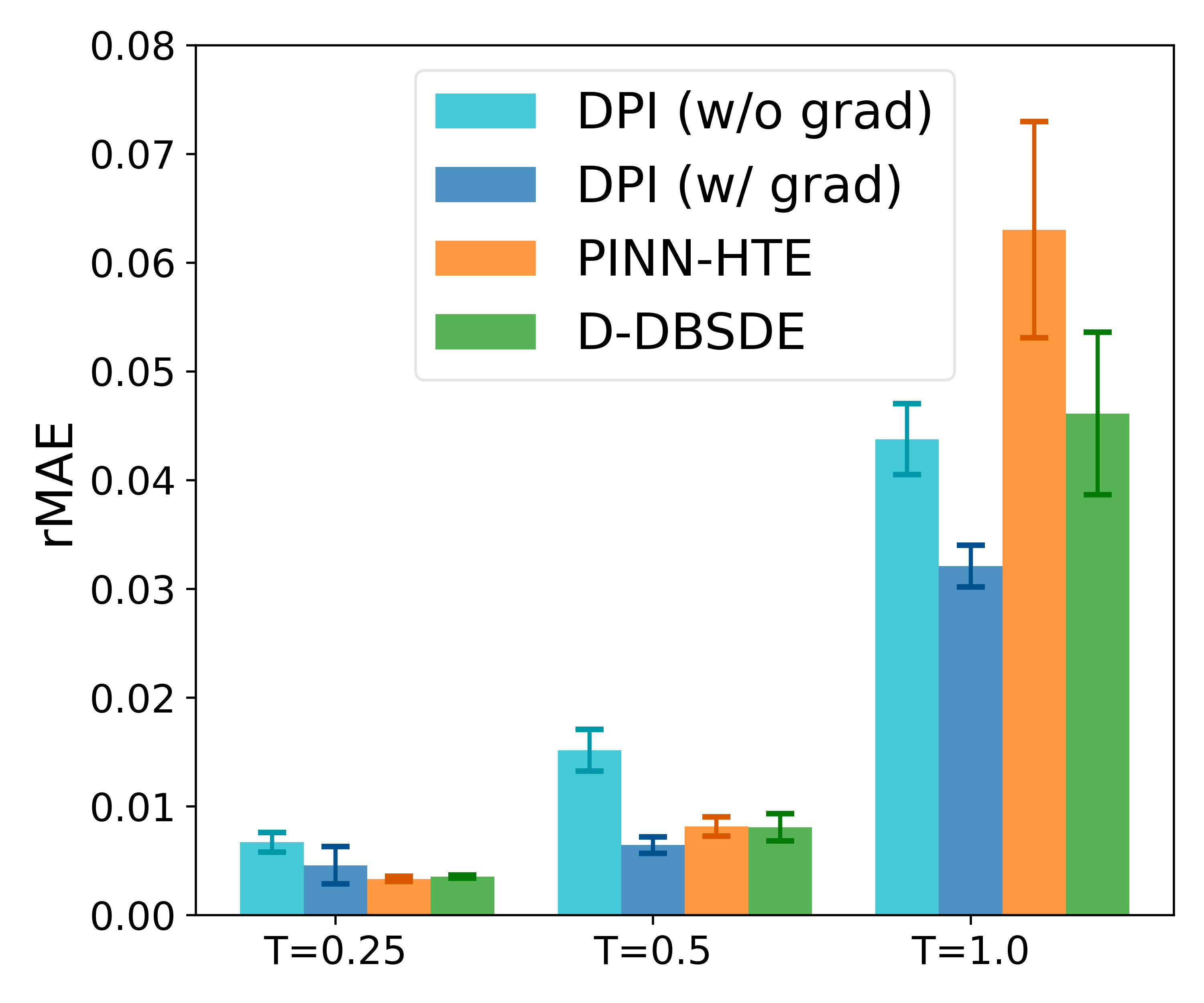}
    \end{subfigure}
    \hfill
    \begin{subfigure}[b]{0.49\textwidth}
        \centering
        \includegraphics[width=\textwidth]{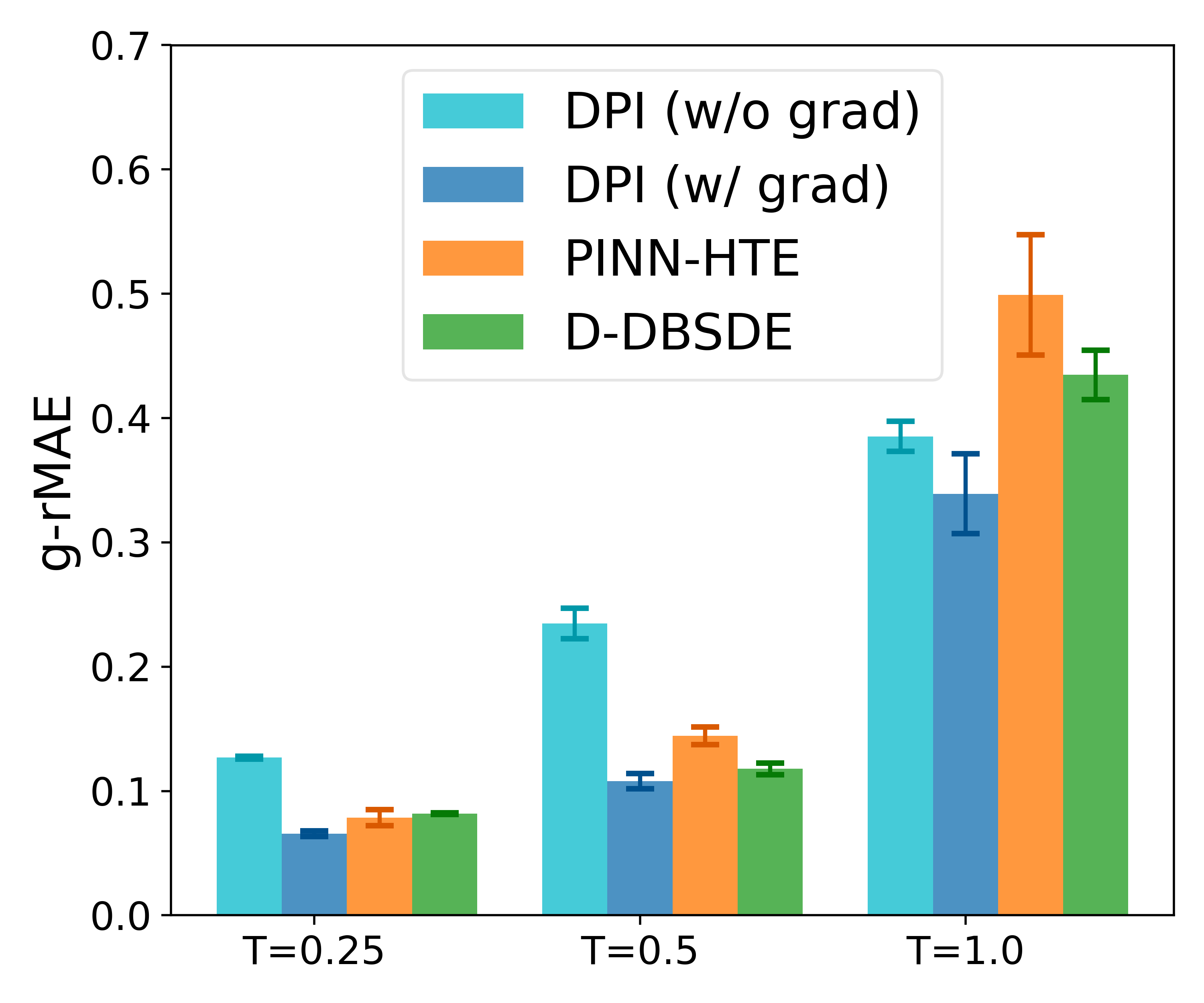}
    \end{subfigure}
\caption{Comparison of the relative errors for \( u \) and \( \nabla u \) among DPI, PINN-HTE, and D-DBSDE with different time horizons $T$ in the HJB equation~\eqref{eq:hjb}.}
    \label{fig:hjb_100d}
\end{figure}

We further validate the obtained solution by simulating the reverse SDE~\eqref{eq:reverse_sde} through the approximated score. As shown in Figure~\ref{fig:hjb_100d}, while DPI demonstrates superiority over the other two methods, the g-rMAE remains high, which hinders accurate sampling in 100 dimensions. Therefore, we use a 10-dimensional example instead for demonstration purposes.
To create a multimodal distribution that may challenge classical MCMC methods, we modify the target density $p_0(x)$ by selecting the means \(\mu_0^{(k)}\) to be more widely separated, uniformly sampled from \([-2, 2]\) instead of \([-1, 1]\) in each dimension, and by using a smaller covariance matrix \(\Sigma_0^{(k)} = \II_d\) instead of \(2\II_d\). We solve the corresponding HJB equation~\eqref{eq:hjb} with \( T = 0.25 \). We employ DPI with \(\lambda = 100.0 \), and initialize the sample distribution $\xi$ in~\eqref{eq:forward_SDE_x0} as \(\mathcal{N}(0, 2\II_d)\) to solve the problem. The final optimized network $\hat{u}(t,x)$ achieves an rMAE of 0.011 and a g-rMAE of 0.076. We then simulate the reverse SDE~\eqref{eq:reverse_sde} using the learned score $\nabla_x \hat{u}(t,x)$ and initiating the state $\tilde{X}_0$ according to the true density $p(T, x) $ to obtain final samples $\tilde{X}_T$. As shown in Figure~\ref{fig:hjb_10d}, the projected sample distribution from $\tilde{X}_T$ aligns well with the true distribution $p_0(x)$, demonstrating the effectiveness of our sampling procedure through solving the HJB~equation~\eqref{eq:hjb}. In future work, we plan to explore higher dimensions and longer time horizons to enhance the reliability of the sampling performance.

\begin{figure}[!htb]
    \centering
    \begin{subfigure}[b]{0.49\textwidth}
        \centering
        \includegraphics[width=\textwidth]{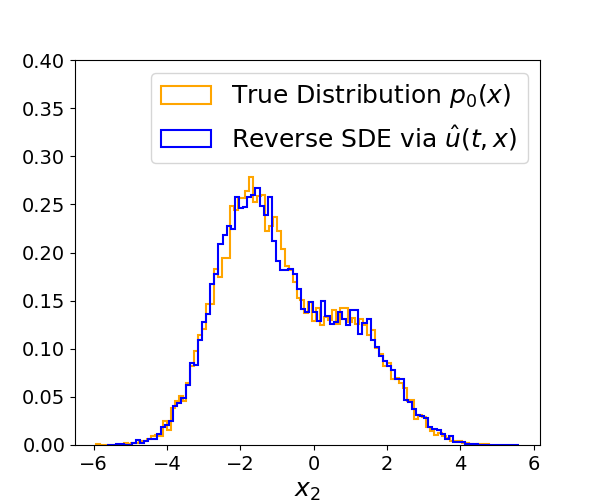}
    \end{subfigure}
    \hfill
    \begin{subfigure}[b]{0.49\textwidth}
        \centering        \includegraphics[width=\textwidth]{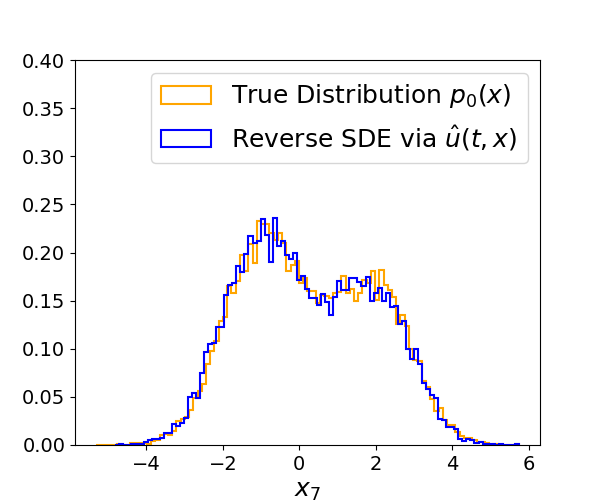}
    \end{subfigure}
\caption{Comparison of the projected sample distribution of the true distribution $p_0(x)$ and the distribution of $\tilde{X}_T$ obtained through reverse SDE~\eqref{eq:reverse_sde} via $\hat{u}(t,x)$ for a 10-dimensional Gaussian mixture density.}
    \label{fig:hjb_10d}
\end{figure}

\subsection{A fully nonlinear example}\label{sec:fully}
Finally we consider a fully nonlinear PDE modified from~\cite{beck2019machine}, which is related to $G$-Brownian motion \citep{peng2007g}
\begin{equation}\label{eq:gbm}
   \partial_t u(t, x) + \frac{1}{2} \Delta u(t, x) + \frac{1}{4} \sum_{i=1}^d \left|   \frac{\partial^2 u}{\partial x_i^2}(t, x) \right| - h(t, x) = 0.
\end{equation}
We construct the exact solution as a two-layer neural network with 
\begin{align*}
    u^*(t, x) &=   \sum_{j=1}^J v_j \sin \left( t + \sum_{i=1}^d w_i^j x_i \right), 
\end{align*}
and $h$ is set to satisfy the PDE \eqref{eq:gbm}
\[
h(t,x) = \partial_t u^*(t, x) + \frac{1}{2} \Delta u^*(t, x) + \frac{1}{4} \sum_{i=1}^d \left|   \frac{\partial^2 u^*}{\partial x_i^2}(t, x) \right|.
\] 
The parameters are sampled from $w_i^j \sim \frac{1}{\sqrt{d}}\mathcal{N}(0, 1)$, $v_j \sim \mathcal{N}(0, 1)$. We set $J = 2$ and randomized three groups of parameters for the exact solution, each serving as a different model to solve. The horizon is $T=1.0$.

In this problem, we use the original PINN rather than PINN-HTE since we need to compute all diagonal components of the Hessian matrix in the nonlinearity term. We also compare our method to DBDP \citep{pham2021neural}, which is designed to solve fully nonlinear problems. For DPI with gradient supervision, we use $\lambda=100.0$. For DBDP, we choose $\delta t = 0.02$ and set the number of gradient descent steps per sub-iteration to 150, ensuring that the computational cost aligns with PINN and DPI. The hyperparameter $\delta t$ has been tuned to achieve optimal performance within the given computational budget.
As shown in Figure~\ref{fig:gbm}, DPI with gradient supervision outperforms the other tested methods for above problems. The improvement of DPI brought by gradient supervision highlights the importance of gradient supervision in handling problems with higher order nonlinearity. It is also worth noting that fully nonlinear problems place greater demands on GPU memory during sampling than semilinear problems. By leveraging additional GPUs for parallel sampling, we anticipate a significant reduction in the time required for DPI sampling, which could lead to faster and more accurate results.

\begin{figure}[!htb]
    \centering
    \begin{subfigure}[b]{0.49\textwidth}
        \centering
        \includegraphics[width=\textwidth]{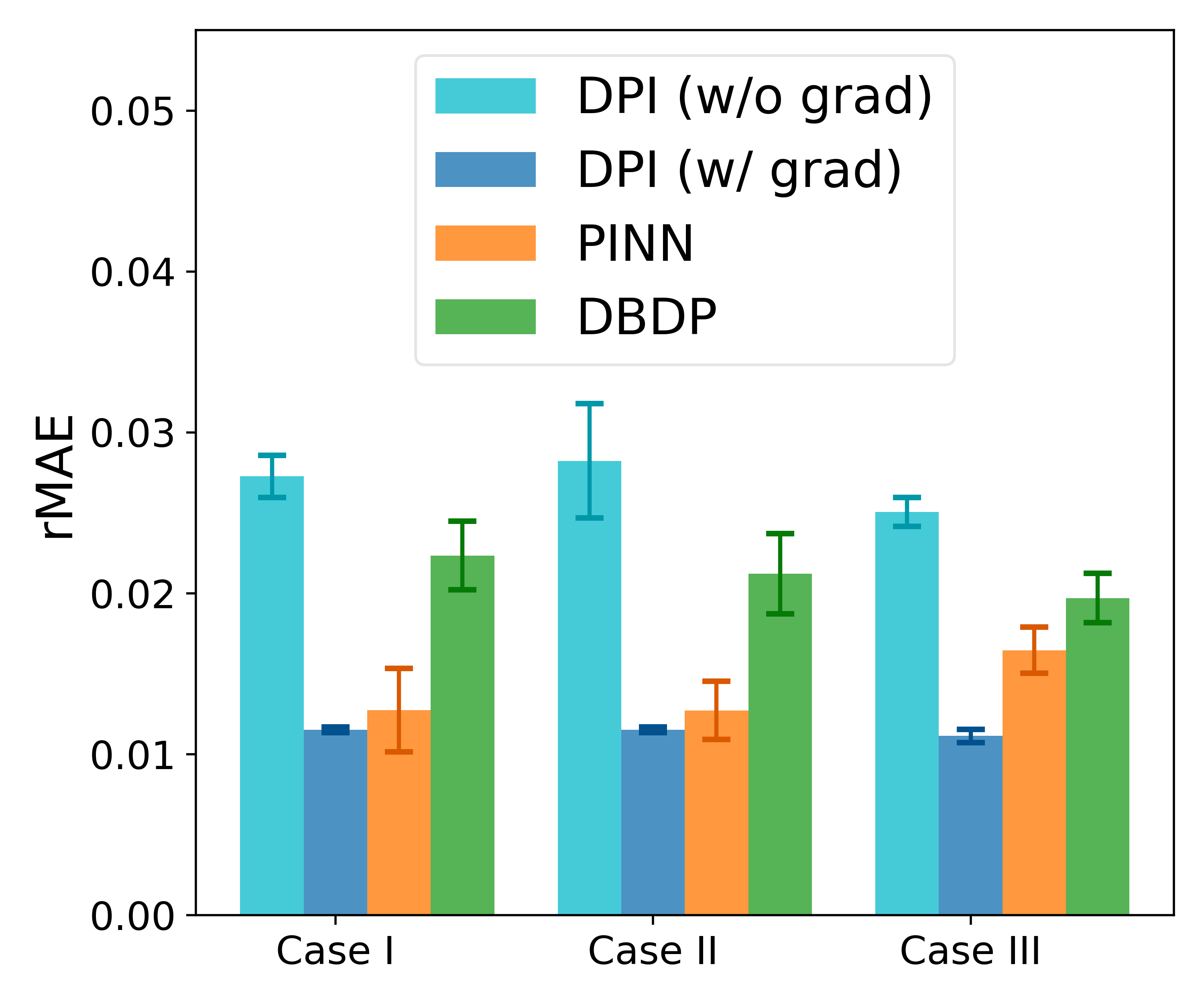}
    \end{subfigure}
    \hfill
    \begin{subfigure}[b]{0.49\textwidth}
        \centering
        \includegraphics[width=\textwidth]{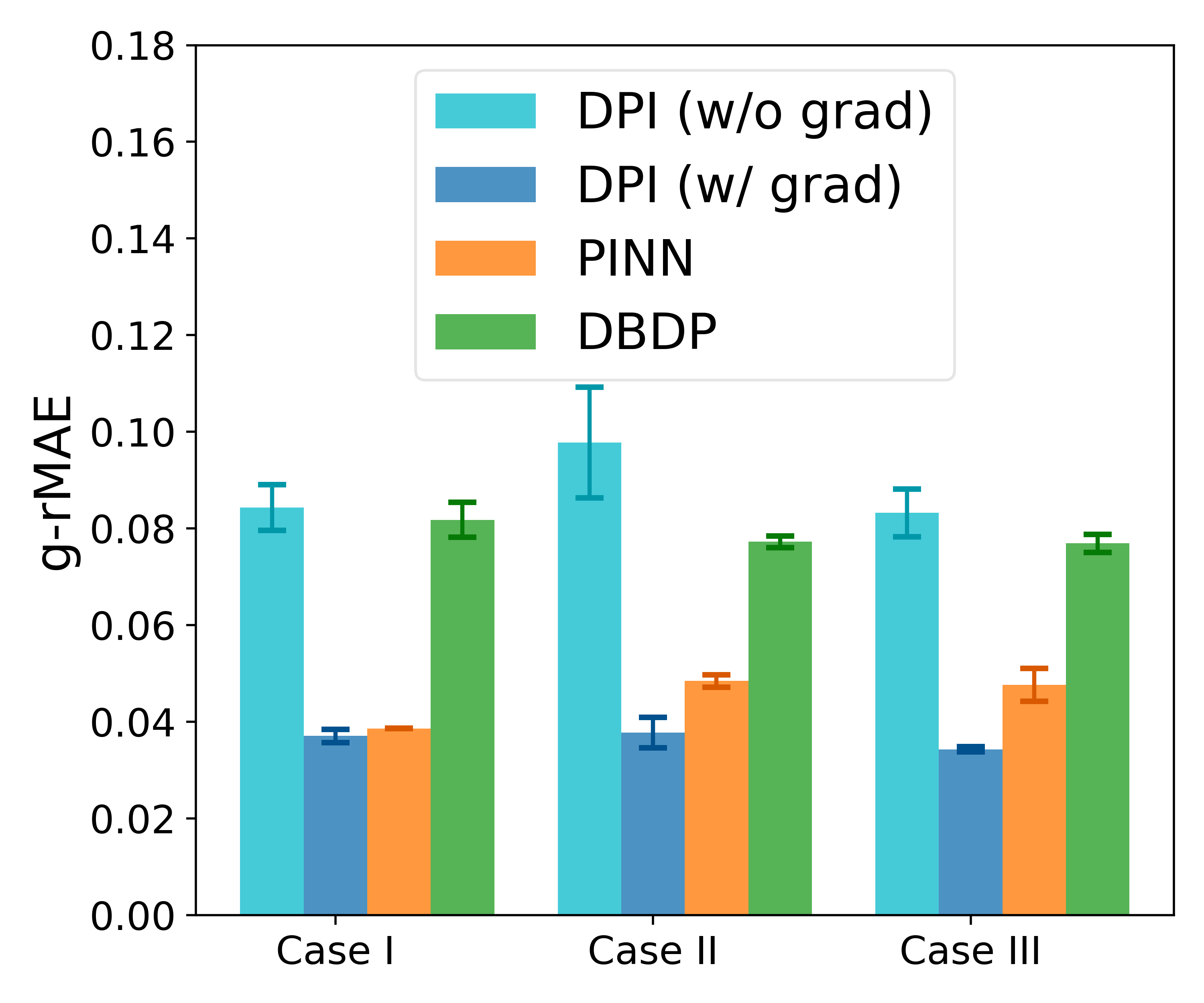}
    \end{subfigure}
\caption{Comparison of the relative errors for \( u \) and \( \nabla u \) among DPI, PINN-HTE, and DBDP in the fully nonlinear problem~\eqref{eq:gbm} with the exact solution randomized differently in three cases.}
    \label{fig:gbm}
\end{figure}

\section{Conclusion}
In this study, we introduce the Deep Picard iteration (DPI) method, a novel deep learning approach for solving high-dimensional semilinear and fully nonlinear PDEs. The method utilizes Picard iteration to transform the optimization challenges of neural network-based PDE solutions as standard regression tasks involving function values and gradients. Our experimental results demonstrate that DPI is robust across various parameter settings, consistently achieving superior performance compared to other state-of-the-art methods.

Future work will focus on several key aspects to further enhance the effectiveness of DPI. We plan to explore parallel data generation techniques to accelerate the method, making DPI scalable for even larger and more complex problems. Additionally, we intend to systematically study the impact of the drift \( \mu \) and diffusion \( \sigma \)  in training data generation \eqref{eq:forward_SDE_x0} on the final solution’s accuracy. 
Moreover, it is observed that the loss functions in other methods, such as PINNs and Deep BSDEs, can be recasted into a regression form by freezing certain parameters in the loss function with an additional fixed-point iteration, similar to the approach used in DPI. Investigating the performance of these methods under such modifications would be of interest. Finally, we are interested in extending the current approach to problems with spatial boundaries to broaden its applicability.

\bibliographystyle{plain}
\bibliography{ref}
\newpage
\appendix
\section{DPI with Variance-Reduced Hessian Estimator}\label{appendix: Hessian}

In this section, we introduce a Hessian estimator enhanced by variance reduction. The discussion proceeds in three parts. First, we present the basic form of the Hessian estimator from Theorem~2.3 in~\cite{ELWORTHY1994252}, which exhibits infinite variance. Second, we apply a variance reduction technique to this basic estimator, yielding a refined Hessian estimator with finite variance. Our approach is inspired by~\cite{germain2021neural}, which addresses a similar challenge in discrete time. Finally, we present numerical results demonstrating the performance of the DPI method combined with the variance-reduced Hessian estimator on the fully nonlinear example in Section~\ref{sec:fully}.

\subsection{Basic form of Hessian estimator}
By applying equation~(15) in~\cite{ELWORTHY1994252}, we have
\begin{equation*}
\label{eq:basic_Hess_estimator}
    \nabla_x^2 u(t,x) = \EE\left[g(X_T^{t,x})K_T^{t,x}\right] + \int_t^T \EE\left[f_u(s,X_s^{t,x}) K_s^{t,x}\right]\rmd s.
\end{equation*}
Here $K_s^{t,x}$ is another stochastic process, whose explicit form for the general case is given by equation~(15) in~\cite{ELWORTHY1994252}. For notational simplicity, here we only present explicit forms for the three types of processes considered in the main text: Brownian motion, geometric Brownian motion, and the Ornstein–Uhlenbeck process. In each case, we further consider
\begin{equation}\label{eq:Hessian-estimator}
    \tilde{K}_s^{t,x} \coloneqq \EE\left[K_s^{t,x}\mid X_s^{t,x}\right].
\end{equation}
Following the same argument used in the sample generation discussion of Section~\ref{sec:alg}, which leverages the property of conditional expectation, we obtain
\begin{equation*}
    \nabla_x^2 u(t,x) = \EE\left[g(X_T^{t,x})\tilde{K}_T^{t,x}\right] + \int_t^T \EE\left[f_u(s,X_s^{t,x}) \tilde{K}_s^{t,x}\right]\rmd s,
\end{equation*}
which yields an estimator with smaller variance.

\begin{enumerate}
    \item Brownian motion ($\mu \equiv 0$ and $\sigma \equiv \II_d$): 
    $$
    X_s^{t,x} = x + (W_s - W_t),\, K_s^{t,x} = \frac{4}{(s-t)^2}(W_s - W_{\frac{s+t}{2}})(W_\frac{s+t}{2} - W_t)\transpose.
    $$
   Define \(A_1=W_s-W_{\frac{s+t}{2}}\) and \(A_2=W_{\frac{s+t}{2}}-W_t\). Thus, we have \(X_s^{t,x}-x=A_1+A_2=W_s-W_t\).

Further define \(B_1=A_1+A_2\) and \(B_2=A_1-A_2\), giving \(A_1=\frac{B_1+B_2}{2}\), \(A_2=\frac{B_1-B_2}{2}\). By definition, we have
\begin{equation}
     \tilde{K}_s^{t,x} = \EE[K_s^{t,x} \mid X_s^{t,x}] = \frac{4}{(s-t)^2} \EE[A_1A_2\transpose \mid B_1].
\end{equation}
Since $A_1$ and $A_2$ are independent Gaussian vectors with \(A_1,A_2\sim N(0,\frac{s-t}{2}\II_d)\), then $B_1$ and $B_2$ are also are independent Gaussian vectors with mean zero and $\EE [B_2 B_2\transpose] = (s-t) \II_d$. Therefore,

\begin{align*}
    \EE[A_1A_2\transpose \mid B_1]&=\frac14 \EE[B_1B_1\transpose + B_2 B_1\transpose - B_1B_2\transpose-B_2B_2\transpose \mid B_1]\\& = \frac{1}{4}(B_1B_1\transpose-(s-t)\II_d),
\end{align*}
which gives
\begin{align*}
    \tilde{K}_s^{t,x} =\frac{1}{(s-t)^2}[(W_s-W_t)(W_s-W_t)\transpose-(s-t)\II_d].
\end{align*}
 \item Geometric Brownian motion ($\mu \equiv 0$ and $\sigma = \diag(x)$):
 \begin{align*}
     X_s^{t,x} &= \diag(\exp(-\frac{1}{2}(s-t) + W_s - W_t))x \\
    K_s^{t,x} &= \frac{4}{(s-t)^2}x^{-1}(W_s - W_{\frac{s+t}{2}})(W_\frac{s+t}{2} - W_t)\transpose x^{-1} \\
    &\quad - \frac{2}{s-t}x^{-2}\diag(W_\frac{s+t}{2} - W_t)
 \end{align*}
 where $x^{-1} = \diag(x_1^{-1},\dots,x_d^{-1})$ and $x^{-2} = \diag(x_1^{-2},\dots,x_d^{-2})$. Similar to the case of Brownian motion, we have 
 \begin{align*}
     \tilde{K}_s^{t,x} &= \frac{1}{(s-t)^2}x^{-1}[(W_s - W_t)(W_s - W_t)\transpose - (s-t)\II_d]x^{-1} \\
     &\quad - \frac{1}{s-t} x^{-2}\diag(W_s - W_t).
 \end{align*} 
 \item  Ornstein–Uhlenbeck process ($\mu = -\theta x$ and $\sigma \equiv \II_d$):
   
    \begin{align*}
        &X_s^{t,x} =  e^{-\theta (s-t)}x +  \int_{t}^s e^{\theta(r-s)}\rmd W_r, \\
        & K_s^{t,x} = \frac{4}{(s-t)^2}\Big(\int_t^{\frac{s+t}{2}} e^{-\theta(r-t)}\rmd W_r\Big)\Big(\int_{\frac{s+t}{2}}^s e^{-\theta(r-t)}\rmd W_r\Big)\transpose.
    \end{align*} 
    Define 
\begin{align*}
    &A_1=\int_t^{\frac{s+t}{2}} e^{-\theta(r-t)}\,\rmd W_r,\quad A_2=\int_{\frac{s+t}{2}}^s e^{-\theta(r-t)}\,\rmd W_r,\\
     &B_1=A_1+A_2,\, B_2=A_1-A_2,\, \widetilde{B} = \int_t^s e^{-\theta(s-r)}\,\rmd W_r.
\end{align*}

Then, 
\[
K_s^{t,x}=\frac{4}{(s-t)^2}A_1A_2\transpose=\frac{1}{(s-t)^2}(B_1B_1\transpose + B_2 B_1\transpose - B_1B_2\transpose-B_2B_2\transpose).
\]
Note that $B_1$, $B_2$, and $\widetilde{B}$ are mean-zero joint Gaussian vectors. Moreover, through It\^{o} isometry, we have
\begin{align*}
    &\EE[B_1 B_1\transpose] = \EE[B_2B_2\transpose] = \EE[\widetilde{B}\widetilde{B}\transpose] = v \II_d,\, v \coloneqq\frac{1-e^{-2\theta(s-t)}}{2\theta}, \\
    &\EE[B_1\widetilde{B}\transpose] = \int_{t}^s e^{-\theta(s-t)} \II_d\rmd r = e^{-\theta(s-t)}(s-t)\II_d, \\
    &\EE[B_2\widetilde{B}\transpose] = \int_{t}^{\frac{s+t}{2}} e^{-\theta(s-t)}\II _d\rmd r  - \int_{\frac{s+t}{2}}^s e^{-\theta(s-t)}\II_d\rmd r = 0.
\end{align*}
Since $\EE[B_2\widetilde{B}\transpose] = 0$, we know that $B_2$ and $\widetilde{B}$ are independent. Hence,
\[
\mathbb{E}[B_2B_2\transpose \mid \widetilde{B}]= \EE[B_2B_2\transpose] = v\II_d.
\]

Furthermore, by standard Gaussian conditioning we have
\[
\mathbb{E}[B_1\,|\,\widetilde{B}]=\frac{e^{-\theta(s-t)}(s-t)}{v}\,\widetilde{B},\quad
\operatorname{Var}(B_1\,|\,\widetilde{B})=\Bigl(v-\frac{e^{-2\theta(s-t)}(s-t)^2}{v}\Bigr)\II_d,
\]
so that
\[
\mathbb{E}[B_1B_1\transpose \mid \widetilde{B}]=\frac{e^{-2\theta(s-t)}(s-t)^2}{v^2}\,\widetilde{B}\widetilde{B}\transpose+\Bigl(v-\frac{e^{-2\theta(s-t)}(s-t)^2}{v}\Bigr)\II_d.
\]
Meanwhile, we also have
\[
\mathbb{E}[B_2 B_1\transpose - B_1 B_2\transpose\,|\,\widetilde{B}] = 0.
\]
Indeed, for any \(1\le i,j\le d\), \(i\neq j\), we have
\(\mathbb{E}[B_1^i B_2^j \mid \widetilde{B}]=0\), since \(B_2^j\) is independent of the pair \((B_1^j,\widetilde{B})\) and thus
\[
\mathbb{E}[B_1^i B_2^j \mid \widetilde{B}] = \mathbb{E}[B_2^j]\mathbb{E}[B_1^i \mid \widetilde{B}] = 0.
\]

Finally, we conclude
\begin{align*}
X_s^{t,x} &=  e^{-\theta (s-t)}x + \widetilde{B}, \\
    \tilde{K}_s^{t,x} &= \mathbb{E}[K_s^{t,x}\,|\,X_s^{t,x}]
=\frac{1}{(s-t)^2}\mathbb{E}[(B_1B_1\transpose\, - B_2 B_2\transpose)|\,\widetilde{B}] \\
&=\frac{e^{-2\theta(s-t)}}{v^2}\,\widetilde{B}\widetilde{B}\transpose-\frac{e^{-2\theta(s-t)}}{v}\II_d.
\end{align*}

with 
\[
\widetilde{B} = \int_t^s e^{-\theta(s-r)}\,\rmd W_r,\quad
v=\frac{1-e^{-2\theta(s-t)}}{2\theta}.
\]
It is worth mentioning that $\widetilde{B}$ is a Gaussian random vector with zero mean and covariance matrix $v\II_d$ and we can directly sample it.
\end{enumerate}

Returning to the basic form of the Hessian estimator~\eqref{eq:basic_Hess_estimator}, and noting that $\EE[K_s^{t,x}] = 0$, we can also apply the control variate version of equation~\eqref{eq:Hessian-estimator} to obtain
\begin{equation*}
    \nabla_x^2 u (t,x) = \EE[(g(X_T^{t,x}) - g(x))K_T^{t,x}] + \int_t^T \EE[(f_u(s,X_s^{t,x}) - f_u(t,x))K_T^{t,x}]\rmd s.
\end{equation*}
However, unlike the case of the gradient estimator, this estimator still suffers from infinite variance.
\begin{theorem}\label{thm:infinite-variance-Hessian}
     Assume Assumption~\ref{assump1} holds. Furthermore, $\nabla_x^2 \mu$ and $\nabla_x^2 \sigma$ are bounded continuous. Given a fixed $t \in [0,T)$ and $x \in \mathbb{R}^d$, assume that $g(x) \in C^2(\RR^d)$ with $\nabla_x g(x) \neq 0$, and $f(t,x) \in C^2([0,T] \times \RR^d)$ with $\nabla_x f(t,x) \neq 0$, where both functions have bounded second-order derivatives. We have\footnote{In the following, the Frobenius norm of a matrix will be denoted by $|\cdot|$.}
     \begin{align*}
    \lim_{s \rightarrow T^-}\EE|(g(X_T^{t,x}) - g(x)) K_T^{t,x}|^2 = +\infty,\, 
    \int_{t}^T \EE |(f(s,X_s^{t,x}) - f(t,x)) K_s^{t,x}|^2\rmd s = +\infty.
    \end{align*}
\end{theorem}
The proof will be presented in the next subsection once we introduce the Hessian estimator with finite variance.
Consider $\hat{X}_s^{t,x}$ as the solution to the modified SDE:
\begin{equation*}
    \hat{X}_s^{t,x} = x + \int_{t}^s \mu(r, \hat{X}_r^{t,x}) \rmd r - \int_{t}^s \sigma(r, \hat{X}_r^{t,x}) \rmd W_r, \quad s \in [t,T].
\end{equation*}
This formulation substitutes the Brownian motion $\{W_s\}$ in the SDE \eqref{eq:forward_SDE_general} with its negative counterpart $\{-W_s\}$. Likewise, we define $\hat{K}_s^{t,x}$ as the random variable where the Brownian motion in $K_s^{t,x}$ is replaced by $\{-W_s\}$. Given that $\{-W_s\}$ shares the same distribution as $\{W_s\}$, it follows that the pair $(\hat{X}_s^{t,x}, \hat{K}_s^{t,x})$ and $(X_s^{t,x}, K_s^{t,x})$ have identical distributions. Hence, we have
\begin{align*}
     \nabla_x^2 u(t,x) &= \EE\left[(g(X_T^{t,x}) - g(x))K_T^{t,x}\right] + \int_t^T \EE\left[(f_u(s,X_s^{t,x}) - f(t,x)) K_s^{t,x}\right]\rmd s \\
     &= \EE\left[(g(\hat{X}_T^{t,x}) - g(x))\hat{K}_T^{t,x}\right] + \int_t^T \EE\left[(f_u(s,\hat{X}_s^{t,x}) - f(x)) \hat{K}_s^{t,x}\right]\rmd s \\
     &= \frac{1}{2}\,\EE\left[(g(X_T^{t,x}) - g(x))K_T^{t,x} + (g(\hat{X}_T^{t,x}) - g(x))\hat{K}_T^{t,x}\right] \\
     &\quad + \frac{1}{2}\int_t^T \EE\left[(f_u(s,X_s^{t,x}) - f(t,x)) K_s^{t,x} + (f_u(s,\hat{X}_s^{t,x}) - f(t,x)) \hat{K}_s^{t,x}\right]\rmd s.
\end{align*}

\subsection{Variance-reduced Hessian estimator}
\begin{theorem}\label{thm:finite-variance-Hessian}
     Assume Assumption~\ref{assump1} holds. Furthermore, $\nabla_x^2 \mu$ and $\nabla_x^2 \sigma$ are bounded continuous. Given a fixed $t \in [0,T)$ and $x \in \mathbb{R}^d$, assume that $g(x) \in C^2(\RR^d)$, and $f(t,x) \in C^2([0,T] \times \RR^d)$, where both functions have bounded second-order derivatives. We have
     \begin{align*}
    &\sup_{s \in [t,T)}\EE|(g(X_T^{t,x}) - g(x))K_T^{t,x} + (g(\hat{X}_T^{t,x}) - g(x))\hat{K}_T^{t,x}|^2 < +\infty,\, \\
    &\int_{t}^T \EE |(f_u(s,X_s^{t,x}) - f(t,x)) K_s^{t,x} + (f_u(s,\hat{X}_s^{t,x}) - f(t,x)) \hat{K}_s^{t,x}|^2\rmd s < +\infty.
    \end{align*}
\end{theorem}
\begin{proof}
   We focus here on proving only the case of Brownian motion($\mu \equiv 0$ and $\sigma \equiv \II_d$), omitting a detailed proof of the general case. The reason is that the general proof, while sharing the same essential ideas, is more complicated and involves numerous computations regarding the order of various terms, which
   might obscure the core intuition for the reader. Nevertheless, we will briefly point out key differences between the Brownian case and the general setting whenever they arise in the proof.

   Recalling that we use $C$ as a positive constant independent of $t$, $s$ and $x$, We begin by demonstrating
   \begin{equation}\label{eq:Hessian_proof_eq_1}
       \EE|(g(X_T^{t,x}) - g(x) - \nabla_x g(x)\cdot(X_T^{t,x}-x))K_T^{t,x}|^2 \le C.
   \end{equation}
   First, the Cauchy-Schawarz inequality gives us
   \begin{align*}
       &\EE|(g(X_T^{t,x}) - g(x) - \nabla_x g(x)(X_T^{t,x}-x))K_T^{t,x}|^2 \\
       \le &\Big(\EE|g(X_T^{t,x}) - g(x) - \nabla_x g(x)\cdot(X_T^{t,x}-x)|^4\Big)^{\frac{1}{2}}\Big(\EE|K_T^{t,x}|^4\Big)^{\frac{1}{2}}.
   \end{align*}
   By the mean value theorem, there exists $\eta \in [0,1]$ such that $$g(X_T^{t,x}) - g(x) - \nabla_x g(x)\cdot(X_T^{t,x}-x) = \frac{1}{2} (X_T^{s,x} - x)\transpose \nabla_x^2 g(\eta x + (1-\eta)X_T^{t,x})(X_T^{s,x} - x).$$
   Noticing that $\nabla_x^2 g$ is bounded, we have
   \begin{align*}
       \EE|g(X_T^{t,x}) - g(x) - \nabla_x g(x)\cdot(X_T^{t,x}-x)|^4 &\le C\EE|X_T^{t,x} - x|^8 \\&= C\EE|W_T - W_t|^8  \le C(T-t)^4.
   \end{align*}
   Moreover,
   \begin{equation*}
       \EE|K_T^{t,x}|^4 = \frac{4^4}{(T-t)^8}\EE|(W_T - W_{\frac{T+t}{2}})(W_\frac{T+t}{2} - W_t)\transpose|^4 \le C(T-t)^{-4}.
   \end{equation*}
   Hence, we obtain the inequality \eqref{eq:Hessian_proof_eq_1}. We remark that, in the general case, the estimates of $\EE|X_T^{t,x} - x|^8$ and $\EE|K_T^{t,x}|^4$ follow from standard SDE estimation techniques, akin to the proofs of \eqref{thm:eq2} and \eqref{thm:eq3}.
   
   Noticing that $(\hat{X}_T^{t,x}, \hat{K}_T^{t,x})$ and $(X_T^{t,x}, K_T^{t,x})$ have identical distributions, we have
   \begin{equation*}
         \EE|(g(\hat{X}_T^{t,x}) - g(x) - \nabla_x g(x)\cdot(\hat{X}_T^{t,x}-x))\hat{K}_T^{t,x}|^2 \le C.
   \end{equation*}
   Therefore, by Cauchy-Schwartz inequality,
   \begin{align*}
       &\EE|(g(X_T^{t,x}) - g(x))K_T^{t,x} + (g(\hat{X}_T^{t,x}) - g(x))\hat{K}_T^{t,x}|^2 \\
       \le &3\{\EE|(g(X_T^{t,x}) - g(x) - \nabla_x g(x)\cdot(X_T^{t,x}-x))K_T^{t,x}|^2\\
       & +\EE|(g(\hat{X}_T^{t,x}) - g(x) - \nabla_x g(x)\cdot(\hat{X}_T^{t,x}-x))\hat{K}_T^{t,x}|^2\\
       & + \EE|\nabla_x g(x)\cdot(X_T^{t,x} - x) K_T^{t,x} + \nabla_x g(x)\cdot(\hat{X}_T^{t,x} - x) \hat{K}_T^{t,x}|^2\}.
   \end{align*}
   We need only to prove that $\EE|\nabla_x g(x)\cdot(X_T^{t,x} - x) K_T^{t,x} + \nabla_x g(x)\cdot(\hat{X}_T^{t,x} - x) \hat{K}_T^{t,x}|^2 \le C$. In the case of Brownian motion, we have $K_T^{t,x} = \hat{K}_T^{t,x}$ and $X_T^{t,x} + \hat{X}_T^{t,x} = 2x$. Therefore, $\EE|\nabla_x g(x)\cdot(X_T^{t,x} - x) K_T^{t,x} + \nabla_x g(x)\cdot(\hat{X}_T^{t,x} - x) \hat{K}_T^{t,x}|^2= 0$. In the general case,
   \begin{align*}
       &\EE|\nabla_x g(x)\cdot(X_T^{t,x} - x) K_T^{t,x} + \nabla_x g(x)\cdot(\hat{X}_T^{t,x} - x) \hat{K}_T^{t,x}|^2 \\
       \le\,&2\EE|\nabla_x g(x)\cdot(X_T^{t,x} + \hat{X}_T^{t,x} - 2x)K_T^{t,x}|^2 + 2\EE|\nabla_x g(x)\cdot(\hat{X}_T^{t,x}-x)(K_T^{t,x} - \hat{K}_T^{t,x})|^2 \\
       \le\,&C \Big(\EE|X_T^{t,x} + \hat{X}_T^{t,x} - 2x|^4\Big)^{\frac{1}{2}}\Big(\EE|K_T^{t,x}|^4\Big)^{\frac{1}{2}} + C\Big(\EE|\hat{X}_T^{t,x}  - x|^4\Big)^{\frac{1}{2}}\Big(\EE|K_T^{t,x} - \hat{K}_T^{t,x}|^4\Big)^{\frac{1}{2}}.
   \end{align*}
   Standard estimation of SDE gives that $\EE|(X_T^{t,x} + \hat{X}_T^{t,x} - 2x)|^4 \le C(T-t)^{4}$, $\EE|K_T^{t,x}|^4 \le C(T-t)^{-4}$, $\EE|\hat{X}_T^{t,x}  - x|^4 \le C(T-t)^2$ and $\EE|K_T^{t,x} - \hat{K}_T^{t,x}|^4 \le C(T-t)^{-2}$. We omit the proof of this part.

   Follows the same proof, we can show that
   $$
   \EE |(f_u(s,X_s^{t,x}) - f(t,x)) K_s^{t,x} + (f_u(s,\hat{X}_s^{t,x}) - f(t,x)) \hat{K}_s^{t,x}|^2 \le C.
   $$
   Hence,
   \begin{equation*}
       \int_t^T \EE |(f_u(s,X_s^{t,x}) - f(t,x)) K_s^{t,x} + (f_u(s,\hat{X}_s^{t,x}) - f(t,x)) \hat{K}_s^{t,x}|^2 \rmd s \le C,
   \end{equation*}
   which concludes the proof.
\end{proof}
Now we return to present the proof of Theorem \ref{thm:infinite-variance-Hessian}.
\begin{proof}[Proof of Theorem~\ref{thm:infinite-variance-Hessian}]
Similarly to the proof of Theorem~\ref{thm:finite-variance-Hessian}, we only provide the proof for the Brownian motion case.
We intend to prove that
\begin{align*}
    &\EE|(g(X_T^{t,x}) - g(x)) K_T^{t,x}|^2 \ge C(T-t)^{-1},\\
    &\EE|(f(s,X_s^{t,x}) - f(t,x)) K_s^{t,x}|^2 \ge C(s-t)^{-1}.
\end{align*}
Recalling the inequality \eqref{eq:Hessian_proof_eq_1}
\begin{equation*}
    \EE|(g(X_T^{t,x}) - g(x) - \nabla_x g(x)\cdot(X_T^{t,x}-x))K_T^{t,x}|^2 \le C,
\end{equation*}
and similarly,
\begin{equation*}
    \EE|(f(s,X_s^{t,x}) - f(t,x) - \nabla_x f(t,x)\cdot(X_s^{t,x}-x))K_s^{t,x}|^2 \le C,
\end{equation*}
we only need to prove that
\begin{align*}
    &\EE|\nabla_x g(x)\cdot(X_T^{t,x} - x) K_T^{t,x}|^2 \ge C(T-t)^{-1}, \\
    &\EE|\nabla_x f(t,x)\cdot(X_s^{t,x} - x) K_s^{t,x}|^2 \ge C(s-t)^{-1}.
\end{align*}
The proof of the above two equations is the same, and we only give the proof of the first one. Noticing that for the any $a, b \in \RR^d$, $|a b\transpose| = |a||b|$ by the definition of Frobenius norm, we have
\begin{align*}
    &\EE|\nabla_x g(x)\cdot(X_T^{t,x} - x) K_T^{t,x}|^2 \\
    = &\frac{16}{(T-t)^4}\EE|\nabla_x g(x)\cdot(W_T - W_t) (W_T - W_\frac{T+t}{2})(W_\frac{T+t}{2} - W_t)\transpose|^2 \\
    =& \frac{16}{(T-t)^4} \EE|\nabla_x g(x)\cdot (W_T - W_t)|^2 |W_T - W_\frac{T+t}{2}|^2 |W_\frac{T+t}{2} - W_t|^2 \\
    \ge& \frac{C}{T-t}|\nabla_x g(x)|^2,
    \end{align*}
which concludes the proof.
\end{proof}

\subsection{Numerical experiments}\label{appendix: Hessian_exp}
A natural extension of the DPI is to incorporate the Hessian estimator analyzed above. For each sampled point $(t_i, x_i)$, we compute the Hessian matrix label $h_i$ using the introduced finite-variance estimator and incorporate it into our regression framework by augmenting the loss function. This allows us to match $h_i$ closely with $\nabla^2_x u_{\theta_k}(t_i, x_i)$, computed via second-order automatic differentiation. In principle, including the Hessian provides further information of the curvature of $u$, potentially capturing more nuanced local behavior.  
Specifically,
$h_i$ is computed using the Monte Carlo estimator given by 
\begin{align*}
    h_i &= \frac{1}{M}\sum_{j=1}^M \Bigg[(g(X_T^{t_i,x_i,i,j})-g(x_i)) K_T^{t_i,x_i,i,j} + (g(\hat{X}_T^{t_i,x_i,i,j})-g(x_i)) \hat{K}_T^{t_i,x_i,i,j} \\
        &\quad\quad\quad\quad\quad  + (T-t_i) \Big((f_{u_k}(s^{i,j},X_{s^{i,j}}^{t_i,x_i,i,j})- f_{u_k}(t_i,x_i)) K_{s^{i,j}}^{t_i,x_i,i,j} \\
        &\quad\quad\quad\quad\quad\quad\quad\quad\quad\quad  + (f_{u_k}(s^{i,j},\hat{X}_{s^{i,j}}^{t_i,x_i,i,j})- f_{u_k}(t_i,x_i)) \hat{K}_{s^{i,j}}^{t_i,x_i,i,j})\Big)\Bigg],
    \end{align*}
where $\{W_r^{i,j}\}_{1 \leq i \leq N, 1 \leq j \leq M, r \in [t_i, T]}$ are independently sampled paths of Brownian motions, time points $\{s^{i,j}\}_{1 \leq i \leq N, 1\leq j\leq M}$ are uniformly sampled from $[t_i, T]$, and $X_s^{t,x,i,j}$, $\hat{X}_s^{t,x,i,j}$,$K_s^{t,x,i,j}$ and $\hat{K}_s^{t,x,i,j}$ are samples to $X_s^{t,x}$, $\hat{X}_s^{t,x}$, $K_s^{t,x}$ and $\hat{K}_s^{t,x}$ by replacing $W_t$ with $W_t^{i,j}$. In the following numerical example, we consider the case of Brownian motion and use $\tilde{K}_s^{t,x} = \EE\left[K_s^{t,x}\mid X_s^{t,x}\right], \hat{\tilde{K}}_s^{t,x} = \EE\left[\hat{K}_s^{t,x}\mid \hat{X}_s^{t,x}\right]$ instead of $K_s^{t,x}, \hat{K}_s^{t,x}$ to further reduce variance as discussed. In this case,
\begin{align*}
    &X_s^{t,x} = x + (W_s - W_t),\, \hat{X}_s^{t,x} = x - (W_s - W_t),\, \\
    &\tilde{K}_s^{t,x} = \hat{\tilde{K}}_s^{t,x} = \frac{1}{(s-t)^2}[(W_s-W_t)(W_s-W_t)\transpose-(s-t)\II_d].
\end{align*}

We then generate additional labels for the Hessian and incorporate them into our regression framework by augmenting the loss function. The modified loss now includes a term of the form $\frac{\mu}{d^2} |h_i - \nabla^2_x u_\theta(t_i, x_i)|^2$,
where \(\mu \ge 0\) is a hyperparameter that balances the importance of the Hessian term relative to the value and gradient terms. In doing so, the overall loss function becomes
\begin{equation}\label{eq:dpi_hess}
\begin{aligned}
\mathcal{L}_{\text{DPI}}(\theta) = & \frac{1}{N} \sum_{i=1}^N \left[ \left|y_i - u_\theta(t_i, {x_i})\right|^2 + \frac{\lambda}{d}\left|{z_i} - \nabla_{{x}} u_\theta(t_i, {x_i})\right|^2\right. \\
&\quad\quad\quad\quad +\left. \frac{\mu}{d^2}\left|{h_i} - \nabla^2_{{x}} u_\theta(t_i, {x_i})\right|^2\right].
\end{aligned}
\end{equation}
This extension leverages both first and second-order information during supervised learning, resulting in a more robust approximation of the underlying function.

Here we conduct experiments on the fully nonlinear problem in Section~\ref{sec:fully} to evaluate the impact of incorporating a Hessian term into the loss function. By adjusting the magnitude of the parameter \(\mu\), we investigate the effect of different strengths on the Hessian term. For simplicity, all experiments are performed under the PDE parameters corresponding to ``Case I" in the main text. The relative mean absolute error of the Hessian (h-rMAE) is defined similarly to the previously computed rMAE and g‑rMAE as follows:
\[
    \operatorname{h-rMAE} = \frac{1}{d^2} \sum_{k=1}^{d} \sum_{j=1}^{d} \frac{\displaystyle{\sum_i\left|\partial_{x_j}\partial_{x_k} u_\theta(t_i, X_{t_i}) - \partial_{x_j}\partial_{x_k} u^*(t_i, X_{t_i})\right|}}{\displaystyle{\sum_i|\partial_{x_j} u^*(t_i, X_{t_i})|}}.
\]

The experimental results, shown in Figure~\ref{fig:hess}, display the relative errors of the value, gradient, and Hessian for different \(\mu\) under a fixed gradient parameter \(\lambda\) in each column. The left column corresponds to \(\lambda=0.1\), while the right column corresponds to \(\lambda=1.0\). When \(\lambda\) is small (for example, \(\lambda=0.1\)), the insufficient supervision constraint on the gradient results in overall performance that is slightly inferior to the case when \(\lambda=1.0\). In such cases, appropriately increasing the penalty weight on the Hessian can significantly improve the rMAE, reflecting the importance of Hessian information in capturing local curvature. However, when \(\lambda\) is set to an appropriate value such that rMAE is already very low, further increasing the penalty weight for the Hessian term mainly reduces the relative error of the Hessian itself, with only limited improvements in the relative errors of the value and gradient. Overall, when gradient supervision is sufficient, the additional Hessian penalty does not substantially enhance the overall results; its primary contribution lies in improving the accuracy of the Hessian estimation, thereby providing the model with more detailed second-order information.

\begin{figure}[!htb]
    \centering
    \begin{subfigure}[b]{0.49\textwidth}
        \centering
\includegraphics[width=\textwidth]{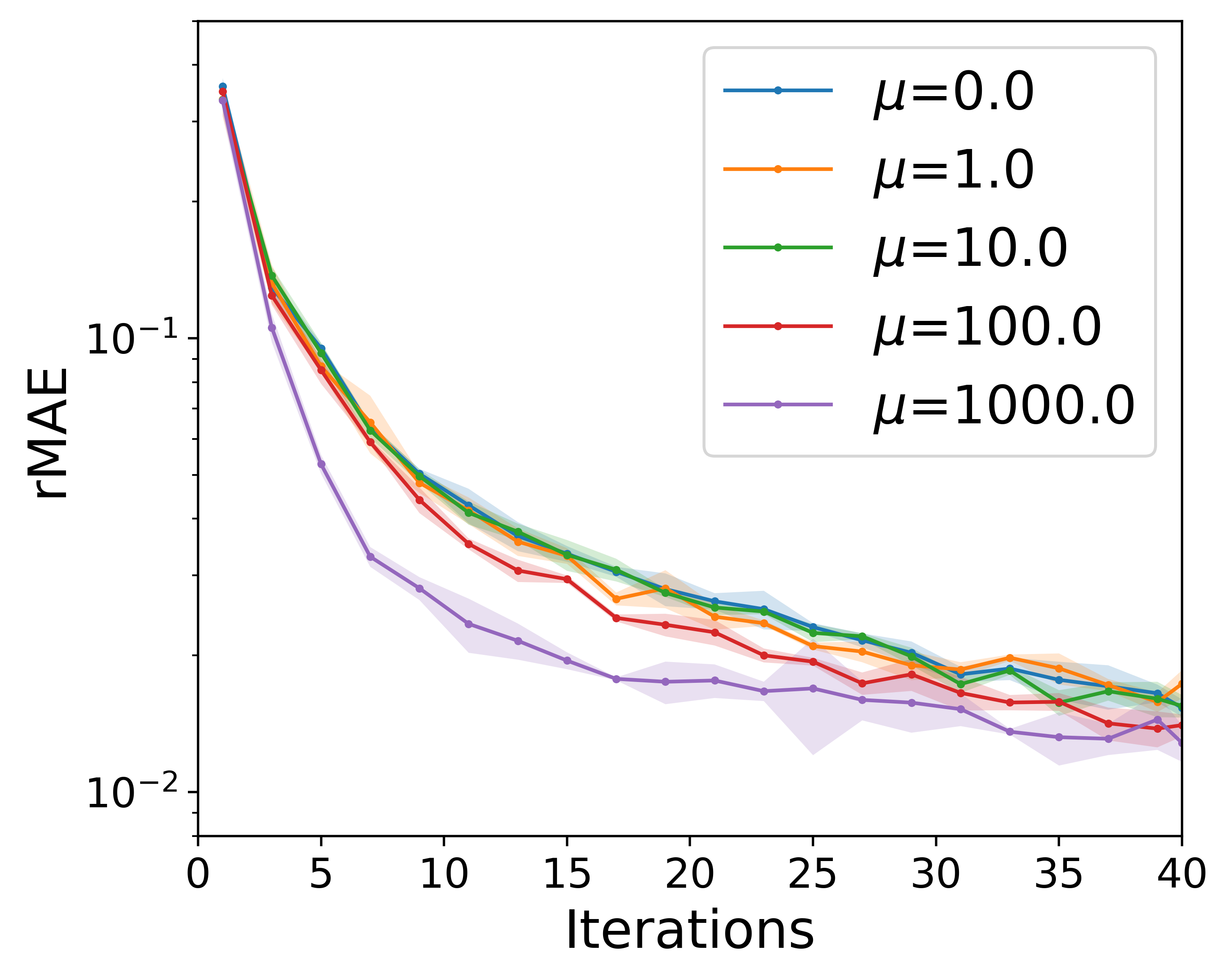}
    \end{subfigure}
    \hfill
    \begin{subfigure}[b]{0.49\textwidth}
        \centering
        \includegraphics[width=\textwidth]{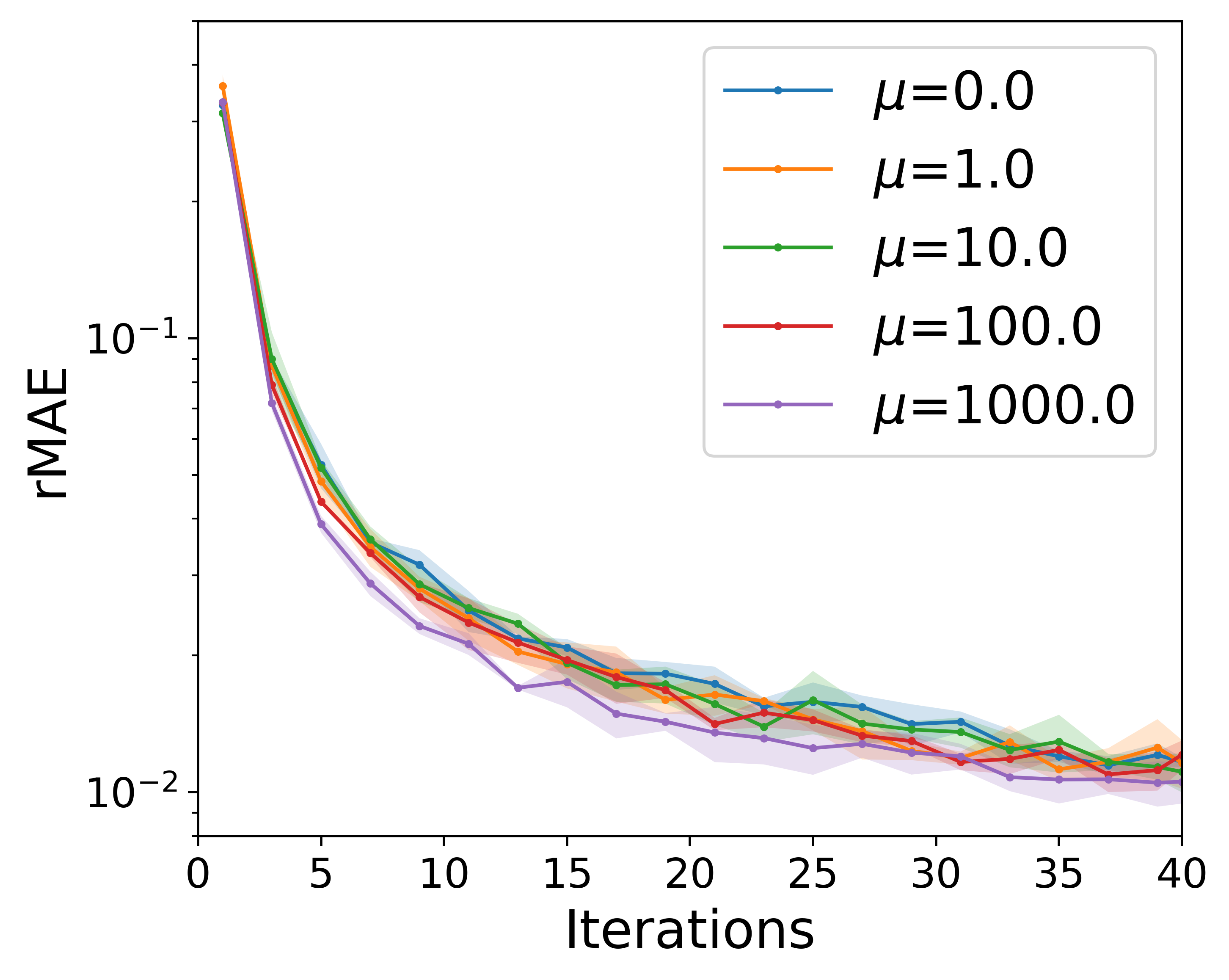}
    \end{subfigure}

        \begin{subfigure}[b]{0.49\textwidth}
        \centering
\includegraphics[width=\textwidth]{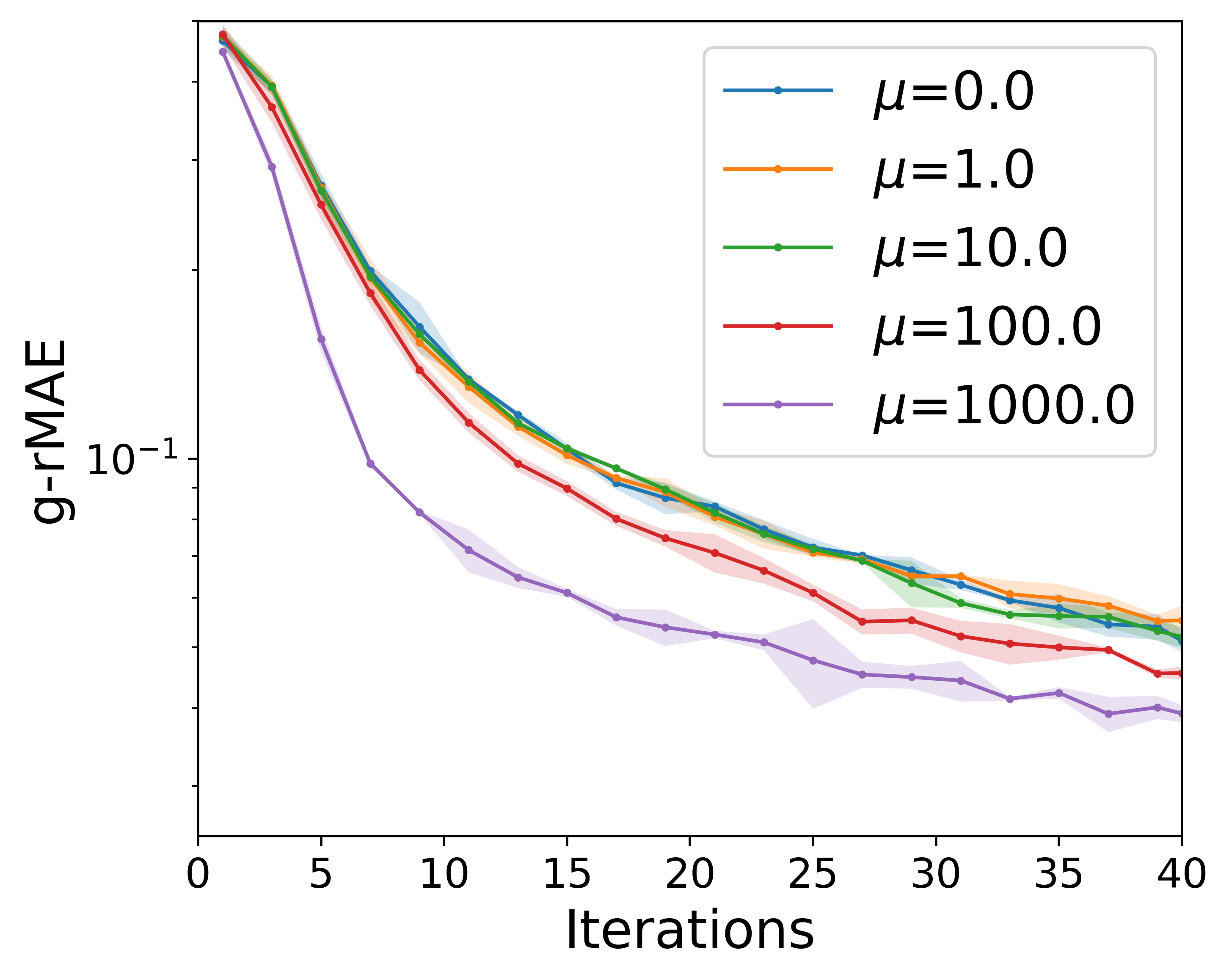}
    \end{subfigure}
    \hfill
    \begin{subfigure}[b]{0.49\textwidth}
        \centering
        \includegraphics[width=\textwidth]{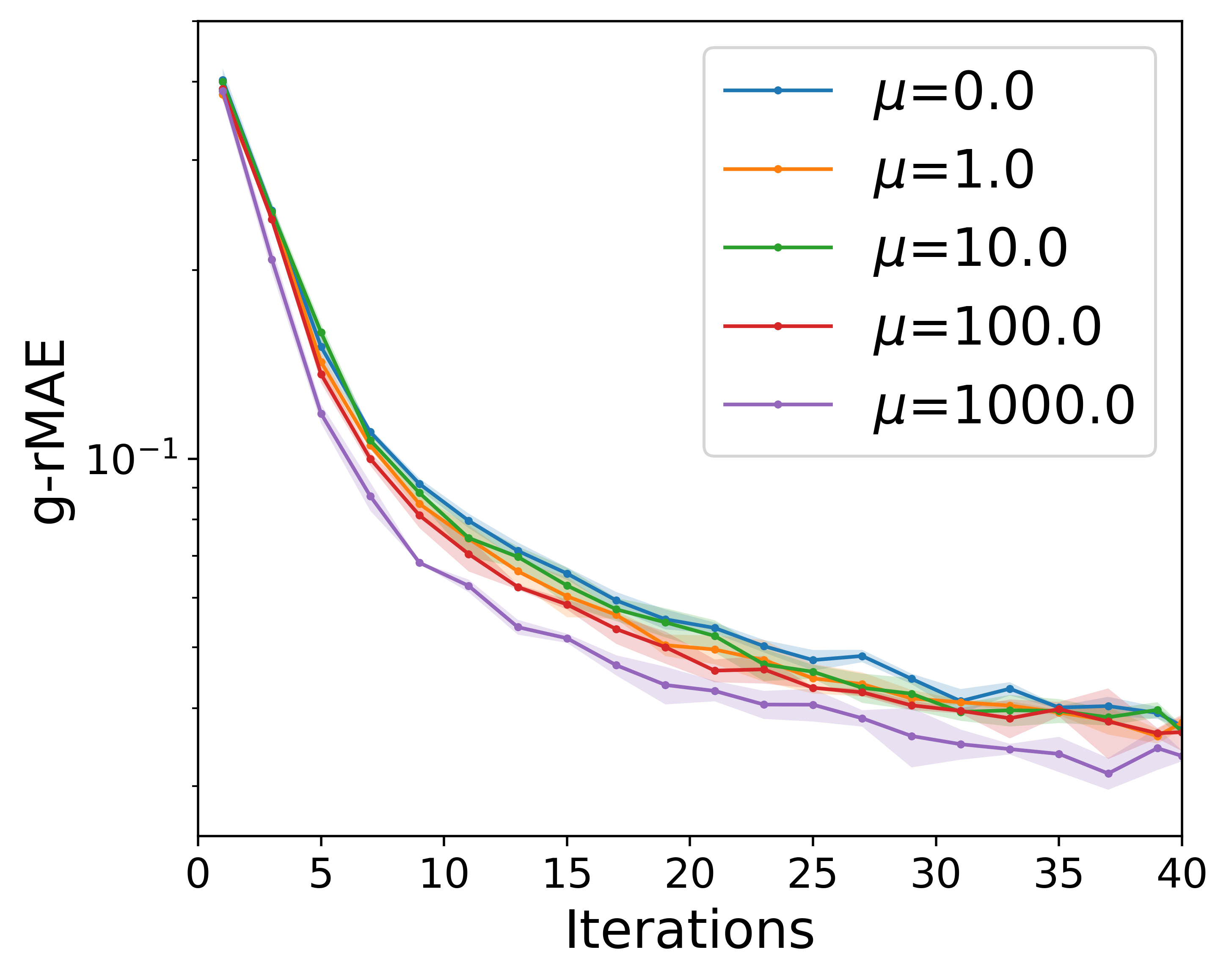}
    \end{subfigure}

        \begin{subfigure}[b]{0.49\textwidth}
        \centering
\includegraphics[width=\textwidth]{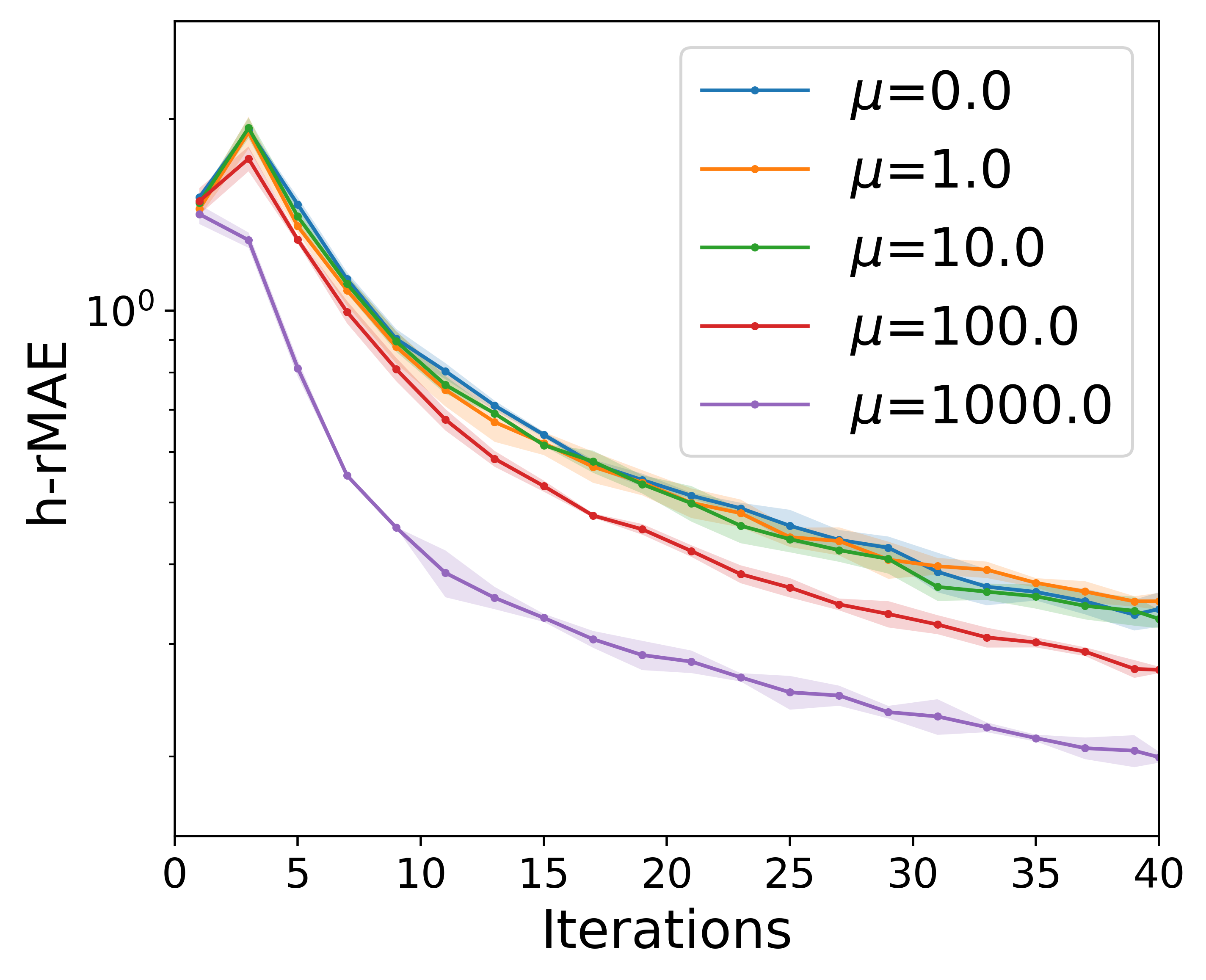}
    \end{subfigure}
    \hfill
    \begin{subfigure}[b]{0.49\textwidth}
        \centering
        \includegraphics[width=\textwidth]{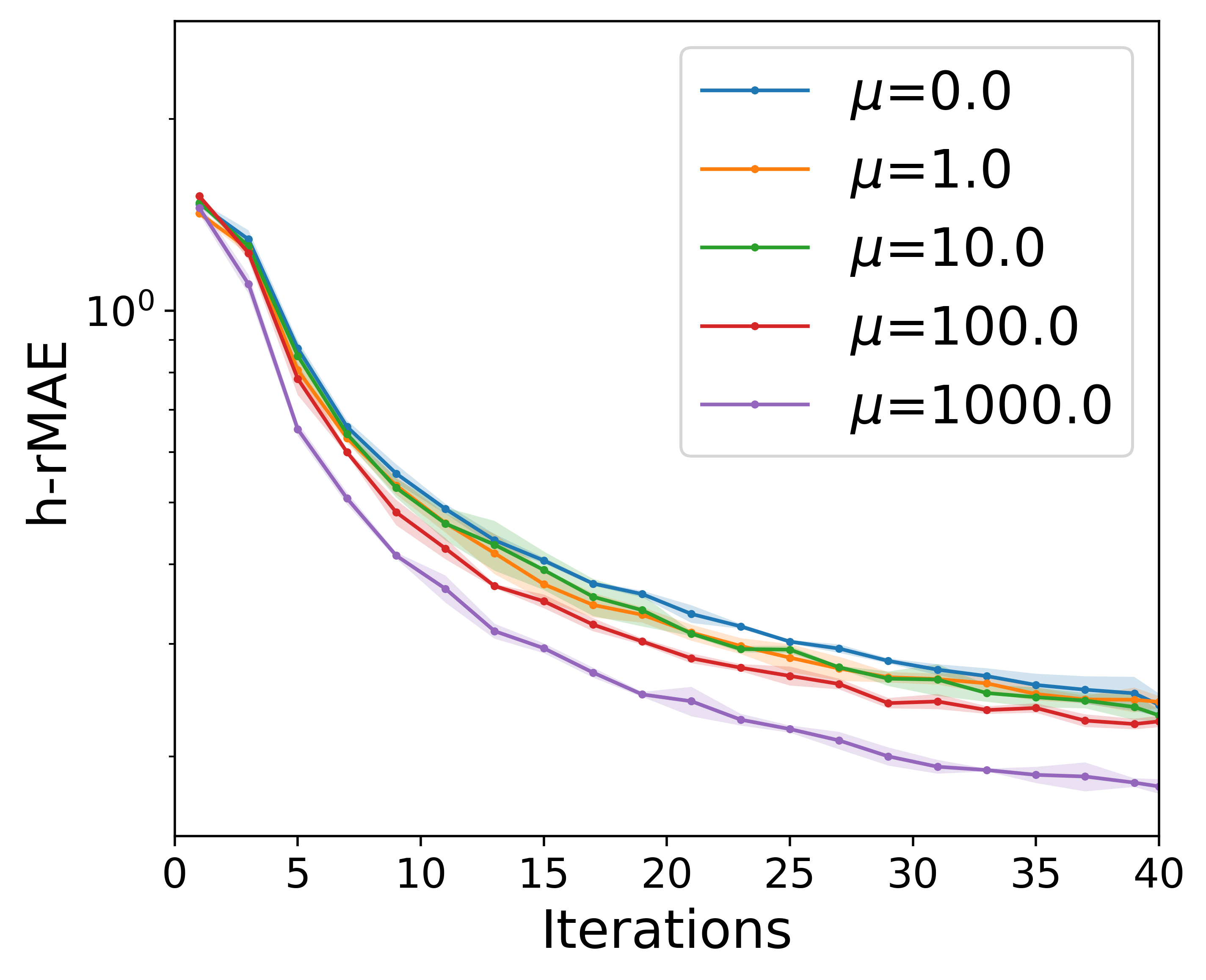}
    \end{subfigure}
    \caption{In the fully nonlinear problem ~\eqref{eq:gbm}, the relative errors for \( u \), \( \nabla u \), and \( \nabla^2 u \) are compared for different Hessian weights \( \mu \) using the DPI loss function~\eqref{eq:dpi_hess} including the Hessian term. The left column represents the setting with a gradient weight of \( \lambda=0.1 \), while the right column represents the setting with \( \lambda=1.0 \).}
    \label{fig:hess}
\end{figure}

\section{Test Error Curves During Training}
\label{appendix:errorcurve}

\begin{figure}[!htb]
    \centering
    \begin{subfigure}[b]{0.49\textwidth}
        \centering
\includegraphics[width=\textwidth]{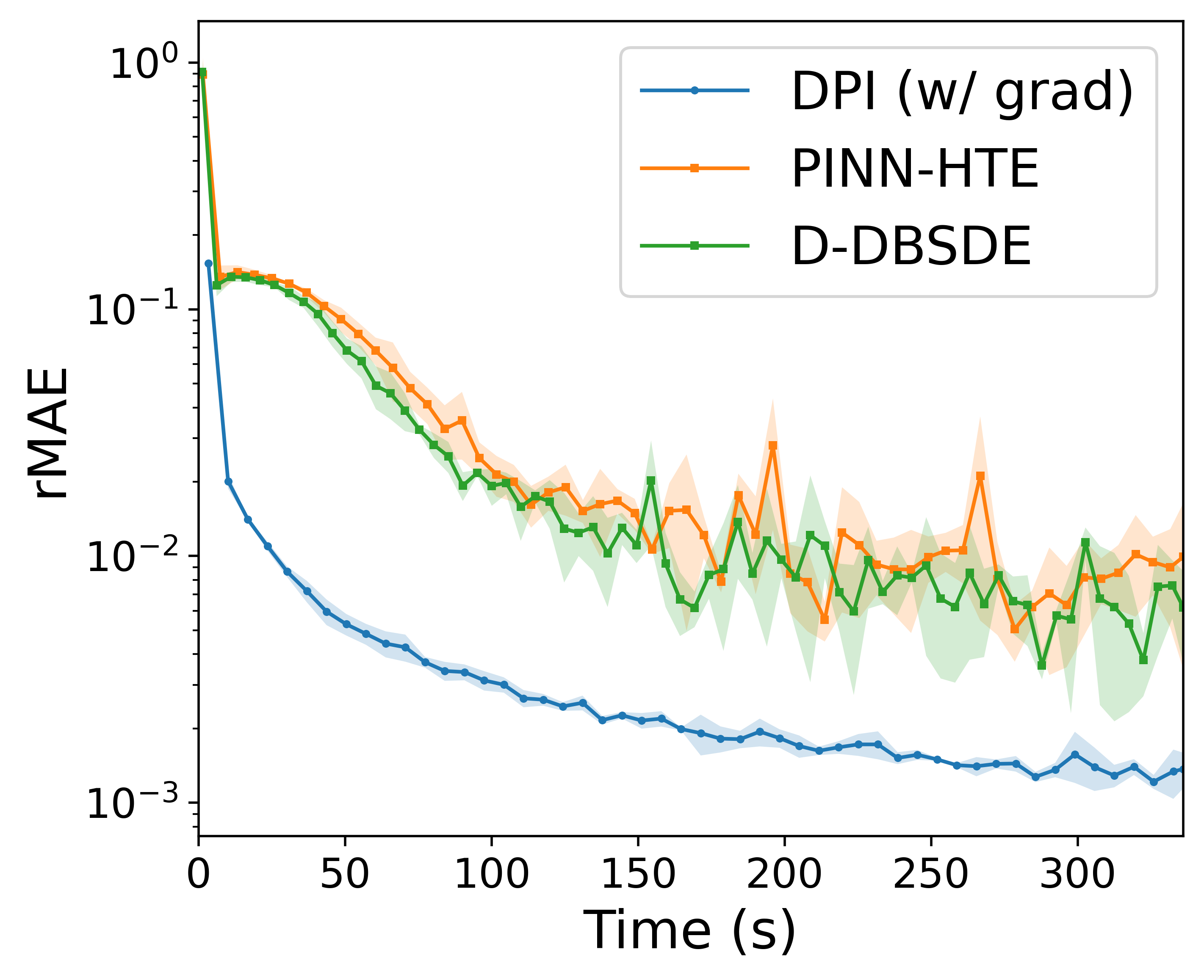}
    \end{subfigure}
    \hfill
    \begin{subfigure}[b]{0.49\textwidth}
        \centering
        \includegraphics[width=\textwidth]{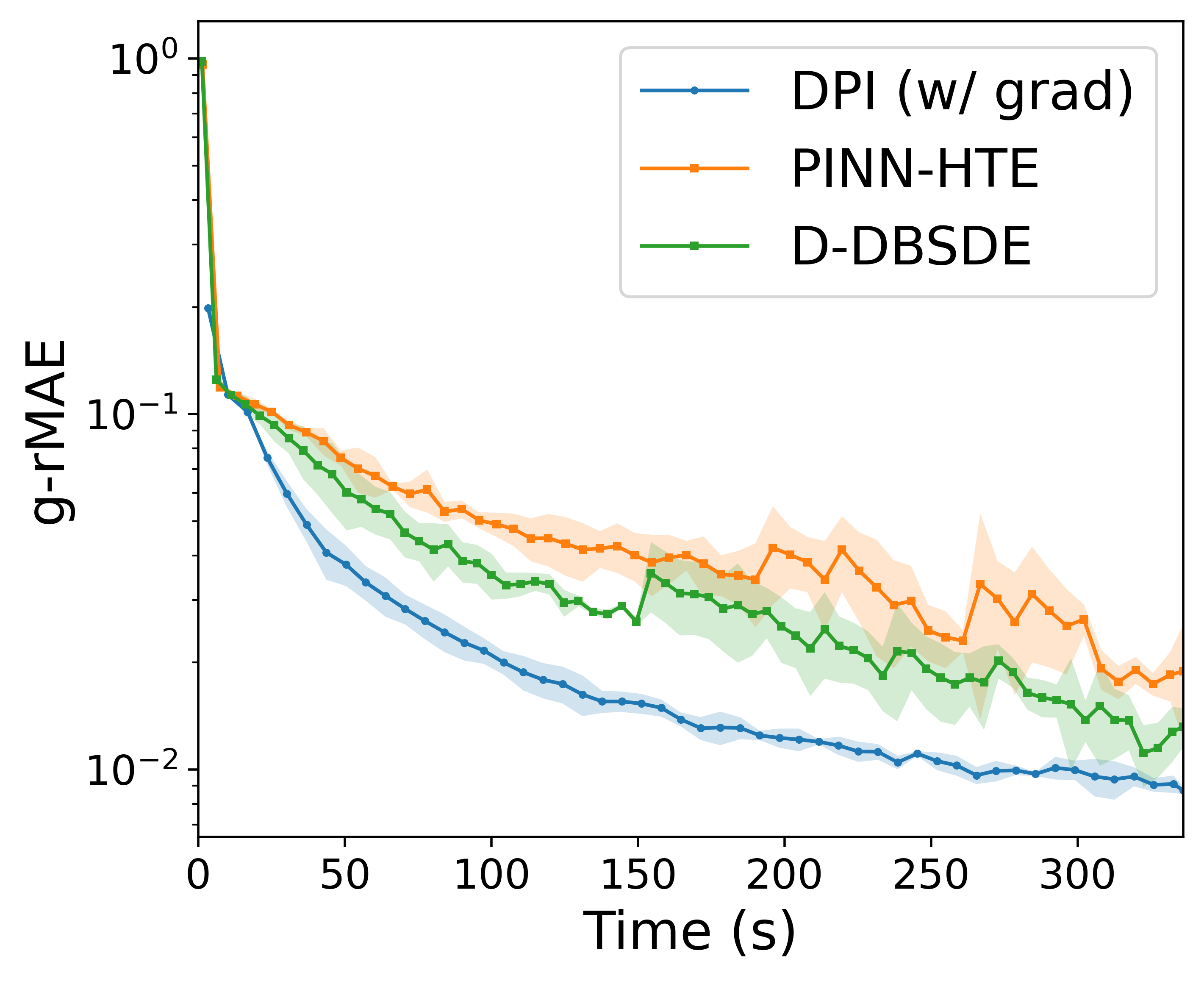}
    \end{subfigure}
    \begin{subfigure}[b]{0.49\textwidth}
        \centering
\includegraphics[width=\textwidth]{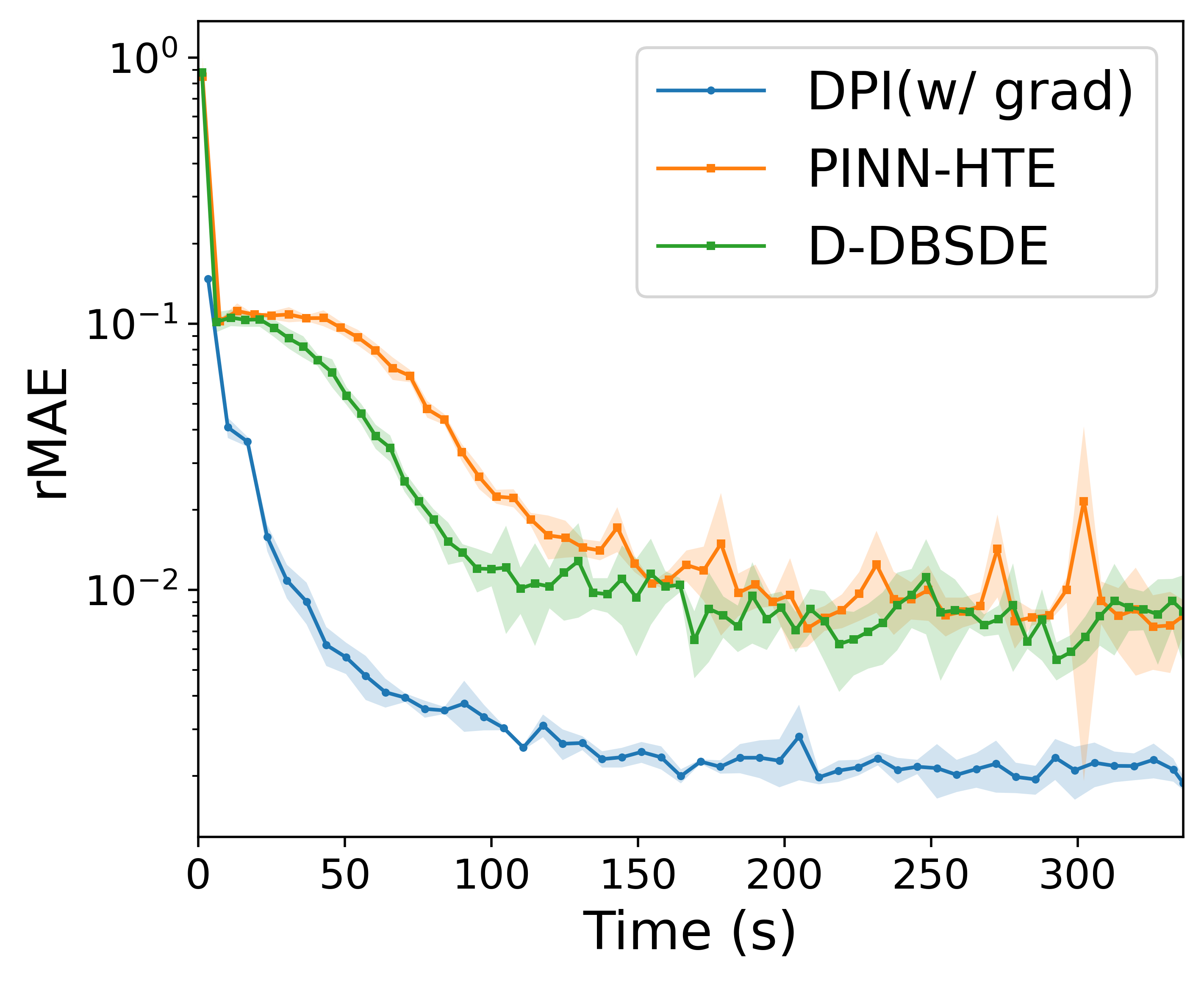}
    \end{subfigure}
    \hfill
    \begin{subfigure}[b]{0.49\textwidth}
        \centering
        \includegraphics[width=\textwidth]{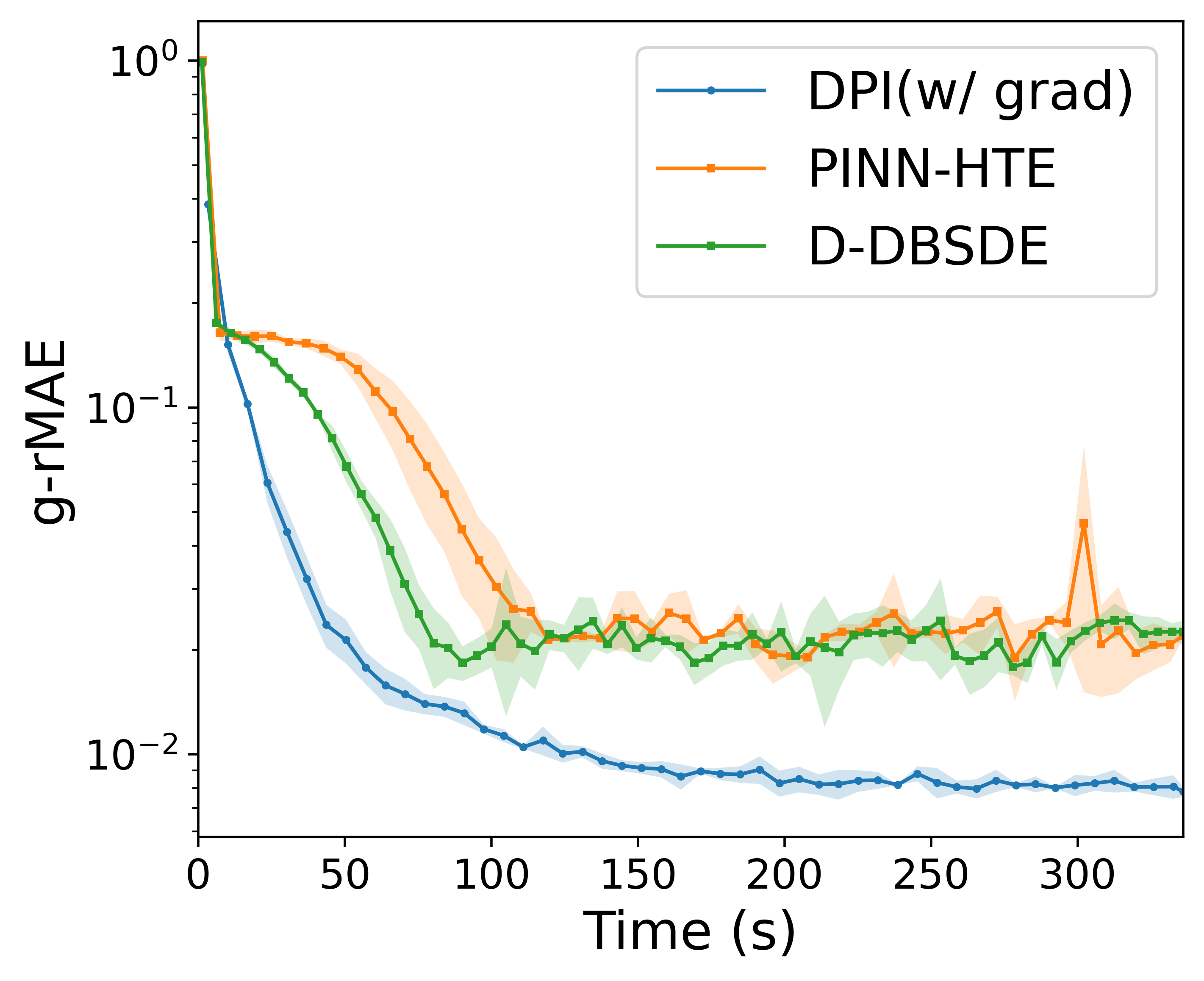}
    \end{subfigure}

    \begin{subfigure}[b]{0.49\textwidth}
        \centering
\includegraphics[width=\textwidth]{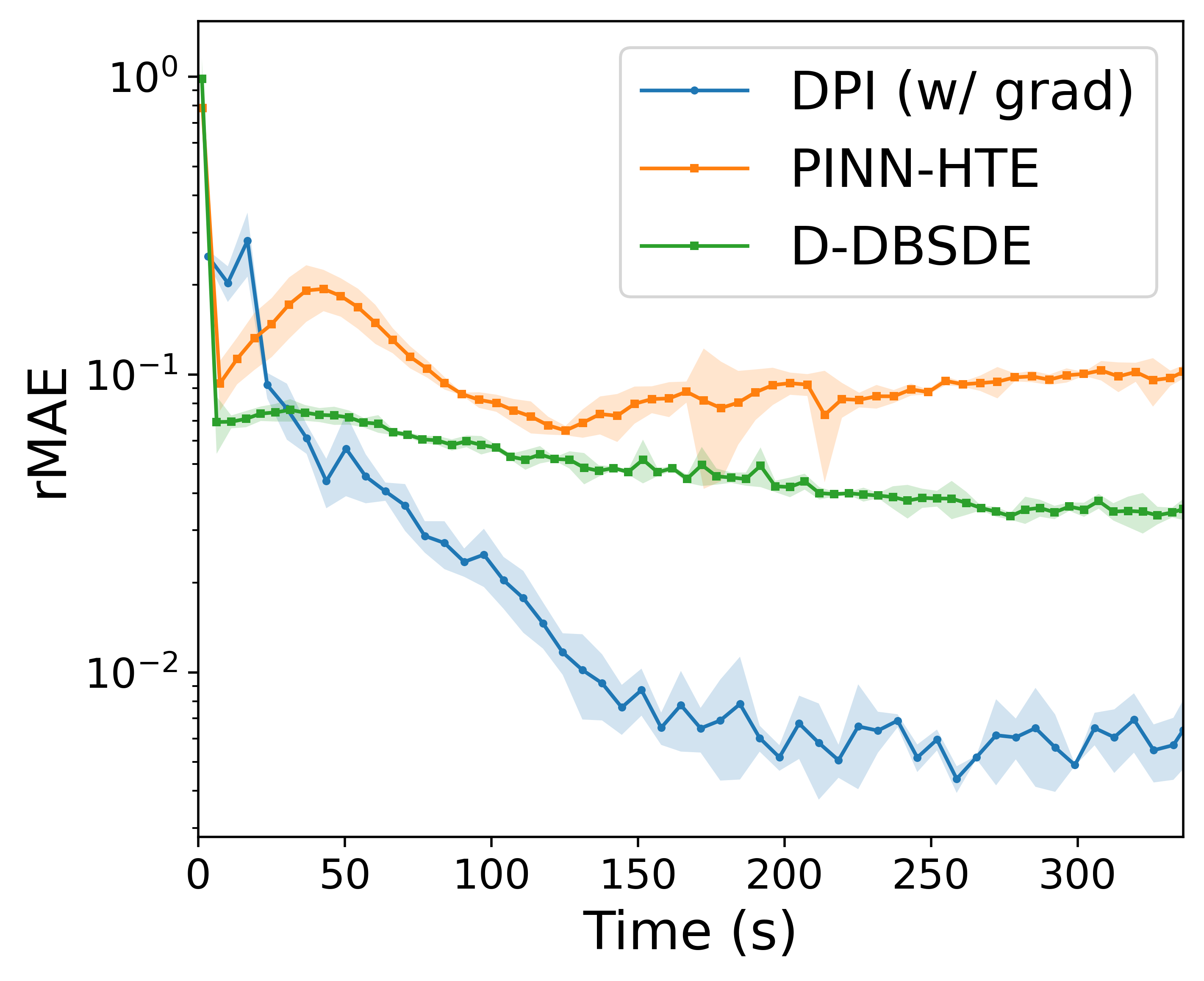}
    \end{subfigure}
    \hfill
    \begin{subfigure}[b]{0.49\textwidth}
        \centering
        \includegraphics[width=\textwidth]{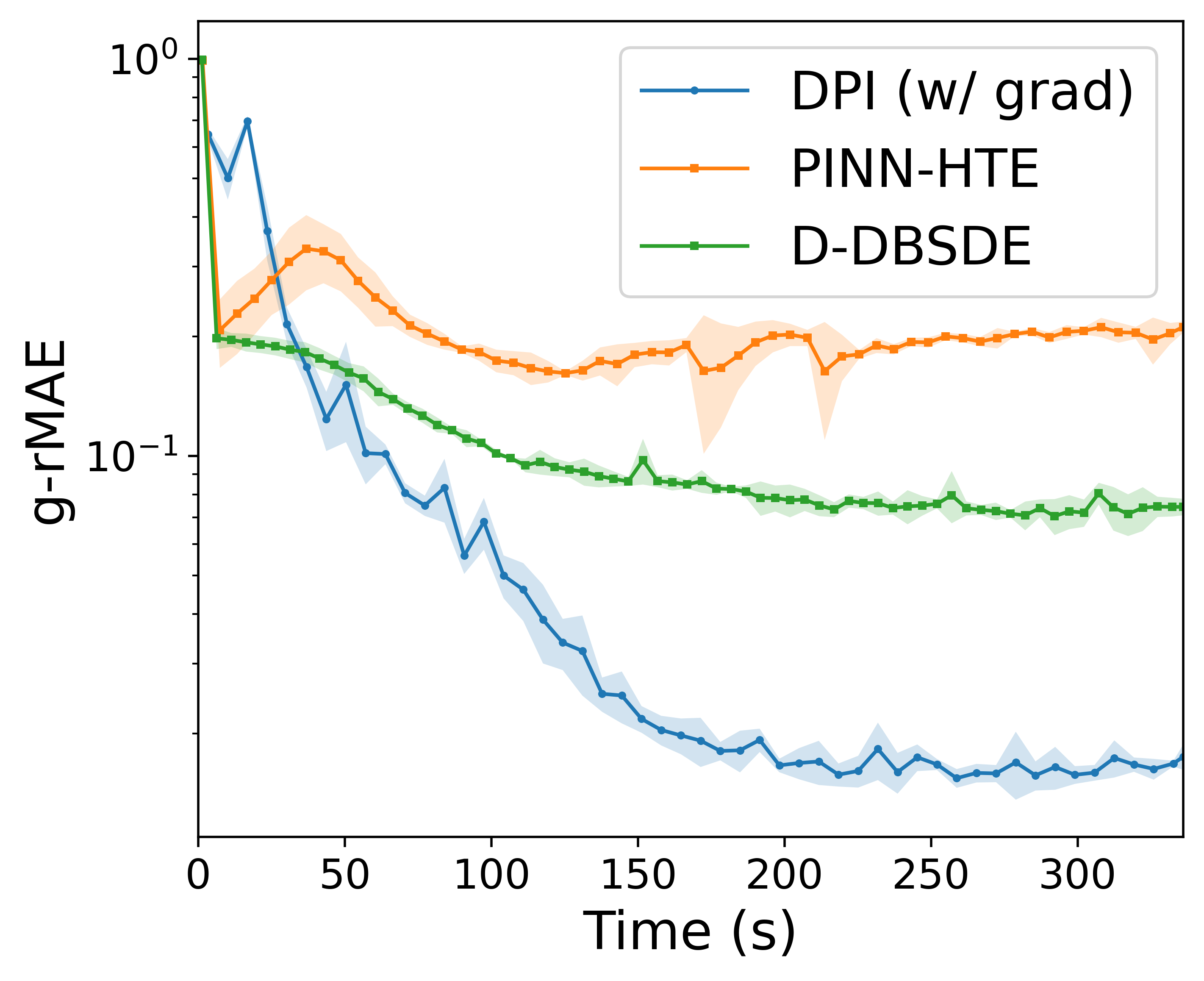}
    \end{subfigure}
\caption{The relative error curves for \( u \) and \( \nabla u \) during training of DPI, PINN-HTE and D-DBSDE for the Burgers-type problem in Section \ref{sec:cha}. From top to bottom, the results correspond to $\kappa=1.0, 2.5, 5.0$.}
    \label{fig:burgers_more}
\end{figure}

\begin{figure}[!htb]
    \centering
    \begin{subfigure}[b]{0.49\textwidth}
        \centering
\includegraphics[width=\textwidth]{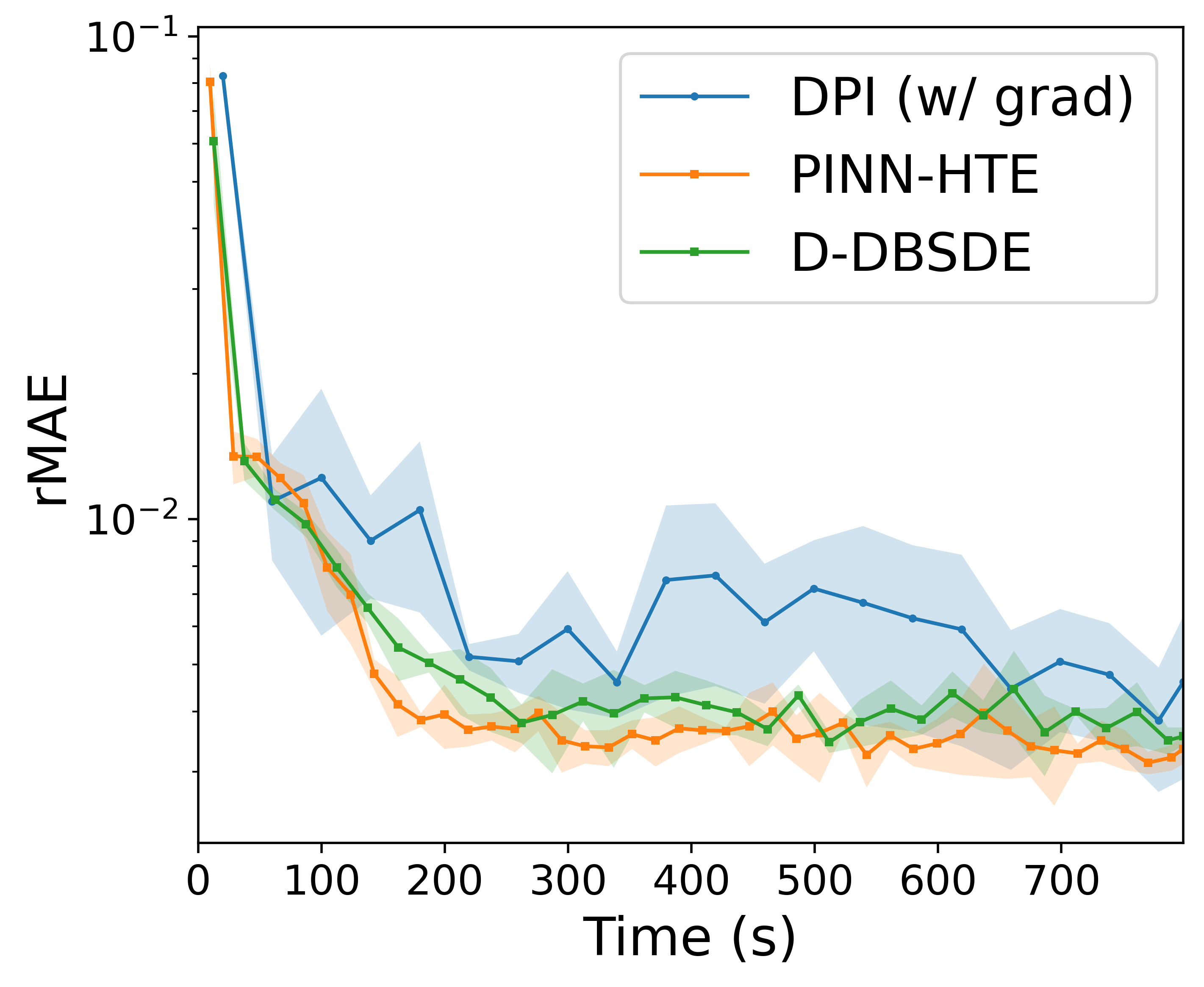}
    \end{subfigure}
    \hfill
    \begin{subfigure}[b]{0.49\textwidth}
        \centering
        \includegraphics[width=\textwidth]{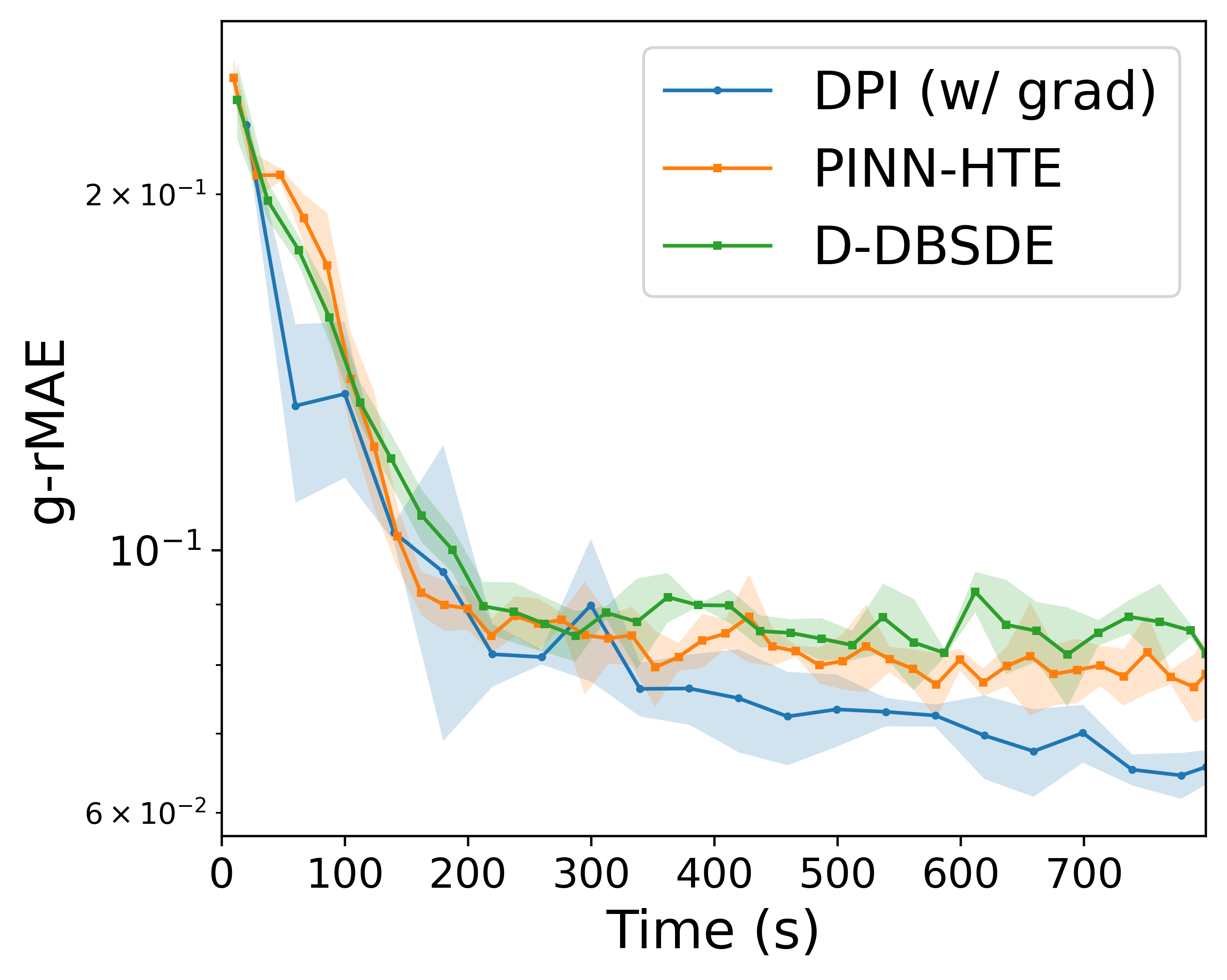}
    \end{subfigure}
    \begin{subfigure}[b]{0.49\textwidth}
        \centering
\includegraphics[width=\textwidth]{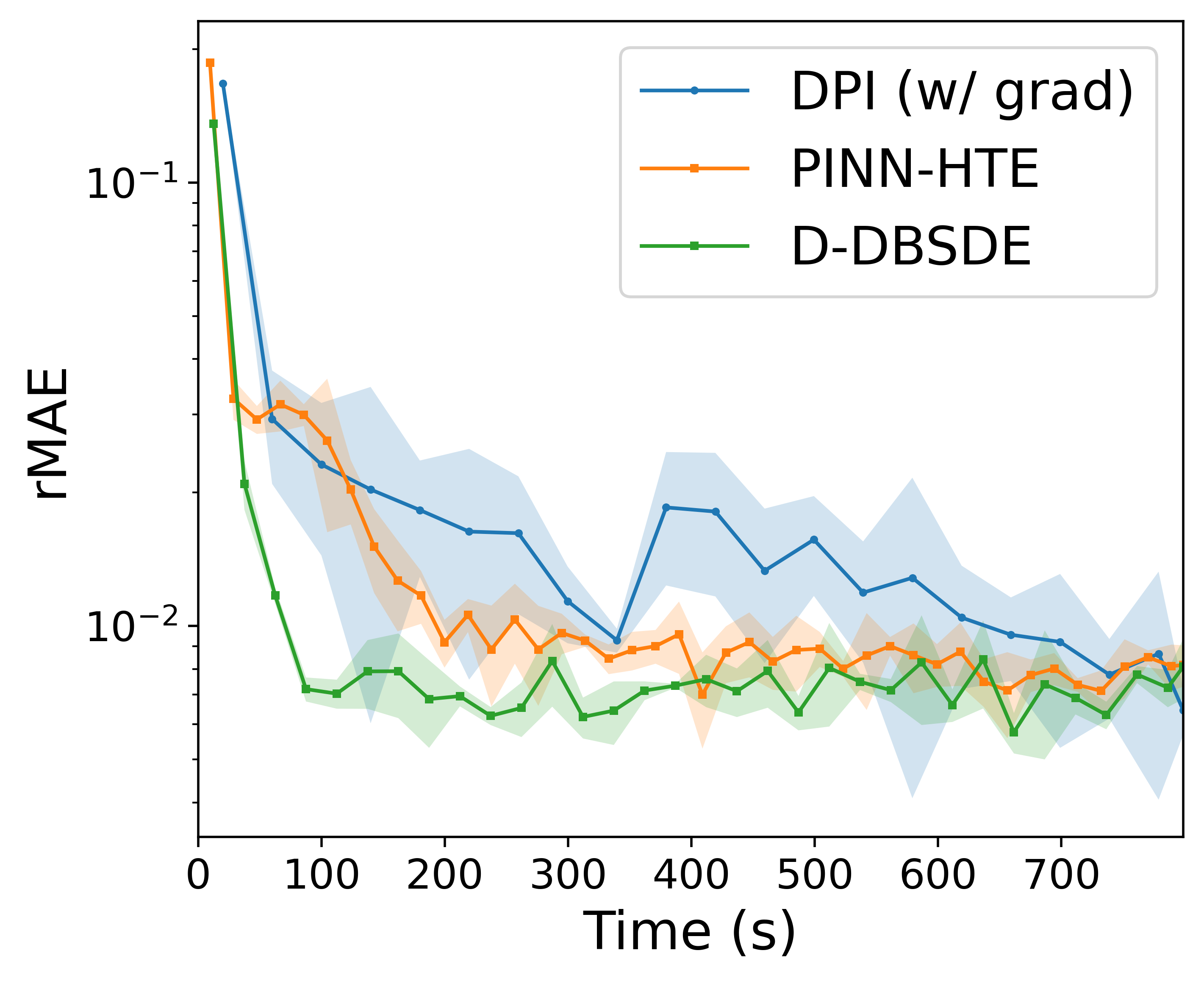}
    \end{subfigure}
    \hfill
    \begin{subfigure}[b]{0.49\textwidth}
        \centering
        \includegraphics[width=\textwidth]{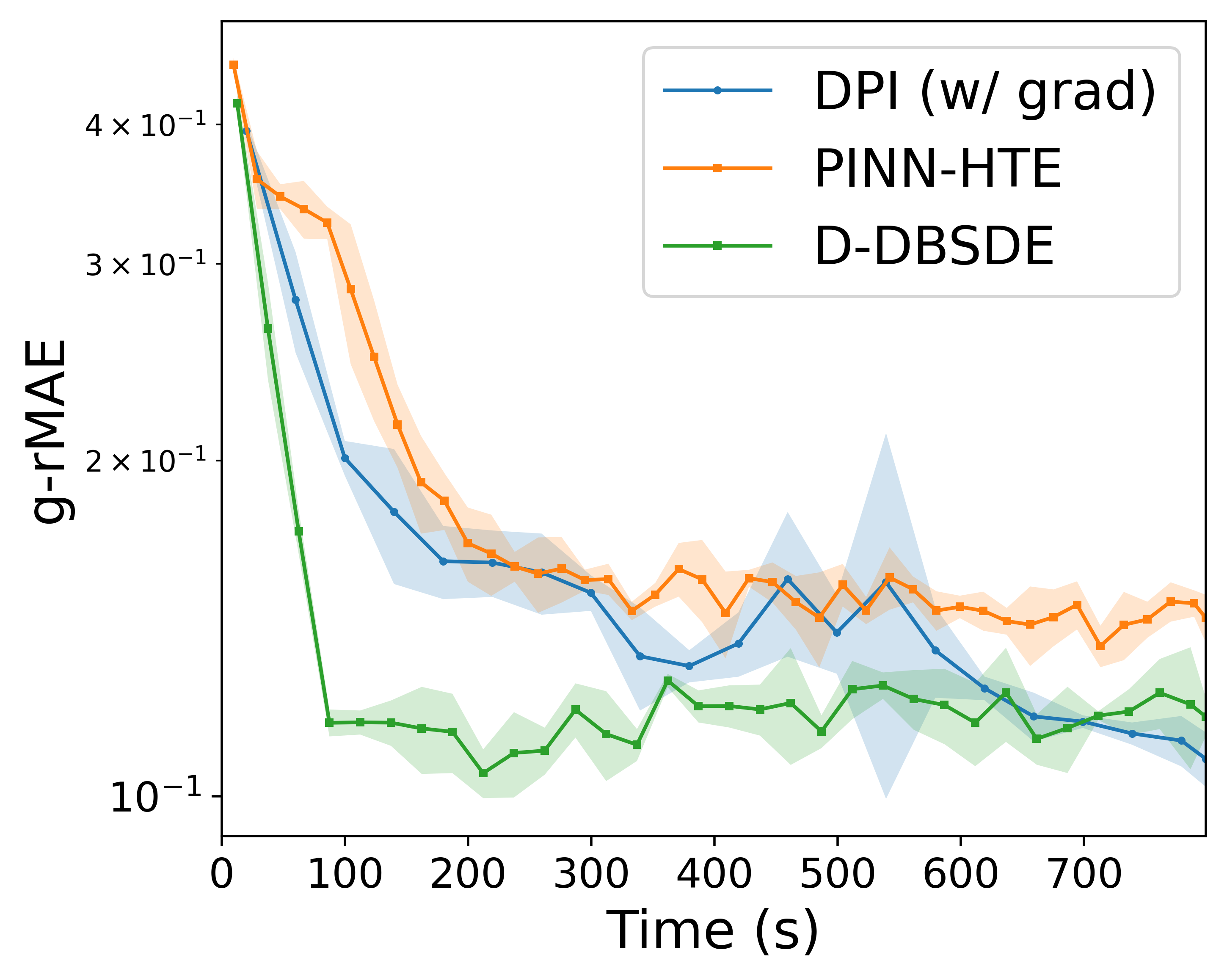}
    \end{subfigure}

    \begin{subfigure}[b]{0.49\textwidth}
        \centering
\includegraphics[width=\textwidth]{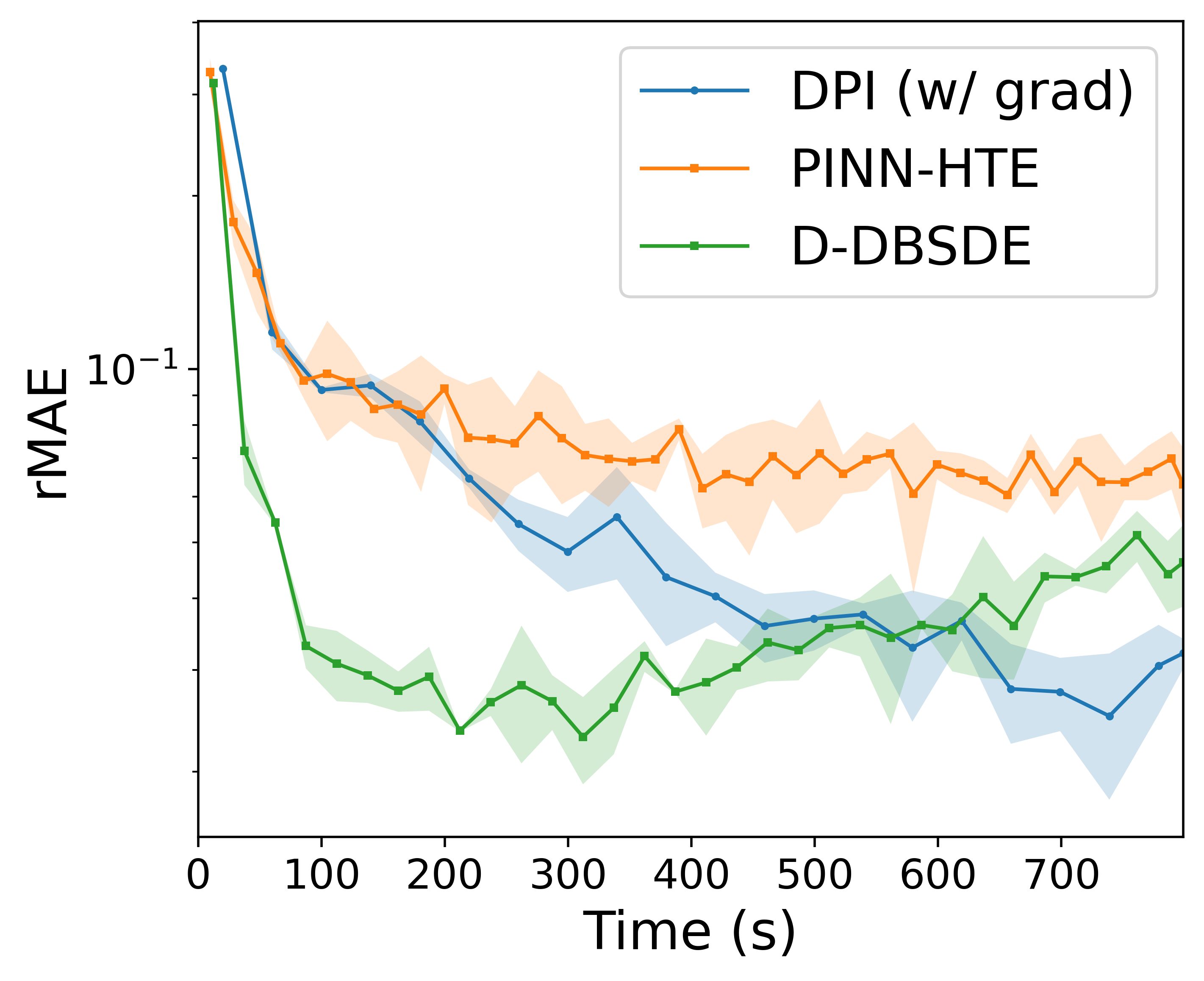}
    \end{subfigure}
    \hfill
    \begin{subfigure}[b]{0.49\textwidth}
        \centering
        \includegraphics[width=\textwidth]{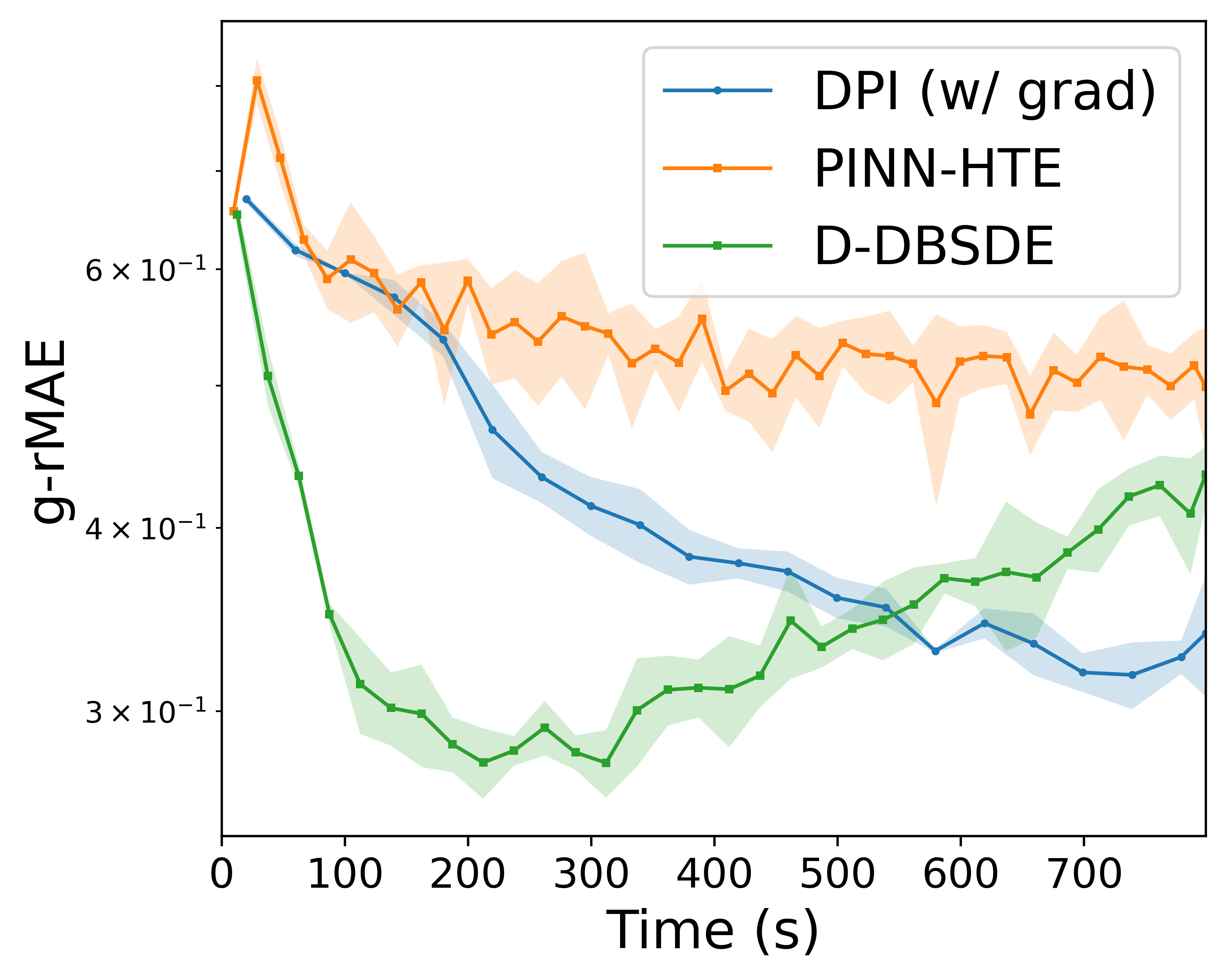}
    \end{subfigure}
\caption{The relative error curves for \( u \) and \( \nabla u \) during training of DPI, PINN-HTE and D-DBSDE for the HJB problem in Section \ref{sec:hjb}. From top to bottom, the results correspond to $T=0.25, 0.5, 1.0$.}
    \label{fig:hjb}
\end{figure}

\begin{figure}[!htb]
    \centering
    \begin{subfigure}[b]{0.49\textwidth}
        \centering
\includegraphics[width=\textwidth]{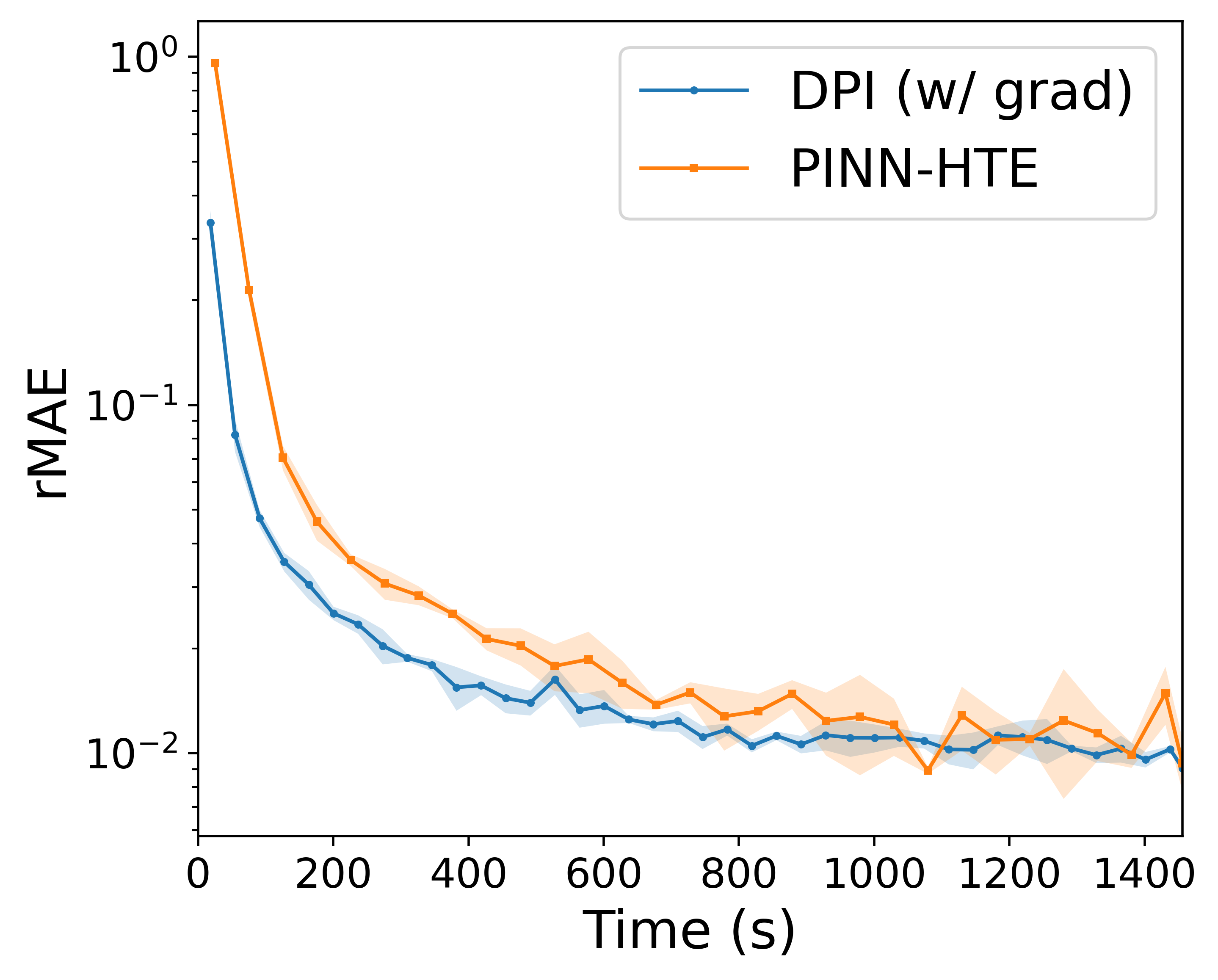}
    \end{subfigure}
    \hfill
    \begin{subfigure}[b]{0.49\textwidth}
        \centering
        \includegraphics[width=\textwidth]{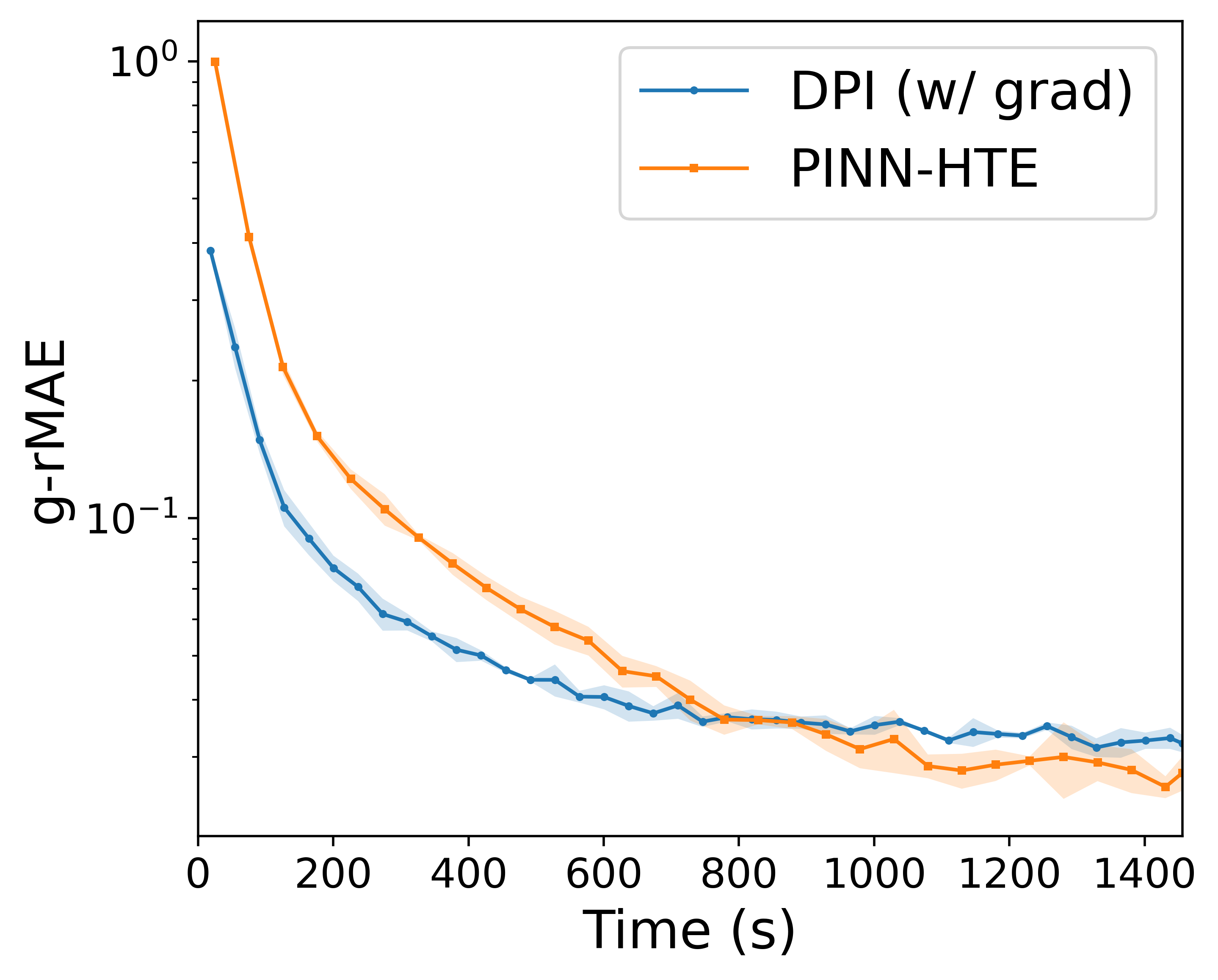}
    \end{subfigure}
    \begin{subfigure}[b]{0.49\textwidth}
        \centering
\includegraphics[width=\textwidth]{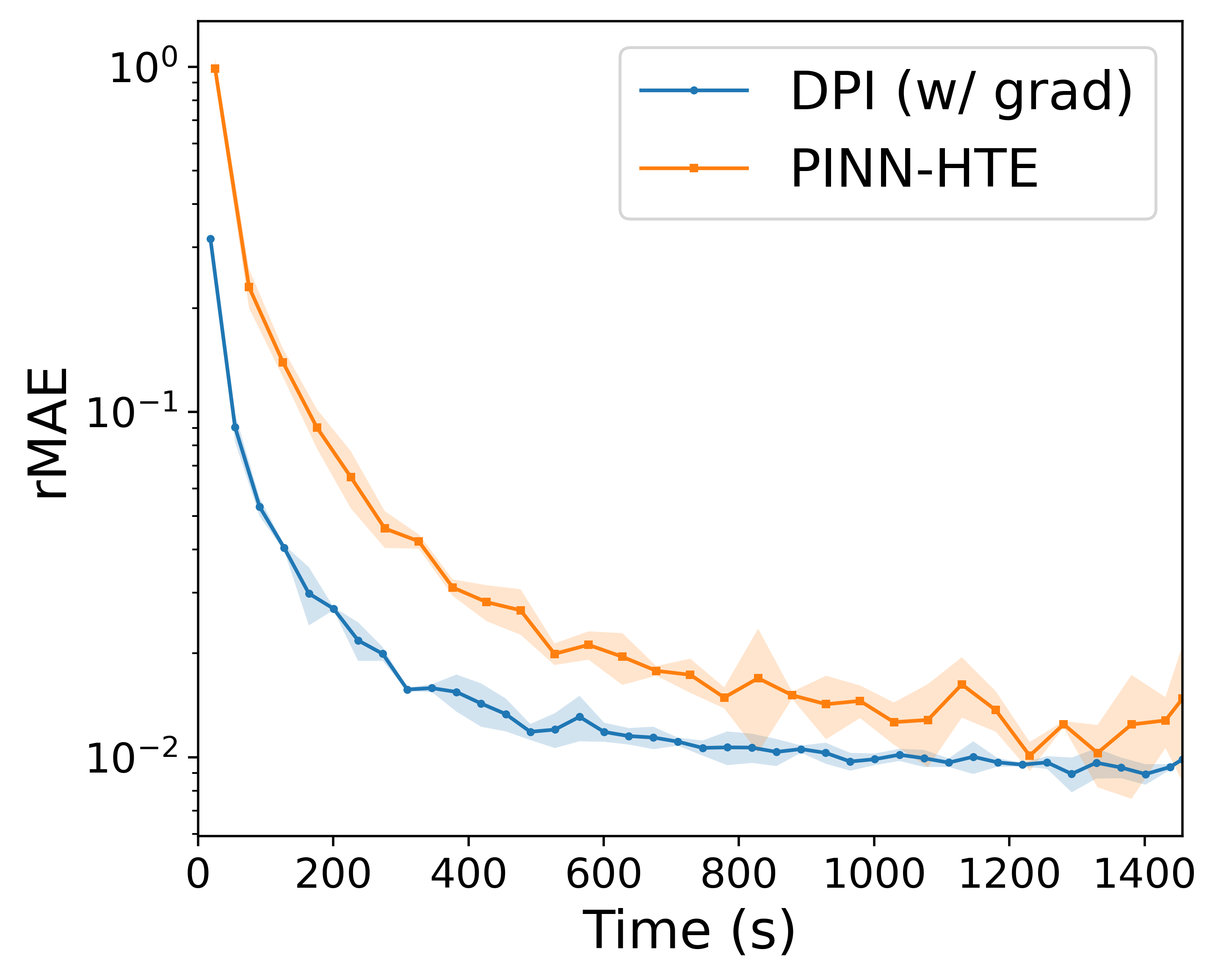}
    \end{subfigure}
    \hfill
    \begin{subfigure}[b]{0.49\textwidth}
        \centering
        \includegraphics[width=\textwidth]{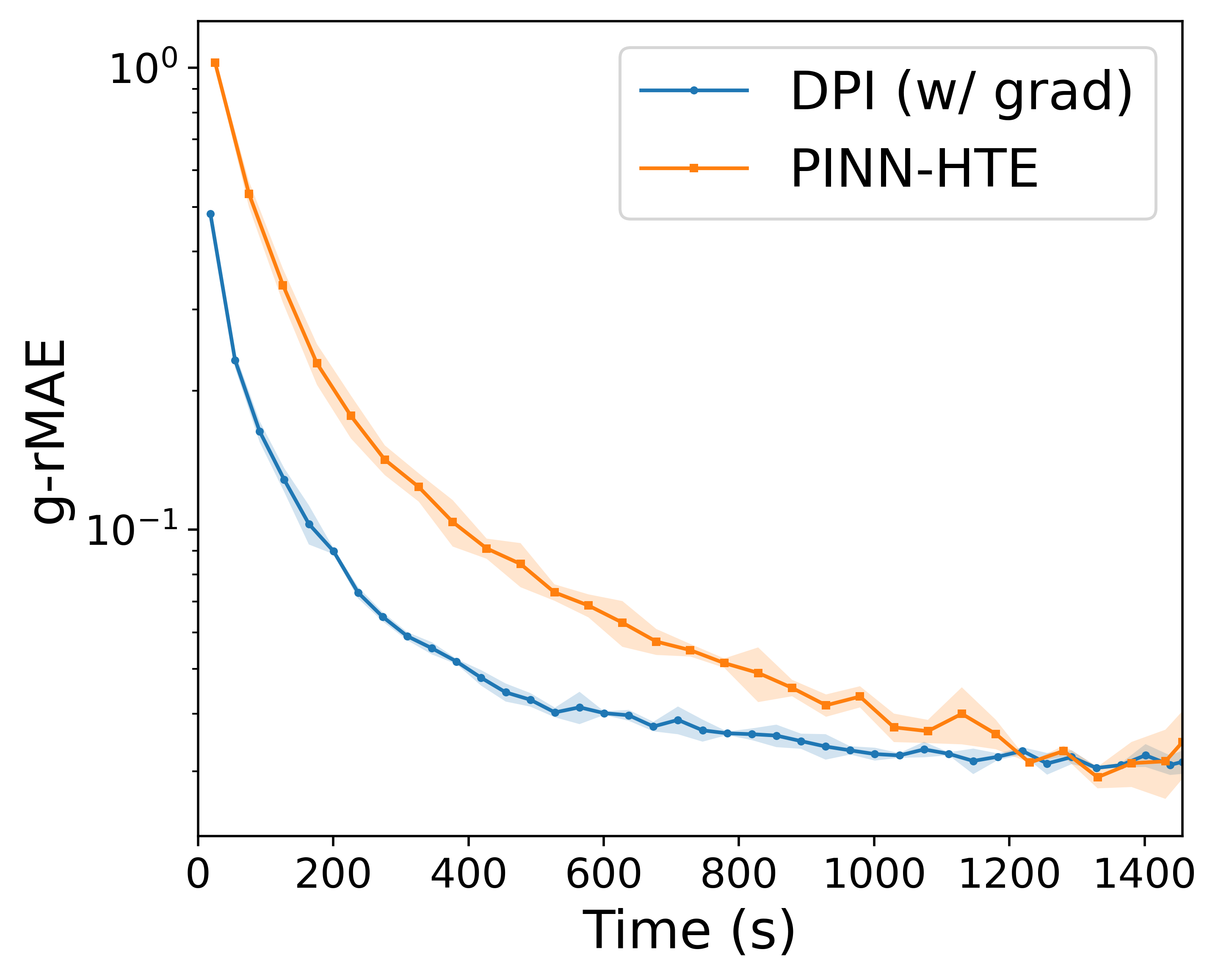}
    \end{subfigure}

    \begin{subfigure}[b]{0.49\textwidth}
        \centering
\includegraphics[width=\textwidth]{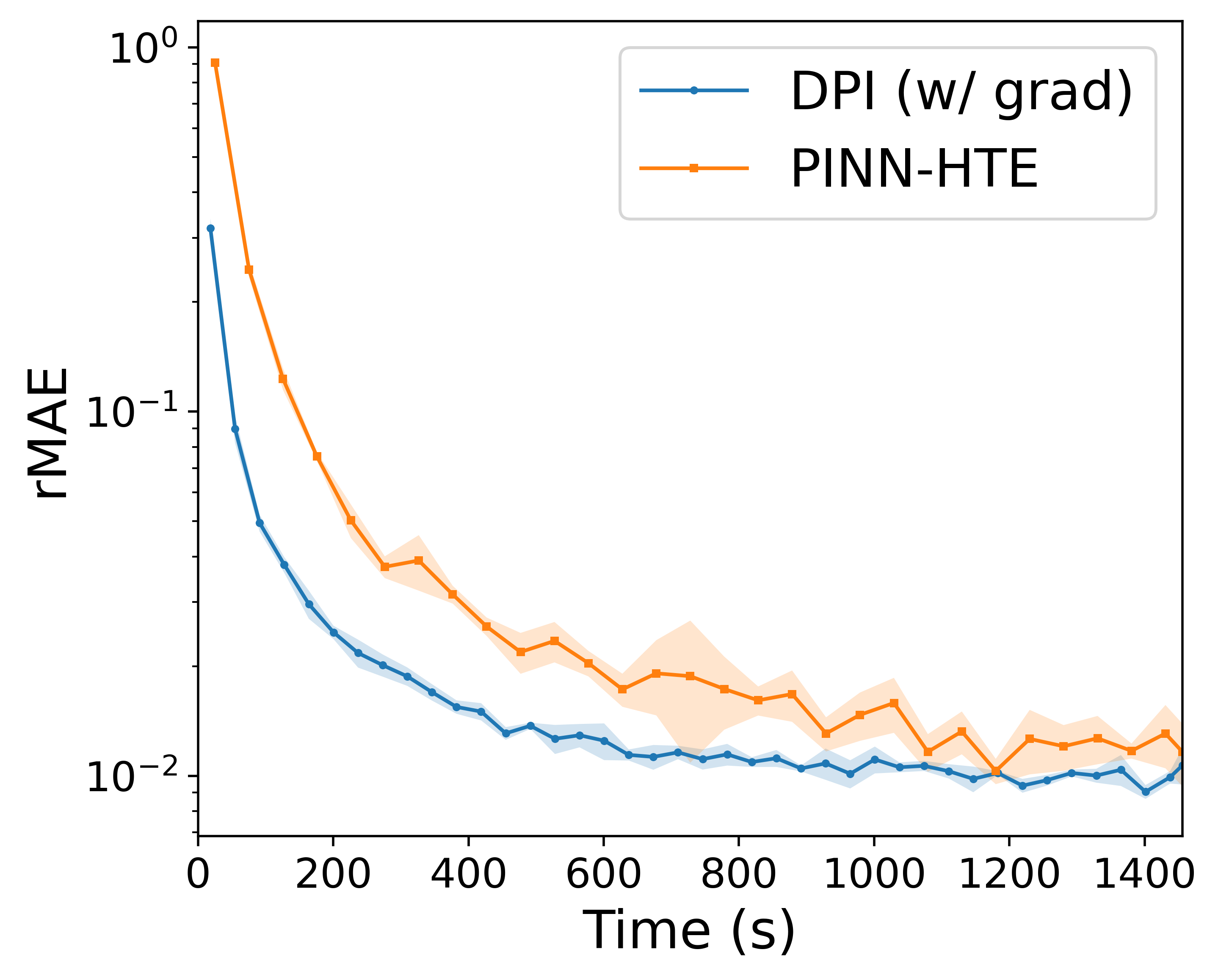}
    \end{subfigure}
    \hfill
    \begin{subfigure}[b]{0.49\textwidth}
        \centering
        \includegraphics[width=\textwidth]{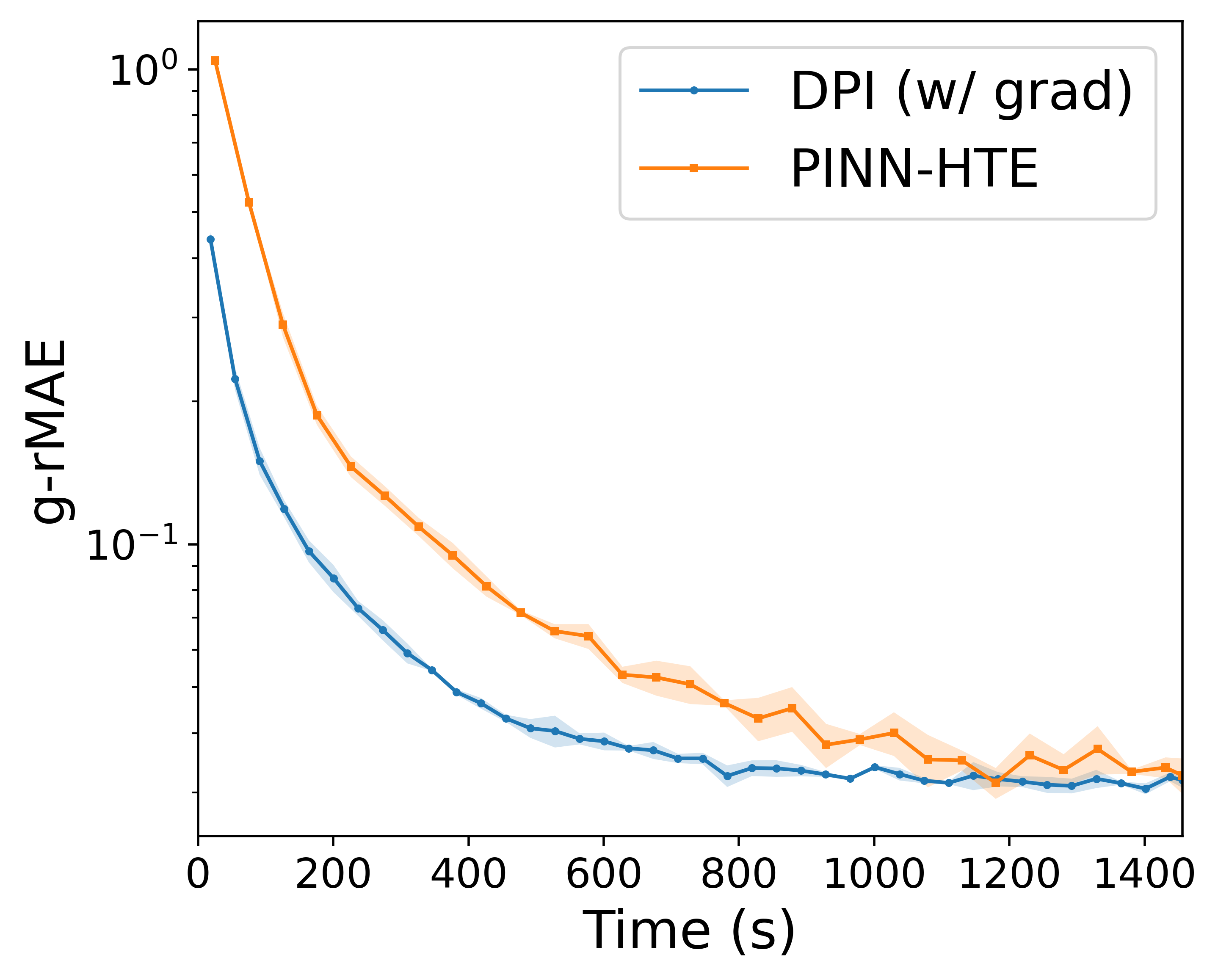}
    \end{subfigure}
\caption{The relative error curves for \( u \) and \( \nabla u \) during training of DPI and PINN-HTE for fully nonlinear problem in Section \ref{sec:fully}. From top to bottom, the results correspond to Case I, II, and III.}
    \label{fig:fully_nonlinear}
\end{figure}

We plot the curves of rMAE and g-rMAE at each checkpoint for different methods during the training process. As shown in Figure~\ref{fig:burgers_more} for the Burgers-type problem in Section \ref{sec:cha}, the results clearly indicate that DPI consistently reduces errors with extended training, whereas D-DBSDE and PINN-HTE frequently become trapped in local minima. Similar behaviors are observed in Figures~\ref{fig:hjb} and \ref{fig:fully_nonlinear} for the HJB and fully nonlinear problems, respectively, further highlighting the robustness and superiority of the DPI method. We emphasize that for the HJB problem with $T=0.25, 0.5$, results are similar as the problems with these shorter time horizons are relatively easy to optimize for all methods. However, at $T=1.0$, PINN-HTE and D-DBSDE exhibit instability due to increased difficulty of the problem, while DPI maintains consistently decreasing errors throughout the training and demonstrates stable and superior performance. We also note that DBDP proceeds backward in time, solving the problem one step at a time, which makes it incompatible with computing solution errors over the entire time interval as done for the other methods. Therefore, only the results for DPI and PINN-HTE are reported in Figure~\ref{fig:fully_nonlinear}.

The error curves during training further confirm that DPI is notably more robust to varying hyperparameter choices than PINNs or DBSDEs. To illustrate this, we use the Burgers-type PDE with \(\kappa = 5.0\) as an example. Specifically, the test error curves shown in Figure \ref{fig:unstable_pinn} for PINN-HTE and Figure \ref{fig:unstable_dbsde} for D-DBSDE demonstrate how the terminal weight \(\lambda_T\) significantly influences performance, sometimes leading to non-monotonically decreasing error curves with notable variance. In contrast, as shown in Figure \ref{fig:unstable_dpi}, DPI consistently exhibits monotonically converging error curves across a wider range of values for \(\lambda\), highlighting its robustness during training.

\begin{figure}[!htb]
    \centering
    \begin{subfigure}[b]{0.49\textwidth}
        \centering
\includegraphics[width=\textwidth]{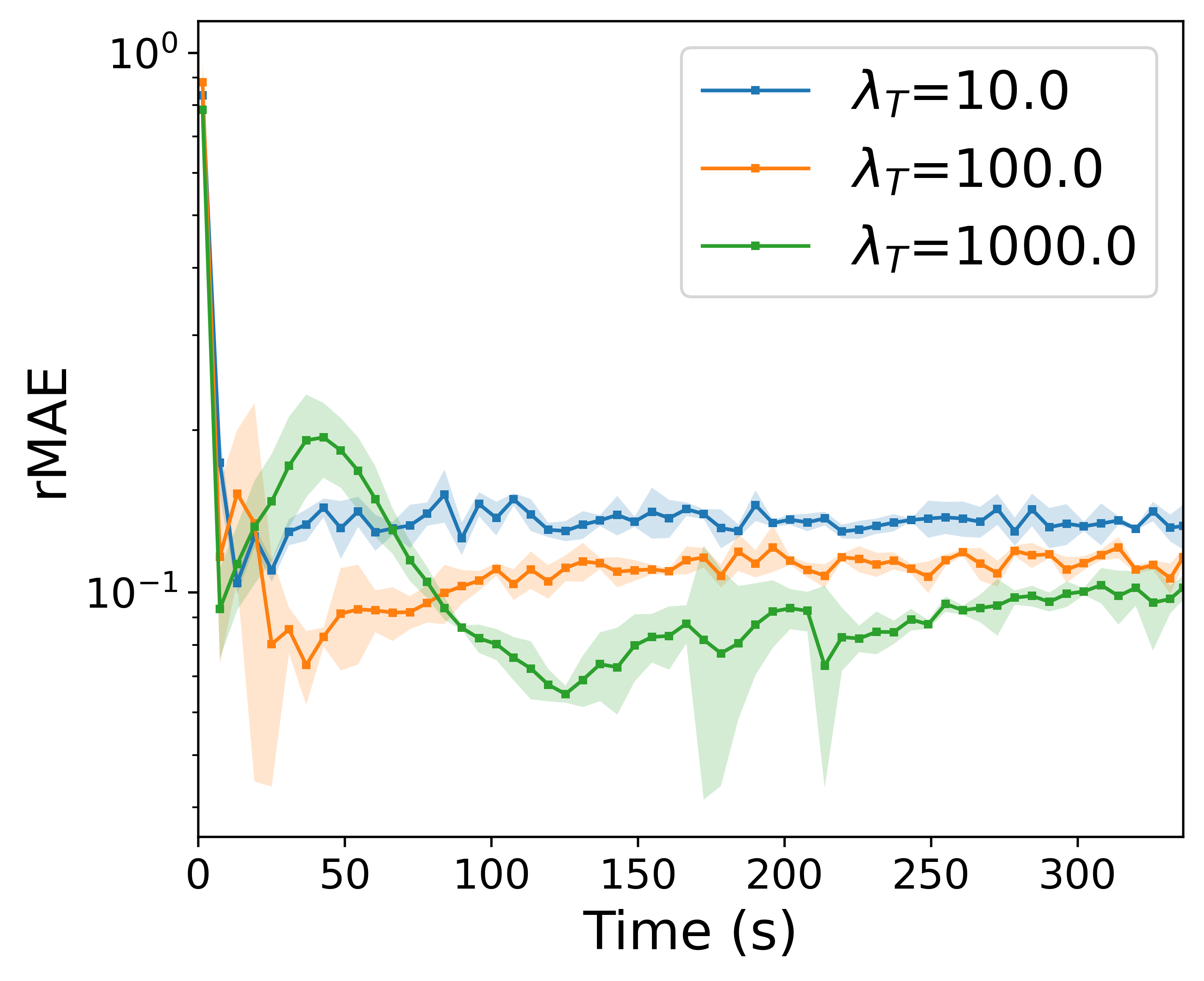}
    \end{subfigure}
    \hfill
    \begin{subfigure}[b]{0.49\textwidth}
        \centering
        \includegraphics[width=\textwidth]{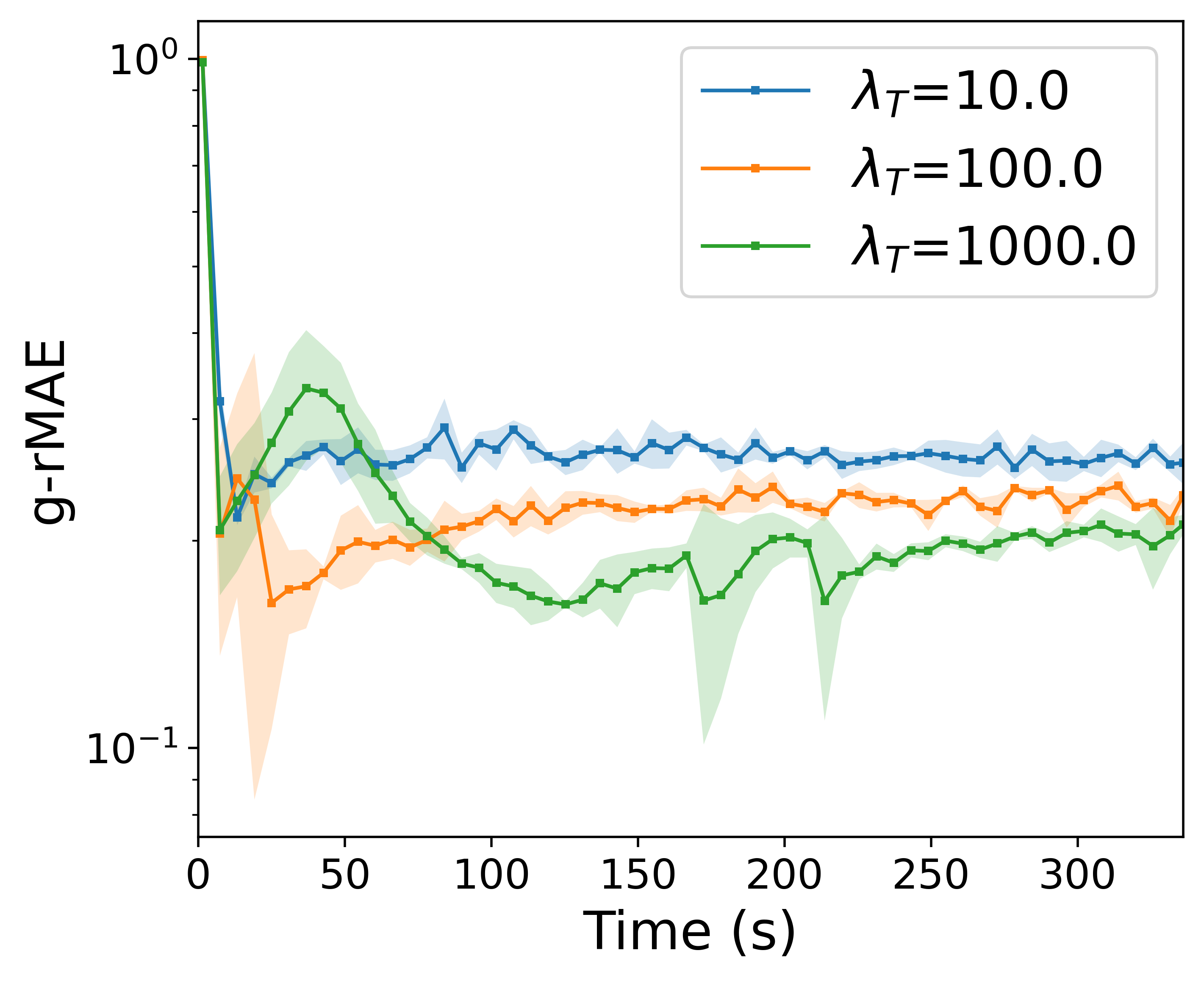}
    \end{subfigure}
\caption{The relative error curves for \( u \) and \( \nabla u \) during training of PINN-HTE with varying terminal weight $\lambda_T$ for Burgers-type PDE with $\kappa=5.0$ in Section \ref{sec:cha}.}
    \label{fig:unstable_pinn}
\end{figure}

\begin{figure}[!htb]
    \centering
    \begin{subfigure}[b]{0.49\textwidth}
        \centering
\includegraphics[width=\textwidth]{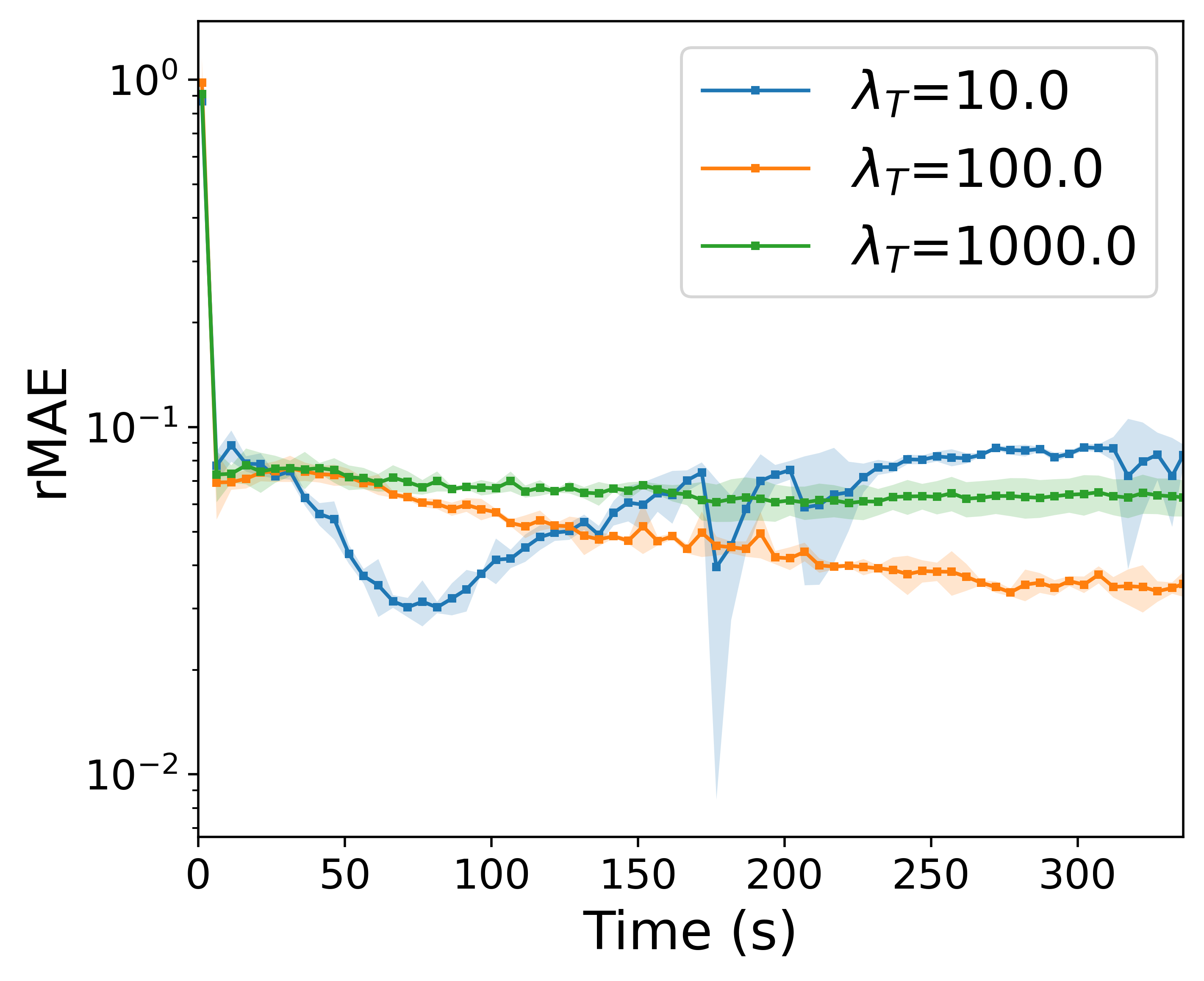}
    \end{subfigure}
    \hfill
    \begin{subfigure}[b]{0.49\textwidth}
        \centering
        \includegraphics[width=\textwidth]{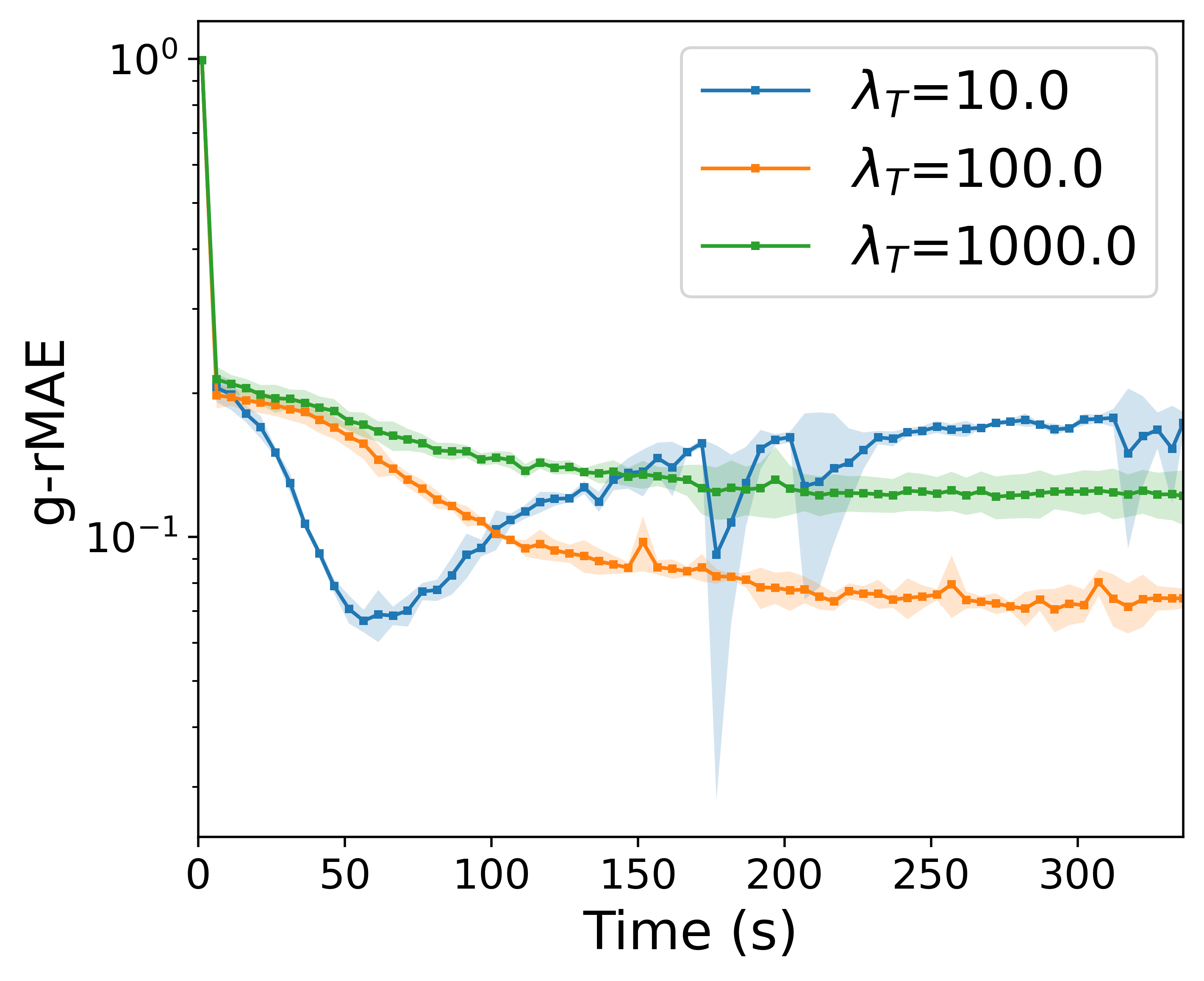}
    \end{subfigure}
\caption{The relative error curves for \( u \) and \( \nabla u \) during training of D-DBSDE with varying terminal weight $\lambda_T$ for Burgers-type PDE with $\kappa=5.0$ in Section \ref{sec:cha}.}
    \label{fig:unstable_dbsde}
\end{figure}

\begin{figure}[!htb]
    \centering
    \begin{subfigure}[b]{0.49\textwidth}
        \centering
\includegraphics[width=\textwidth]{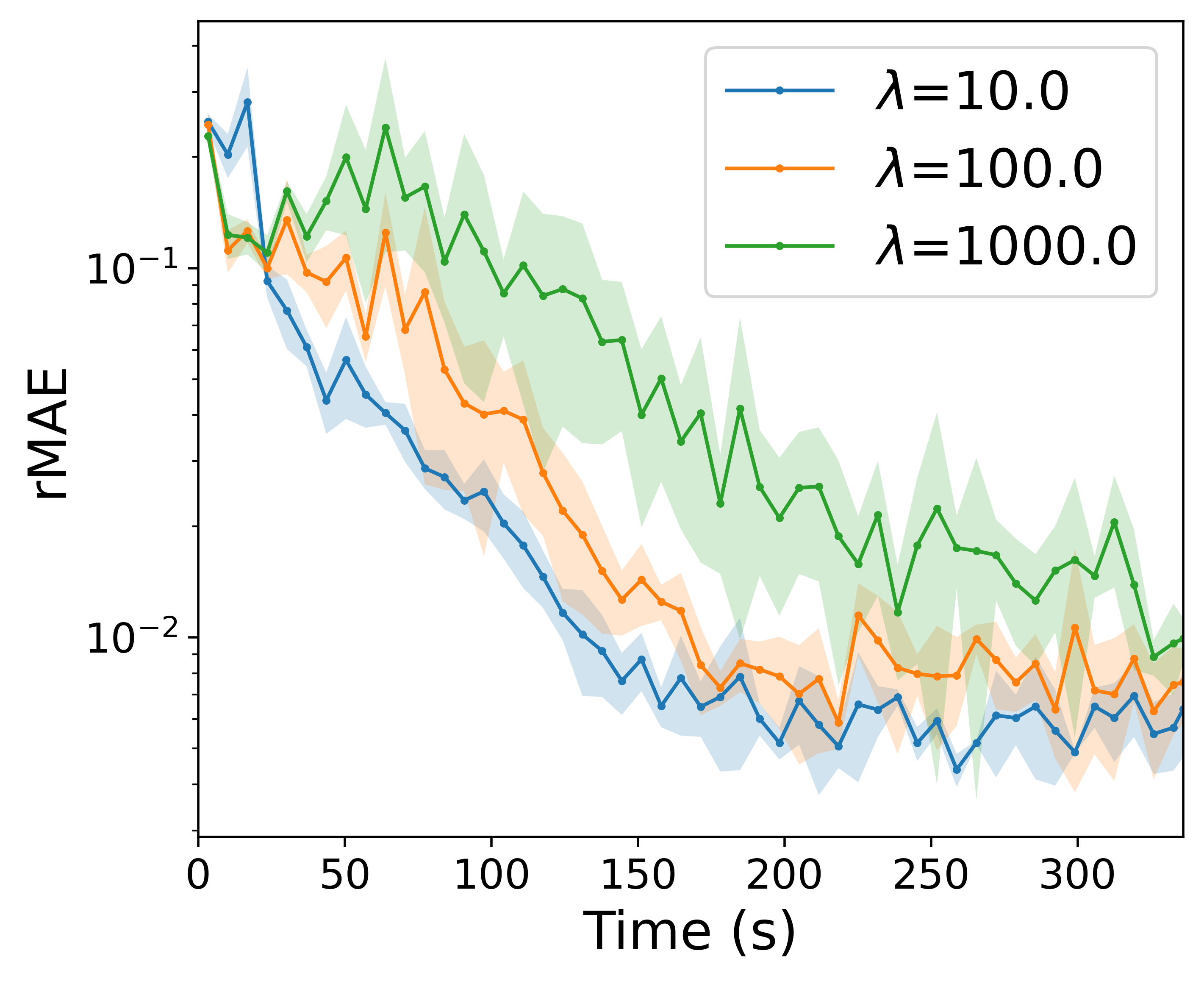}
    \end{subfigure}
    \hfill
    \begin{subfigure}[b]{0.49\textwidth}
        \centering
        \includegraphics[width=\textwidth]{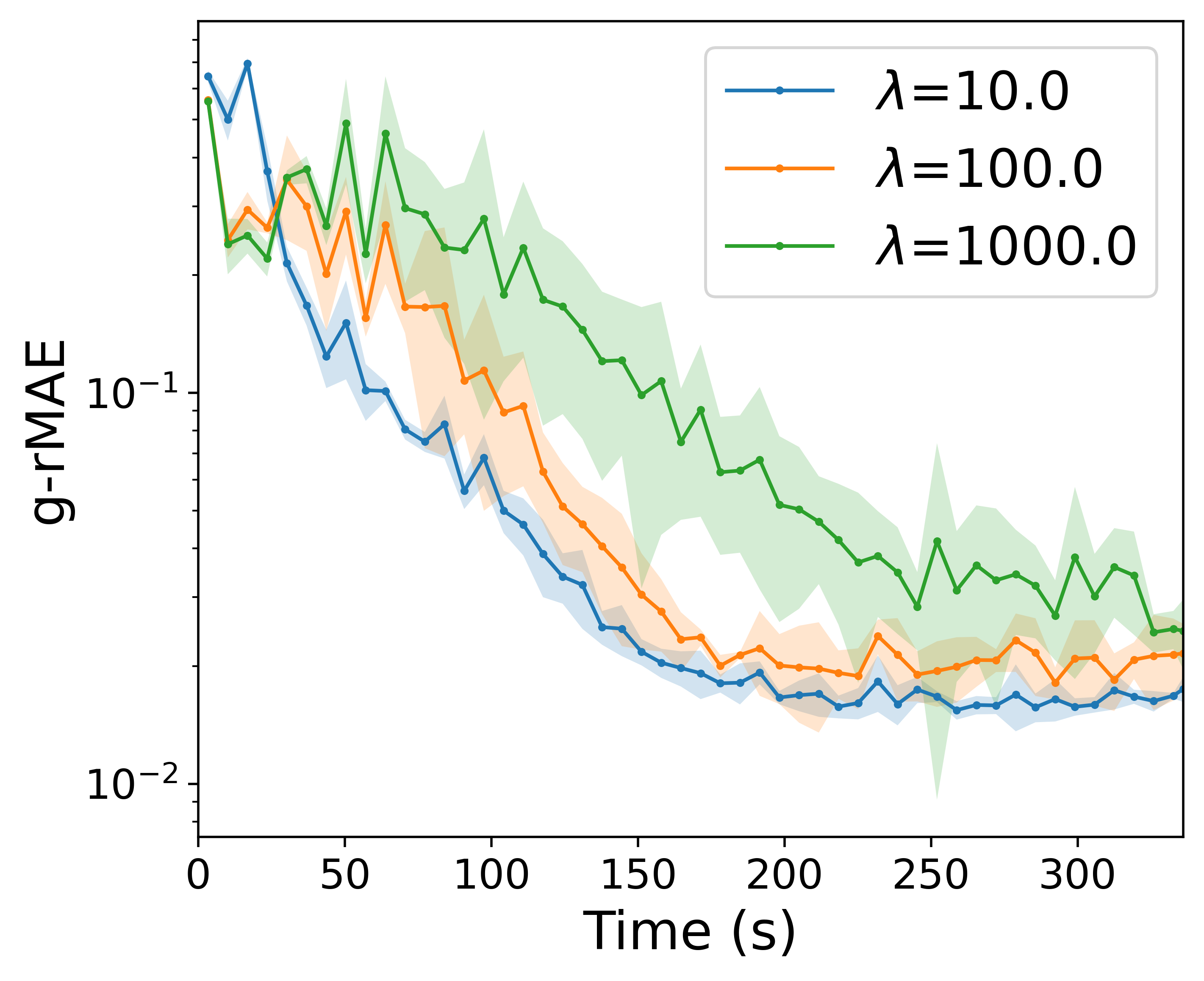}
    \end{subfigure}
\caption{The relative error curves for \( u \) and \( \nabla u \) during training of DPI with varying weight $\lambda$ for Burgers-type PDE with $\kappa=5.0$ in Section \ref{sec:cha}.}
    \label{fig:unstable_dpi}
\end{figure}

\end{document}